\documentclass[11pt]{article}
\usepackage{bbm}
\usepackage{psfrag,epic,eepic}
\usepackage{fullpage,color}
\usepackage[applemac]{inputenc} 
\usepackage{amsmath,amsfonts,amssymb,amsthm,mathrsfs}
\usepackage[a4paper,vmargin={3.5cm,3.5cm},hmargin={2.5cm,2.5cm}]{geometry}
\usepackage[font=sf, labelfont={sf,bf}, margin=1cm]{caption}
\usepackage{graphicx,graphics}
\usepackage{latexsym}
\usepackage{ae,aecompl}
\usepackage[english]{babel}
 \usepackage[colorlinks=true]{hyperref}
\usepackage{enumerate}
\usepackage[normalem]{ulem}
\usepackage{xcolor}
\usepackage[framemethod=tikz]{mdframed}
\usepackage{tcolorbox}
\usepackage{pstricks}

\definecolor{gris25}{gray}{0.75}

\definecolor{mycolor}{rgb}{0, 0, 0.1}

\newmdenv[innerlinewidth=0.5pt, roundcorner=4pt,linecolor=mycolor,innerleftmargin=6pt,
innerrightmargin=6pt,innertopmargin=6pt,innerbottommargin=6pt]{mybox}

\newcommand{\E}{\mathbb{E}}

\newcommand{\F}{\ensuremath{\mathcal{F}}}
\newcommand{\G}{\ensuremath{\mathcal{G}}}

\newcommand{\e}{{\mathrm e}}
\newcommand{\R}{{\mathbb{R}}}

\newcommand{\N}{\mathbb{N}}
\newcommand{\D}{\mathcal{D}}
\newcommand{\Z}{\mathbb{Z}}

\newcommand{\veps}{\varepsilon}

\newcommand{\pr}{\mathbb{P}}
\newcommand\numberthis{\addtocounter{equation}{1}\tag{\theequation}}

\renewcommand{\theequation}{\arabic{equation}}

\newtheorem{thm}{Theorem}[section]
\newtheorem{defn}[thm]{Definition}
\newtheorem{lem}[thm]{Lemma}
\newtheorem{prop}[thm]{Proposition}

\date{}

\title{\bf{\textsc{Bivariate Markov chains converging to \\ Lamperti transform Markov Additive Processes}}}

\author{B\'en\'edicte Haas\thanks{ Universit\'e Paris 13,  Sorbonne Paris Cit\'e, LAGA, CNRS UMR 7539, 93430 Villetaneuse, haas@math.univ-paris13.fr}  \quad \& \hspace{0.2cm} Robin Stephenson\thanks{NYU--ECNU Institute of Mathematical Sciences at NYU Shanghai, robin.stephenson@normalesup.org \newline \text{} \quad  This work is partially supported by the ANR GRAAL ANR--14--CE25--0014.}}

\begin{document}

\maketitle

\begin{abstract}
Motivated by various applications, we describe the scaling limits of bivariate Markov chains $(X,J)$ on $\mathbb Z_+ \times \{1,\ldots,\kappa\}$ where $X$ can be viewed as a position marginal and $\{1,\ldots,\kappa\}$ is a set of $\kappa$ types. The chain starts from an initial value $(n,i)\in \mathbb Z_+ \times \{1,\ldots,\kappa\}$, with $i$ fixed and $n \rightarrow \infty$, and typically we will assume that the macroscopic jumps of the marginal $X$ are rare, i.e. arrive with a probability proportional to a negative power of the current state. We also assume that $X$ is non--increasing. We then observe different asymptotic regimes according to whether the rate of type change is proportional to, faster than, or slower than the macroscopic jump rate. In these different situations, we obtain in the scaling limit Lamperti transforms of Markov additive processes, that sometimes reduce to standard positive self--similar Markov processes. 
As first examples of applications, we study the number of collisions in coalescents in varying environment and the scaling limits of Markov random walks with a barrier. This completes previous results obtained by Haas and Miermont \cite{HM11} and Bertoin and Kortchemski \cite{BK14} in the monotype setting.
In a companion paper, we will use these results as a building block to study the scaling limits of multi--type Markov branching trees, with applications to growing models of random trees and multi--type Galton--Watson trees.
\end{abstract}

%*******************************************************

\section*{Introduction}

%*******************************************************

\fontsize{10.75}{12.9}\selectfont 

Let $\left((X(k),J(k)),k\geq0\right)$ be a Markov chain on $\mathbb Z_+ \times \{1,\ldots,\kappa\}$ for some integer $\kappa\geq 2$, with transition probabilities $p_{n,i}(m,j)$ such that the first component $(X(k),k \geq 0)$ is non--increasing, i.e.
$$
p_{n,i}(m,j)=0,  \quad \text{when } m>n.
$$
We view the marginal $X$ as the position component of the chain, and $J$ as its type. When the process starts from $(n,i)\in \mathbb Z_+ \times \{1,\ldots,\kappa\}$, we will refer to it as $(X_n^{(i)},J_n^{(i)})$. Since the process $X_n^{(i)}$ is non--increasing and $\mathbb Z_+$--valued, it is absorbed after a finite (random) time, denoted by $A_n^{(i)}$.
Our goal is to give conditions on the transition probabilities under which a suitable rescaling of the process 
\begin{equation}
\label{Pr}
\left(\left(X^{(i)}_n(\lfloor t \rfloor),t\geq 0\right), A^{(i)}_n\right) 
\end{equation}
has a non--trivial limit as $n \rightarrow \infty$ (for all $i$). 

\bigskip

This question has already been studied in the monotype setting ($\kappa=1$), see \cite{HM11} for the non--increasing case and \cite{BK14} for more general cases. Several applications to random walks with a barrier, Bessel--type random walks, exchangeable coalescence or fragmentation--coalescence processes, random trees and random planar maps have since then been developed (\cite{HM11, BK14, HM12, BCK15}).  Still in the monotype setting, there is also a series of papers describing the behavior of the absorption time only, under various assumptions on the transition probabilities, see e.g. \cite{IM08} and \cite{AM15} and the references therein. 

\vspace{0.01cm}

Our goal is to extend the results of \cite{HM11} to the multi--type setting. With this, we aim at developing new applications. One important application will concern the description of the scaling limits of multi--type Markov branching trees, with in turn applications to growing models of random trees and multi--type Galton--Watson trees. See the forthcoming work \cite{HS15}. Roughly, a family of random trees is said to satisfy the Markov branching property if, for each tree of the family, the subtrees above a given height are independent, with distributions that depend only on their sizes. This property arises in several natural situations, see e.g. \cite{Ald96, BerFire, BDMcS08, HMPW08, HM12, HS14} and the survey \cite{HSurvey16}. 
It turns out that multi--type versions of such families also arise naturally, with strong connections with multi--type fragmentation processes as developed in \cite{Ber08}. This is the topic of the forthcoming work \cite{HS15}. At the end of the present paper, we develop other applications of the asymptotic study of (\ref{Pr}), to models of coalescents in varying environment and to multi--type random walks with a barrier.

\bigskip

Let us  now briefly recall the results of \cite{HM11}. In that case we remove any notation referring to the type and denote $(p_n(k))$ the  transitions probabilities of the chain $(X_n)$. When $(X_n)$ is non--increasing, it has been shown that if for all continuous functions $f:[0,1]\rightarrow \mathbb R$
\begin{equation}
\label{HypMonotype}
n^{\gamma} \times \sum_{k\geq 0}f\left(\frac{k}{n}\right)\left(1-\frac{k}{n}\right)p_n(k) \underset{n\rightarrow \infty}\longrightarrow \int_{[0,1]}f(x)\mu(\mathrm dx)
\end{equation} 
for some $\gamma>0$ and some non--zero, finite, non--negative measure $\mu$ on $[0,1]$,  
then
\[
\left(\frac{X_n(\lfloor n^{\gamma}t \rfloor)}{n},t\geq 0\right) \overset{\mathrm{(d)}}{\underset{n\rightarrow \infty}\longrightarrow} X_{\infty}
\]
for the Skorokhod topology on the space of c\`adl\`ag functions from $[0,\infty)$ to $[0,\infty)$, 
where $X_{\infty}$ is a positive $\gamma$--self--similar Markov process which is absorbed at 0 in finite time. Note that (\ref{HypMonotype}) means that starting from $n$, the probability to do a jump larger than $\varepsilon n$ is of order $c_{\varepsilon}n^{-\gamma}$, where $c_{\varepsilon}$ increases as $\varepsilon$ decreases, and possibly tends to $+\infty$ as $\varepsilon$ tends to 0. Background on positive self--similar Markov process will be recalled in Section \ref{LampertiMAP}, in particular their connection to L\'evy processes via the Lamperti transformation. Roughly, the point is that any positive self--similar Markov process can be written as the exponential of a time--changed L\'evy process. For our process $X_{\infty}$ the time--change is guided by the parameter $\gamma$ and the L\'evy process is the negative of a subordinator with Laplace transform
\[
\psi(\lambda)=\mu(\{0\})+\mu(\{1\})\lambda+\int_{(0,1)}\big(1-x^{\lambda}\big)\frac{\mu(\mathrm dx)}{1-x}.
\]
It is also known from \cite{HM11}, that jointly with the previous convergence, the absorption time 
\[A_n=\inf \big\{k:X_n(j)=X_n(k), \forall j \geq k \big\}
\]
satisfies
\[
\frac{A_n}{n^{\gamma}} \ \overset{\mathrm{(d)}}{\underset{n\rightarrow \infty}\longrightarrow} \ \inf\{t \geq 0:X_{\infty}(t)=0\},
\] 
and that this limit is distributed as
$$
\int_0^{\infty}\exp(-\gamma \xi_r) \mathrm d r,
$$ 
where $\xi$ is a subordinator with the above Laplace transform $\psi$. Bertoin and Kortchemski have in \cite{BK14} extended these results to non--monotone chains. They obtain similarly  positive (non--monotone) self--similar Markov processes in the scaling limit.  

\bigskip

Coming back to the multi--type setting,  we focus here on \textbf{non--increasing} chains $X$. This is to simplify our approach,   but we believe that
similar results may hold in a non--monotone framework. 
In order to describe the scaling limits of (\ref{Pr}), we will need Lamperti transform Markov additive processes (MAP), as a generalization of Lamperti transform L\'evy processes. 
In general, a MAP is a Markov process $\left((\xi_t,K_t),t \geq 0 \right)$ taking values in $\mathbb R \times \{1,\ldots,\kappa\}$ for some integer $\kappa \geq 1$ such that for all $t \in \mathbb R_+$ and all $(x,i) \in \mathbb R \times \{1,\ldots,\kappa\}$ 
\[
\big(\left(\xi_{t+s}-\xi_t,K_{t+s}),s \geq 0\right) \ | \ (\xi_u,K_u),u\leq t \big) \text{ under } \mathbb P_{(x,i)} \text{ has distribution } \mathbb P_{(0,K_t)}.
\]
Later we will more generally consider MAPs that may possibly be killed (with $\xi$ reaching then $+\infty$). 
See Section \ref{SecMAP} for background, references and the notion of Lamperti transform. Of course, when $\kappa=1$, the first marginal $\xi$ reduces to a standard L\'evy process and its Lamperti transform to a positive self--similar Markov process. 

In the multi--type setting, we will observe three different regimes in the limit. Let us  explain this very roughly here and postpone precise statements to the core of the paper. As in the monotype setting, we will always assume that macroscopic jumps of the $X$--marginal are rare, with a rate of order $n^{-\gamma}, \gamma>0$ when $(X,J)$ is in the state  $(n,i)$, for all types $i$. We will further assume that the rate of type change is of order $n^{-\beta}, \beta\geq0$ when $(X,J)$ is in the state  $(n,i)$, for all types $i$. We will then have to accelerate time by a factor $n^{\gamma}$ in the process (\ref{Pr}) to observe a non--trivial limit. The nature of the limit will depend of the position of $\beta$ compared to $\gamma$:
\begin{enumerate}
\item[$\bullet$] If $\beta=\gamma$ (critical regime), the limiting process is a Lamperti transform MAP involving at most $\kappa$ types.
\item[$\bullet$] If $\beta<\gamma$ (mixing regime), the limiting process is a positive self--similar Markov process, whose distribution is a mixture of the contributions of the different types, depending on the stationary distribution of the types.
\item[$\bullet$] If $\beta>\gamma$ (solo regime), the limiting process  is a positive self--similar Markov process, whose distribution depends only on the initial type.
\end{enumerate}
The third case, when the changes of types occur at rates slower than the macroscopic jumps is the most simple one: in that case the chain will not change type at all in the scaling limit, and we are left with the standard monotype case. In the critical regime, the rescaled chain will locally behave as if it was monotype, until the type changes (after a strictly positive time). Our study in that situation will consist in studying monotype processes on the type--constancy intervals of the chain and then ``gluing" all these processes together. In the mixing regime, the types will change quickly and will give rise in the scaling limit to a stationary distribution, which is used to combine contributions from each type. The formal assumptions corresponding to each of those three cases are given in 
Hypotheses ($\mathsf H_{\mathrm{cr}}$) (Section \ref{SecCritique}, for the critical regime), ($\mathsf H_{\mathrm{mix}}$) (Section \ref{SecMixing}, for the mixing regime) and ($\mathsf H_{\mathrm{sol}}$) (Section \ref{SecSolo}, for the solo regime).  In each of these three situations, we will also describe the scaling limit of the absorption time of the marginal $X_n$.

\bigskip

{\setlength{\baselineskip}{0.8\baselineskip}
\tableofcontents \par}

%*********************************************************************************

\section{Markov additive processes and their Lamperti transforms}
\label{SecMAP}

%*********************************************************************************

This section concerns the continuous--time processes that will arise in the scaling limits of the bivariate Markov chains under consideration. 

\subsection{Generalities on Markov additive processes}

We give here some background on \emph{Markov additive processes} and refer to Asmussen \cite[Chapter XI]{asmussen} for details and applications. 

\smallskip

\begin{defn} Let  $\left((\xi_t,K_t),t \geq 0 \right)$ be a  Markov process on $\mathbb R \times \{1,\ldots,\kappa\}\cup\{(+\infty,0)\}$, where $\kappa \in \mathbb N$, and write $\mathbb P_{(x,i)}$ for its distribution when starting at a point $(x,i)$. It is called a \emph{Markov additive process} \emph{(MAP)} if for all $t \in \mathbb R_+$ and all $(x,i) \in \mathbb R \times \{1,\ldots,\kappa\}$,
\[
\left(\left(\xi_{t+s}-\xi_t,K_{t+s}),s \geq 0\right) \ | \ (\xi_u,K_u),u\leq t , \xi_t<\infty \right) \text{ under } \mathbb P_{(x,i)} \text{ has distribution } \mathbb P_{(0,K_t)},
\]
and $(+\infty,0)$ is an absorbing state.
\end{defn}

\smallskip

Note that MAPs are closely related to Lévy processes. When $\kappa=1$, $\xi$ is clearly a standard Lévy process. In the general case, the chain $(K_t,t\geq 0)$ is a continuous--time Markov chain, and on its constancy intervals, the process $\xi$ behaves as a Lévy process, whose dynamics depend only on the value of the chain $K$. Jumps of $K$ may also induce jumps of $\xi$. As in the discrete setting, we will sometimes refer to $\xi$ as the position marginal, and $K$ as the type marginal. 
 In this paper, unless otherwise stated, we always consider MAPs such that $\xi$ is \textbf{non--decreasing}. The distribution of such a process is then characterized by three families of parameters:
\begin{enumerate}
\item[$\bullet$] the jump rates $(\lambda_{i,j})_{(i,j) \in \{1,\ldots,\kappa\}^2,i\neq j}$ of the chain $(K_t,t\geq 0)$
\item[$\bullet$] a set of distributions on $[0,+\infty)$ $(B_{i,j})_{(i,j) \in \{1,\ldots,\kappa\}^2,i\neq j}$: $B_{i,j}$ is the distribution of the jump of $\xi$ when $K$ jumps from $i$ to $j$
\item[$\bullet$] triplets $(k^{(i)}, c^{(i)},\Pi^{(i)})$, where $k^{(i)},c^{(i)}\geq 0$ and $\Pi^{(i)}$ are $\sigma$--finite measures on $(0,\infty)$ such that $\int_{(0,\infty)} (1 \wedge x) \Pi^{(i)}(\mathrm dx)<\infty$, $1 \leq i \leq \kappa$.
The triplet $(k^{(i)}, c^{(i)},\Pi^{(i)})$ corresponds to the standard parameters (killing rate, drift and L\'evy measure) of the subordinator which $\xi$ follows on the time intervals where $K=i$. We call $(\psi_i)_{i\in\{1,\ldots,\kappa\}}$ the corresponding Laplace exponent, that is, for $i\in\{1,\ldots,\kappa\}$, $q\geq 0$
\[\psi_i(q)=k^{(i)}+c^{(i)}q+\int_0^{\infty} (1-\e^{-qx})\Pi^{(i)}(\mathrm{d}x).\]
\end{enumerate}
If $\xi$ is killed at a time $t$, then by convention $\xi_s=+\infty$ and $K_s=0$ for $s\geq t.$ 

\bigskip

\noindent \textbf{Asymptotics.} In most circumstances, we will exclude the cases where the process $(\xi_t,t\geq 0)$ may be absorbed in a constant state after a certain time. Typically this cannot happen, with probability one, as soon as for each type $i \in \{1,\ldots,\kappa\}$:
\begin{equation}
\label{HypNotConstant}
\begin{aligned}
 \text{a)}& \text{ either one of the parameters } k^{(i)}, c^{(i)},\Pi^{(i)} \text{ is not trivial  or }\exists j\neq i \text{ such that } \lambda_{i,j}>0 \text{ and } B_{i,j} \neq \delta_0\\
 & \text{(the type $i$ may induce a jump)} \\
 \text{b)}& \text{ or } \exists j \text{ satisfying a) and a path }i_1=i,i_2,\ldots,i_p=j \text{ such that } \lambda_{i_k,i_{k+1}}>0 \text{ for all } 1\leq k \leq p-1 \\ 
& \text{(in all irreducible components of types, there is at least one type satisfying a))}.
\end{aligned}
\end{equation}
We then note the following simple law of large numbers--type lemma which we will need in what follows. We point out that the limit is not deterministic in general, specifically it depends on which irreducible component the Markov chain of types lands into.

\begin{lem} 
\label{lem:linearite}
Assume $(\ref{HypNotConstant})$ for each type $i \in \{1,\ldots,\kappa\}$. Then, as $t\to\infty$, the random variable $t^{-1}\xi_t$ has a $\mathbb P_{(x,i_0)}$--almost--sure limit, which is strictly positive (and possibly infinite), for all $(x,i_0) \in \mathbb R \times \{1,\ldots,\kappa\}$.
\end{lem}

\begin{proof} If the process is killed, then of course the wanted limit is $+\infty$. Thus in following we condition $K$ on arriving into the irreducible component of a type $i$, such that $k^{(j)}=0$ for all $j$ in this component. By (\ref{HypNotConstant}) we can assume that $i$ satisfies a). Let then $(T_n,n\in \N)$ be the successive return times to $i$. By the law of large numbers, both $T_n/n$ and $\xi_{T_n}/n$ have strictly positive limits a.s. as $n$ tends to infinity (possibly an infinite limit for $\xi_{T_n}/n$). For $t\geq 0$, we then let $n(t)$ be the unique integer such that $T_{n(t)}\leq t <T_{n(t)+1},$ and if we write
\[\frac{\xi_{T_{n(t)}}}{T_{n(t)+1}}\leq\frac{\xi_t}{t}\leq \frac{\xi_{T_{n(t)+1}}}{T_{n(t)}},\]
both bounds converge to the same limit, ending the proof.
\end{proof}

\subsection{Lamperti transform MAPs}
\label{LampertiMAP}

The Lamperti transformation is a time--change used by Lamperti  \cite{Lamp62, Lamp72}  to give a one--to--one correspondence between L\'evy processes and non--negative self--similar Markov processes with a fixed index of self--similarity. We give here a generalization to multi--type self--similar processes. Let $\left((\xi_t,K_t),t \geq 0 \right)$ be a MAP (exceptionally here $\xi$ is not supposed monotone) and $\gamma>0$ be a number we call the \emph{index of self--similarity}. Let also $Z_t=\e^{-\xi_t}$ for $t\geq 0.$ We let $\rho$ be the time--change associated to $Z$ and $\gamma$ by:
	\[\rho(t) = \inf \left\{u, \int_0^u (Z_r)^{\gamma} \mathrm{d}r >t\right\},
\] 
and call \emph{Lamperti transform of $\left((\xi_t,K_t),t \geq 0 \right)$} the process $\left((X_t,J_t),t\geq 0)\right)$ defined by
\begin{equation}
\label{LMAP}
X_t=Z_{\rho(t)}, \quad J_t=K_{\rho(t)}.
\end{equation}
Here, by convention, $\rho(t)=\infty$ if $t \geq \int_0^{\infty} (Z_r)^{\gamma} \mathrm{d}r$ and we let $X_{t}=0$ and  $J_t=0$ for such times $t$. Note that, while $J$ is c\`adl\`ag on $[0,\int_0^{\infty} (Z_r)^{\gamma} \mathrm{d}r)$, it does not have a left limit at $\int_0^{\infty} (Z_r)^{\gamma} \mathrm{d}r$ (whether this integral is finite, or not) in general.

When $\kappa=1$, $\xi$ is a standard Lévy process, and the marginal $X$ is a non--negative self--similar Markov process. Reciprocally, any such Markov process can be written in this form, see  \cite{Lamp72}.
In general, for any $\kappa$, one readily checks that the process $\left((X_t,J_t),t\geq 0)\right)$ is Markovian and $\gamma$--self--similar, in the sense that 
$\left((X_t,J_t),t\geq 0\right)$, started from $(x,i),$ has the same distribution as $ \left((xX'_{x^{-\gamma}t},J'_{x^{-\gamma}t}),t\geq 0\right)$, where  $\left((X'_t,J'_t),t\geq 0\right)$ is a version of the same process which starts at  $(1,i).$ 
We point out that recently, \cite{CPR13,KKPW14} gave a one--to--one correspondence via Lamperti transformation between a family of MAPs with two types and \emph{real--valued} self--similar Markov processes with initial condition different from zero, generalizing the initial result of Lamperti \cite{Lamp72}. 

\subsubsection{Some properties of the time--change}

We give here a few properties of the Lamperti--type time--change introduced above which we will need at various places in the paper. We place ourselves in a more general framework and let $f$ be a non--increasing and c\`adl\`ag function from $[0,\infty)$ to $[0,\infty)$ such that $f(0)=1$. We introduce the notation
\[T_0(f)=\inf\big\{t\geq 0, f(t)=0\}.\]
Let also $\alpha\in\R.$ We then call the \emph{Lamperti time--change} (associated to $f$ and $\alpha$) the function $\tau$ defined for $t\geq0$ by
	\[\tau(t) = \inf \left\{u, \int_0^u f(r)^{\alpha} \mathrm{d}r >t\right\},
\] 
and then call the function $g$ defined by $g(t)=f(\tau(t))$ the Lamperti transform of $f$, where by convention $\inf\{\emptyset\}=\infty$ and $f(\infty)=0$.
Note that $$T_0(g)=\int_0^{T_0(f)} f(r)^{\alpha}\mathrm d r$$ and that $\tau$ induces a bijection between $[0,T_0(f))$ and $[0,T_0(g))$. For $t\geq T_0(g)$, $\tau(t)$ is constantly equal to $T_0(f)$, whereas for $t\leq T_0(g)$,
\[\int_0^{\tau(t)} f(r)^{\alpha} \mathrm{d}r = t,
\]
which implies that $\tau$ is left and right--differentiable everywhere on $[0,T_0(g))$, with  $\tau'(t^{\pm})=$ \linebreak $f(\tau(t)^{\pm})^{-\alpha}$. This means informally that $\tau$ corresponds to a local rescaling of time by a factor $f(\tau(t))^{-\alpha}$. This also explains why, if we let $\rho$ be the Lamperti time--change associated to $g$ and $-\alpha$, then $\rho$ is the inverse bijection of $\tau$ from $[0,T_0(g))$ to $[0,T_0(f))$.

\bigskip

We will need the following lemma which shows that the Lamperti transformation, when $\alpha<0$, behaves well with the $J_1$--Skorokhod topology:
\begin{lem}\label{lampsko}
Let $(f_n)_{n\in\N}$ be a sequence of non--increasing and c\`adl\`ag functions from $[0,\infty)$ to $[0,\infty)$ and assume that $f_n\to f$ in the Skorokhod sense. Let $\alpha<0$ and $\tau_n,\tau,g_n,g$ be the respective Lamperti time--changes and Lamperti transforms of $f_n,f$. Then
\vspace{-0.2cm}
\begin{enumerate}
\item[\emph{(i)}] $\tau_n$ converges uniformly to $\tau$ on compact sets.
\vspace{-0.1cm}
\item[\emph{(ii)}] $g_n$ converges to $g$ in the Skorokhod sense.
\end{enumerate}
\end{lem}
We leave the proof of this lemma to Appendix \ref{sec:appsko}.

\subsubsection{Absorption time}
\label{sec:absorption}

Consider $(\xi^{(i)},K^{(i)})$ a MAP starting from $(0,i)$ (with $\xi$ non--decreasing) and satisfying (\ref{HypNotConstant}) for all types. Let $(X^{(i)},J^{(i)})$ be its Lamperti transform defined by (\ref{LMAP}) and
$$I^{(i)}:=\int_0^\infty (Z^{(i)}_r)^{\gamma} \mathrm dr$$ 
denote the time at which $X^{(i)}$ is absorbed at 0. By Lemma \ref{lem:linearite}, $I^{(i)}<\infty$ a.s.

\bigskip

\noindent {\bf Continuity of $X^{(i)}$ at time $I^{(i)}$}. When $X^{(i)}$ is a standard self--similar Markov process ($\kappa=1$), it is well--known and easy to check that it is absorbed continuously at 0 if and only if $\xi^{(i)}$ is not killed. We will use this on several occasions. Note that this generalizes easily to the multi--type setting. In particular, when there is no killing in the MAP $(\xi^{(i)},K^{(i)})$, the process $X^{(i)}$ is  absorbed continuously at 0. This leads to the following fact, which we will use later on:
let $(T^{(i)}(p),p \geq 1)$ be the successive jump times of the type marginal $J^{(i)}$, with the convention that $T^{(i)}(p)=I^{(i)}$ if there is strictly less than $p$ type changes. Hence either there is some killing in the MAP or the type is asymptotically constant, in which cases $T^{(i)}(p)=I^{(i)}$ for $p$ large enough, or there is no killing and no type is absorbed and the type changes  infinitely often. In this last case, $\rho^{-1}(T^{(i)}(p))$ tends to infinity, hence $T^{(i)}(p)$ tends to $I^{(i)}$, and $X^{(i)}$ is absorbed continuously at $I^{(i)}$. So finally in all cases,
\begin{equation}
\label{eq:cvchgttype}
X^{(i)}\big(T^{(i)}(p)\big) \underset{p\rightarrow \infty}\longrightarrow 0.
\end{equation}

%**********************************************************************

\section{Details on the bivariate Markov chain $(X,J)$}
\label{SecDetails}

%**********************************************************************

We fix here some conventions and notations on the $\mathbb Z_+ \times \{1,\ldots,\kappa\}$--valued Markov chain $(X,J)$ introduced in the Introduction. We recall that the $(p_{(n,i)}(m,j))$ denote the transition probabilities of the chain and that $(X_n^{(i)},J_n^{(i)})$ refers to the chain starting from $(n,i)$.

\bigskip

\noindent \textbf{Absorption time.} For all types $i$ and all integers $n$, let $A_n^{(i)}$ be the first time when the chain $X_n^{(i)}$ is absorbed, i.e.
\begin{equation}
\label{DiscreteAbsorption}
A_n^{(i)}:=\inf\Big\{k \geq 0: X_n^{(i)}(k')=X_n^{(i)}(k) \text{ for all }k'\geq k \Big\}.
\end{equation}
Since the chain is $\mathbb Z_+$--valued and non--increasing, $A_n^{(i)}$ is finite. We decide in the following that once absorbed the chain cannot change its type. This implies that the chain can be absorbed in a state $a$ (for some initial configuration $(n,i)$) if and only if there exists a type $j \in \{1,\ldots,\kappa\}$ such that $$p_{a,j}(a,j)=1.$$ We call such a state $a$ an \emph{absorbing state}. Note that clearly 0 is absorbing, and our convention implies that $p_{0,j}(0,j)=1$ for all types $j$.

\bigskip

\noindent \textbf{Type transition matrix}. We let $P_n(i,j)$ be the probability to pass from type $i$ to type $j$ when the position $X$ is in $n$, i.e. 
\begin{equation}\label{ProbaMatriceType}
P_n(i,j)=\sum_{m=0}^n p_{n,i}(m,j), \quad \forall i,j \in \{1,\ldots,\kappa\}.
\end{equation}

\bigskip

\noindent \textbf{Position transition probabilities}.  On the other hand, we let $p^{(i)}_n(m)$ be the probability to pass from position $n$ to position $m$ when the type $J$ is in $i$, i.e.
\begin{equation}\label{ProbaChgtPosition}
p^{(i)}_n(m)=\sum_{j \in \{1,\ldots,\kappa\}} p_{n,i}(m,j), \quad \forall n,m \in \mathbb Z_+. 
\end{equation}

%**********************************************************************

\section{Critical regime}
\label{SecCritique}

%**********************************************************************

This section is devoted to the case where the macroscopic jump rate and the type change rate of the chain $(X,J)$ are of the same order. This, in general, will give in the scaling limit a Lamperti transform MAP with several types. To simplify, we restrict ourselves to cases where  the limiting MAP is not eventually constant (hence (ii) in the following hypothesis). Formally, we assume throughout the following 

\medskip

\begin{mybox}
\noindent \textbf{Hypothesis $(\mathsf H_{\mathrm{cr}})$.} (i) For all $i,j \in \{1,\ldots,\kappa\}$, there exists finite measures $\mu^{(i,j)}$ on $(0,1]$ such that for all continuous functions $f:[0,1]\rightarrow \mathbb R$,
$$
n^{\gamma} \sum_{m=0}^n f\left(\frac{m}{n}\right) \left(1-\frac{m}{n}\mathbbm 1_{\{j=i\}} \right)p_{n,i}(m,j) \underset{n\rightarrow \infty}\longrightarrow \int_{(0,1]} f(x)\mu^{(i,j)}(\mathrm dx). 
$$
(ii) Moreover, for all $i\in \{1,\ldots,\kappa\}$: 
\begin{enumerate}
\item[a)] either $\sum_{j \in \{1,\ldots,\kappa\}}\mu^{(i,j)}(0,1)>0$ 
\item[b)] or there exists a type $\ell$ such that $\sum_{j \in \{1,\ldots,\kappa\}}\mu^{(\ell,j)}(0,1)>0$ and a path from $i$ to $\ell$, $i_1=i,$ $i_2,\ldots,i_p={\ell}$ such that $\mu^{(i_k,i_{k+1})}(0,1]>0$ for all $1 \leq k \leq p-1$.
\end{enumerate}
\end{mybox}

\noindent Note that under $(\mathsf H_{\mathrm{cr}})$, the probability $P_n(i,j)$ (defined in (\ref{ProbaMatriceType})) gives asymptotically
$$
n^{\gamma} P_n(i,j) \rightarrow \mu^{(i,j)}((0,1]), \quad j \neq i,
$$
in particular, starting from the position $n$, the probability of changing type is asymptotically $$n^{-\gamma} \sum_{j \in \{1,\ldots,\kappa\}\backslash\{i\}} \mu^{(i,j)}((0,1]).$$ 
Concerning large jumps, the probability, starting from position $n$, to do a jump larger than $n\varepsilon$  is asymptotically, for a.e. $\varepsilon>0$, $$n^{-\gamma}\left(\sum_{j\in \{1,\ldots,\kappa\}\backslash\{i\}}\mu^{(i,j)}((0,1-\varepsilon]) +\int_{(0,1-\varepsilon]}\frac{\mu^{(i,i)}(\mathrm dx)}{1-x}\right).$$
Note that the later quantity is finite but tends to $\infty$ when $\varepsilon \rightarrow 0$ and $\int_{(0,1)}(1-x)^{-1}\mu^{(i,i)}(\mathrm dx)$ is infinite.

\bigskip

\noindent \textbf{The limiting process.} From the measures $\mu^{(i,j)}$ appearing in the limit of $(\mathsf H_{\mathrm{cr}})$, we construct the following characteristics of a MAP:
\begin{enumerate}
\item[$\circ$] for all $i\in \{1,\ldots,\kappa\}$, \ $\psi_i(q)=\mu^{(i,i)}(\{1\})q+\int_{(0,1)} (1-x^{q})\frac{\mu^{(i,i)}(\mathrm dx)}{1-x}, \quad q \geq 0,$
\item[$\circ$] for all $i,j\in \{1,\ldots,\kappa\}$, $i\neq j$, \
$\lambda_{i,j} B_{i,j}=\mu^{(i,j)} \circ (-\log)^{-1}.$
\end{enumerate}
Under the assumptions a) and b) of $(\mathsf H_{\mathrm{cr}})$, we see that these characteristics satisfy (\ref{HypNotConstant}), hence the MAP is absorbed at 0 in finite time a.s.

\bigskip

\noindent \textbf{Changing time.} In order to slow down time in $(X^{(i)}_n,J^{(i)}_n)$ and observe in the scaling limit a regular MAP, we introduce the following time--change: 
\begin{equation}
\label{timechange}
\tau_n^{(i)}(t):=\inf\left\{u\geq 0 :\int_0^u \left(\frac{X_n(\lfloor n^{\gamma} r\rfloor)}{n}\right)^{-\gamma} \mathrm dr>u\right\}.
\end{equation}
We then define a new c\`adl\`ag process $(Z_n^{(i)})$ by
\begin{equation}
\label{Zn}
Z_n^{(i)}(t):=\frac{X_n^{(i)}(\lfloor n^{\gamma} \tau_n^{(i)}(t)\rfloor)}{n}, \quad t \geq 0. 
\end{equation}

\bigskip

\noindent We can now state the main result of this section.

\begin{thm} 
\label{ThCritical}
Under assumption $(\mathsf H_{\mathrm{cr}})$, for all types $i$,
$$
\left(\Bigg(\frac{X^{(i)}_n(\lfloor n^{\gamma}t \rfloor)}{n}, Z_n^{(i)}(t)\Bigg), t\geq 0\right) \ \overset{\mathrm{(d)}}{\underset{n \rightarrow \infty} \longrightarrow} \ \big(X^{(i)},  Z^{(i)}\big),
$$
where $-\log(Z^{(i)})$
is the position component of a \emph{MAP} starting from $(0,i)$ with the characteristics $\psi_j,\lambda_{j,l}, B_{j,l}$ defined above, and $X^{(i)}=Z^{(i)}(\rho^{(i)}(\cdot))$ with
$$
\rho^{(i)}(t)=\inf \left\{u:\int_0^u \big(Z^{(i)}(r)\big)^{\gamma} \mathrm dr>t\right\}.
$$
The topology is the product topology on $\mathcal D\left([0,\infty),[0,\infty)\right)^2$, where $\mathcal D\left([0,\infty),[0,\infty)\right)$ is the set of non--negative c\`adl\`ag functions defined on  $[0,\infty)$, endowed with the Skorokhod topology.
\end{thm}

\noindent \textbf{Remark.} It is possible to incorporate types in the above convergence, however  this point involves some subtleties (in the limit, the type process of the MAP is not c\`adl\`ag in general). We refer to the forthcoming Lemma \ref{LemCoupe}, where the convergence of types at the times of type--change
is  shown when the limiting process is supposed to have no absorbing types.  

\bigskip

Next, we want to compare the absorption times. This will be essential for the applications. We emphasize that it is not a direct consequence of the previous theorem. Let $I^{(i)}$ be the absorption time at 0 of $X^{(i)}$ and recall that it is finite a.s.

\begin{thm}
\label{ThCriticalJoint} Additionally to $(\mathsf H_{\mathrm{cr}})$, assume that for all types $i$, there exists a type $j$ such that \linebreak $\mu^{(i,j)}((0,1))>0$. Then, jointly with the previous convergence,
$$
 \frac{A_n^{(i)}}{n^{\gamma}}\overset{\mathrm{(d)}}{\underset{n \rightarrow \infty} \longrightarrow}I^{(i)},
$$
and for all $a \geq 0$, $\mathbb E \big[ \big(I^{(i)}\big)^a\big]<\infty$ and
$$
\mathbb E\left[\left(\frac{A_n^{(i)}}{n^{\gamma}}\right)^a \right] {\underset{n \rightarrow \infty} \longrightarrow} \mathbb E \big[ \big(I^{(i)}\big)^a\big].
$$
\end{thm}

\bigskip

\noindent \textbf{Remark: possible extensions. 1.} The additional assumption of Theorem \ref{ThCriticalJoint} is not necessary. 
%For example, one can easily believe that the convergences of Theorem \ref{ThCriticalJoint} are still valid when assuming that one type does macroscopic jumps and the matrix of type--changes is irreducible in the scale $n^{\gamma}$. 
Other variants are possible but we will not treat them here. Let us simply note, for the interested reader, that the conclusions of Theorem \ref{ThCriticalJoint} are valid as  soon as $(\mathsf H_{\mathrm{cr}})$ is satisfied and for all $a\geq 0$ the moments $\mathbb E[(n^{-\gamma}A^{(i)}_n)^a]$ are uniformly bounded in $n$ -- see the proof of Theorem \ref{ThCriticalJoint} in Section \ref{sec:proofTh2}.

\noindent \textbf{2.} We believe that the two theorems above remain valid if we more generally assume that the measures $\mu^{(i,j)}$ are supported by $[0,1]$, instead of $(0,1]$. This more general setting includes cases where the limiting MAP may be killed. However, in this situation, the proofs require more work than the unkilled cases. Since we will not really need this generalization in applications, this fact is left as a remark and we focus here on cases where the measures $\mu^{(i,j)}$ are supported by $(0,1]$.

\bigskip

Before entering the proofs of Theorem \ref{ThCritical} and Theorem \ref{ThCriticalJoint}, we start by noticing in Section \ref{sec:foreword} that we can do additional assumptions on the model, without loss of generality. Assuming these additional assumptions, we show in Section \ref{sec:truncation} that for all positive integers $p$, the rescaled process $X_n$ killed at its $p$--th  type--change time converges to the  process $X^{(i)}$ killed at its $p$--th  type--change time, as well as related quantities. To see this, we use the results of \cite{HM11} to study the monotype processes on the type--constancy intervals and then ``glue" these processes together. Section  \ref{sec:proofTh1} is devoted to the proof of Theorem \ref{ThCritical} and Section \ref{sec:proofTh2} to that of Theorem \ref{ThCriticalJoint}.

\subsection{Foreword: additional assumptions}\label{sec:absorption}
\label{sec:foreword}

To simplify the proofs, we will do some additional assumptions on the model, without loss of generality. 

\medskip

\noindent \textbf{On absorbing states.} The marginal $X$ of the process $(X,J)$ may have different absorbing states. Note however that under $(\mathsf H_{\mathrm{cr}})$ its set of absorbing states if finite, otherwise there would exist a type $i$ such that $\mu^{(i,j)}((0,1])=0$ for all $j \in \{1,\ldots,\kappa\}$. Now, consider the process defined by
$$
X_n^{(i),\mathrm q}(k)=X_n^{(i)}(k) \mathbbm 1_{\{k \leq A_n^{(i)}\}}, \quad \forall k \in \mathbb N.
$$
Then $(X_n^{(i),\mathrm q},J_n^{(i)})$ is a Markov chain with transition probabilities defined by $q_{k,i}(m,j)=p_{k,i}(m,j)$ for each integer $k$ that is not absorbing, and then all types $i,j$ and all integers $m$, and $q_{a,i}(0,i)=p_{a,i}(a,i)$ for all absorbing states $a$ and all types $i$ (recall that by convention, $a$ is aborbing if there exist a type $i$ such that $p_{(a,i)}(a,i)=1$). Clearly, $X_n^{(i),\mathrm q}$ is absorbed at 0 at time $A_n^{(i)}$ or $A_n^{(i)}+1$, and $\sup_{k \geq 0}|X_n^{(i),\mathrm q}(k)-X_n^{(i)}(k)| \leq \max{\{a: a\text{ is absorbing}\}}$. Moreover, $(\mathsf H_{\mathrm{cr}})$ is clearly satisfied for the transition probabilities $(q_{n,i}(m,j))$ if it holds for the transition probabilities $(p_{n,i}(m,j))$. Consequently, if  Theorem \ref{ThCritical} and Theorem \ref{ThCriticalJoint} are proved for the process $(X_n^{(i),\mathrm q},J_n^{(i)})$, they will also hold for $(X_n^{(i)}, J_n^{(i)})$. 

\smallskip

Hence, 
\textbf{in the forthcoming proofs, we can and will assume that $X^{(i)}_n$ is always absorbed at 0}, with no loss of generality.

\bigskip

\noindent \textbf{On absorbing types in the limit.} Let
$$
\mathcal A^{\mathrm{type}}:=\left\{i\in \{1,\ldots,\kappa\}:\sum_{j \in \{1,\ldots,\kappa\}\backslash \{i\}} \mu^{(i,j)}((0,1])=0\right\}
$$
be the set of types that are absorbing in the limiting MAP $(-\log(Z))$. If $\mathcal A^{\mathrm{type}}$ is not empty, then the upcoming Lemma \ref{LemCoupe} will fail. In order to overcome this difficulty, we can create phantom types in that case. We detail the idea when  $\mathcal A^{\mathrm{type}}$ is reduced to one type and note that it generalizes immediately to the cases where  $\mathcal A^{\mathrm{type}}$ contains more types. So assume that $\mathcal A^{\mathrm{type}}=\{i_0\}$, create a new type $\kappa+1$ and set for all integers $n,m$: 
\begin{enumerate} 
\item[$\circ$] $p^*_{n,i_0}(m,i_0)=(1-n^{-\gamma})p_{n,i_0}(m,i_0)$
\item[$\circ$] $p^*_{n,i_0}(m,\kappa+1)=n^{-\gamma}p_{n,i_0}(m,i_0)$
\item[$\circ$] $p^*_{n,i_0}(m,j)=p_{n,i_0}(m,j)$ for $j\neq i_0,\kappa+1$
\item[$\circ$] $p^*_{n,\kappa+1}(m,\kappa+1)=(1-n^{-\gamma})p_{n,i_0}(m,i_0)$
\item[$\circ$] $p^*_{n,\kappa+1}(m,i_0)=n^{-\gamma}p_{n,i_0}(m,i_0)$
\item[$\circ$] $p^*_{n,\kappa+1}(m,j)=p_{n,i_0}(m,j)$ for $j\neq i_0,\kappa+1$
\item[$\circ$] $p^*_{n,i}(m,j)=p_{n,i}(m,j)$ for $i\neq i_0,\kappa+1$ and all types $j \in \{1,\ldots,\kappa\}$.
\end{enumerate}
One can clearly couple the construction of $(X_n^{(i)},J_n^{(i)})$ with that of a $\mathbb Z_+ \times \{1,\ldots,\kappa+1\}$--valued Markov chain $(X_n^{*,(i)},J_n^{*,(i)})$ with transition probabilities 
$(p^*_{(n,i)}(m,j))$ and such that $X_n^{(i)}=X_n^{*,(i)}$. These new transition probability satisfy $(\mathsf{H_{\mathrm{cr}}})$ with $\mu^{*,(i,j)}=\mu^{(i,j)}$ for all $i \neq i_0,\kappa+1, j \leq \kappa$, $\mu^{*,(i_0,j)}=\mu^{*,(\kappa+1,j)}=\mu^{(i_0,j)}=0$ for all  $j \neq i_0,\kappa+1$, $\mu^{*,(i_0,i_0)}=\mu^{*,(\kappa+1,\kappa+1)}=\mu^{(i_0,i_0)}$ and finally $\mu^{*,(i_0,\kappa+1)}=\mu^{*,(\kappa+1,i_0)}=\delta_1$. Hence the set of absorbing types  in the limiting MAP is here empty, and clearly, if Theorem \ref{ThCritical} and Theorem \ref{ThCriticalJoint} hold for the $*$--model, they also hold for the initial model. 

\smallskip

Hence, 
\textbf{in the forthcoming proofs, we can and will always assume that $\mathcal A^{\mathrm{type}}$ is empty}, with no loss of generality.

\subsection{Truncation}
\label{sec:truncation}

Before proving, strictly speaking, Theorem \ref{ThCritical}, we first study the asymptotic behavior of the process $(n^{-1}X^{(i)}_n(\lfloor n^{\gamma} \cdot \rfloor)$ killed at the first time when it changes type, then killed at the second time when it changes type, and so on. To this end, recall the definition of the absorption time $A_n^{(i)}$ in (\ref{DiscreteAbsorption}) and consider
$$T_{n}^{(i)}(1) := \inf\{k: J_n^{(i)}(k) \neq i \} \wedge A_n^{(i)}$$ the first time where either the process $X^{(i)}_n$ changes its type or is absorbed at 0, and let 
$X_n^{(i)}{|_1}$ be $X^{(i)}_n$ killed at time  $T_{n}^{(i)}(1)$, i.e.:
$$
X_n^{(i)}{|_1}(k)=X^{(i)}_n (k)\mathbbm 1_{\{k < T_{n}^{(i)}(1)\}}, \quad k \geq 0.
$$
We then define recursively $T_{n}^{(i)}(p)$ the $p$--th time at which $X^{(i)}_n$ change its type (with the convention that it is equal to $A_n^{(i)}$ if it reaches 0 before a $p$--th type change) and 
$$
X_n^{(i)}{|_p}(k)=X^{(i)}_n (k)\mathbbm 1_{\{k < T_{n}^{(i)}(p)\}},  \quad k \geq 0.
$$
Lastly, we define similarly the quantities $T^{(i)}(p)$, $X^{(i)}{|_{p}}$ for the limiting process $X^{(i)}$. Note in particular that $X^{(i)}{|_{1}}$ is the Lamperti transform of a standard subordinator with Laplace exponent
$$
\psi(q)=\sum_{j \in \{1,\ldots,\kappa\},j\neq i} \mu^{(i,j)}((0,1])+\mu^{(i,i)}(\{1\})q+\int_{(0,1)} (1-x^{q})\frac{\mu^{(i,i)}(\mathrm dx)}{1-x}.
$$
The goal of this section is to prove the following lemma. We recall that we have assumed that the set 
$\mathcal A^{\mathrm{type}}$ of absorbing types in the limit is empty.

\begin{lem}
\label{LemCoupe}
For all types $i \in \{1,\ldots, \kappa\}$ and all integers $p \geq 1$
\begin{eqnarray*}
\label{EqRec}
&& \left(\frac{X_n^{(i)}{|_p}(\lfloor n^{\gamma}  \cdot \rfloor)}{n}, \ \frac{T_n^{(i)}(p)}{n^{\gamma}}, \ \frac{X_n^{(i)}\big(T_n^{(i)}(p)-1\big)}{n}, \ \frac{X_n^{(i)}\big(T_n^{(i)}(p)\big)}{n}, \ J_n^{(i)}\big(T_n^{(i)}(p)\big) \right)\ \nonumber \\
&&\overset{\mathrm{(d)}}{\underset{n \rightarrow \infty}\longrightarrow} \left(X^{(i)}{|_{p}}, \ T^{(i)}(p), \ X^{(i)}\big(T^{(i)}(p)-\big), \ X^{(i)}\big(T^{(i)}(p)\big), \ J^{(i)}\big(T^{(i)}(p) \big) \right).
\end{eqnarray*}
Moreover, 
\begin{equation*} 
\mathbb E\left[\left(\frac{T_{n}^{(i)}(p)}{n^{\gamma}}\right)^a\right] \rightarrow \mathbb E\left[\left(T^{(i)}(p)\right)^a\right]  \quad \text{for all } a\geq 0.
\end{equation*}\end{lem}

\begin{proof} We proceed by induction on $p$. For $p=1$, the proof relies essentially on Theorems 1 and 2  in \cite{HM11}. The induction then uses the Markov property of the process $(X_n^{(i)},J_n^{(i)})$.

\medskip

\noindent $\bullet$ First, let $p=1$ and note that (for all types $i$) the transition probabilities of the chain $X^{(i)}|_1$ are
$$
q_n^{(i)}(m):=p_{n,i}(m,i) \quad \text{for } m \neq 0  
$$
and
$$
q_n^{(i)}(0):=p_{n,i}(0,i)+\sum_{j \neq i} \sum_{m \geq 0} p_{n,i}(m,j).
$$
By ($\mathsf H_{\mathrm{cr}}$), 
$$
n^{\gamma} \sum_{k\geq0} f\left(\frac{k}{n}\right) \left(1-\frac{k}{n}\right) q_n^{(i)}(k) \underset{n\rightarrow \infty} \longrightarrow \int_{(0,1]} f(x)\mu^{(i,i)}(\mathrm dx) + f(0) \sum_{j \neq i} \mu^{(i,j)}((0,1]).
$$
Consequently, $X_n^{(i)}|_1$ a monotype Markov chain whose transition probabilities satisfy the hypothesis of Theorem 1 and Theorem 2  in \cite{HM11} and we get that
$$
\left(\left(\frac{X_n^{(i)}{|_1}(\lfloor n^{\gamma} t\rfloor)}{n},t\geq 0\right), \frac{T_{n}^{(i)}(1)}{n^{\gamma}}\right) \overset{(\mathrm d)}{\underset{n\rightarrow \infty}\longrightarrow}   \left(X^{(i)}{|_1}, T^{(i)}(1)\right)
$$
together with the convergence of all positive moments of $n^{-\gamma}T_{n}^{(i)}(1)$ towards those of $T^{(i)}(1)$.

\medskip

\noindent $\bullet$ Second, we would like to apply Lemma \ref{lemSko2} to get that
\begin{equation}
\label{TempsPrecedent}
\left(\left(\frac{X_n^{(i)}{|_1}(\lfloor n^{\gamma} t\rfloor)}{n},t\geq 0\right), \frac{T_{n}^{(i)}(1)}{n^{\gamma}}, \frac{X_n^{(i)}\big(T_{n}^{(i)}(1)-1\big)}{n}\right) \overset{(\mathrm d)}{\underset{n \rightarrow \infty} \longrightarrow}   \left(X^{(i)}{|_1}, T^{(i)}(1), X^{(i)}\big(T^{(i)}(1)-\big)\right).
\end{equation}
The proof is not immediate because it is not clear that we are always in situations (i) or (ii) of Lemma \ref{lemSko2}. We need to introduce another Markov chain $\big(\tilde X_n^{(i)}\big)$ on $\mathbb Z_+$ with transition probabilities
$$
\tilde p_n(m):=p_{n,i}(m,i) \quad \text{for } m \neq  \left\lfloor \frac{n}{2} \right\rfloor 
$$
and
$$
\tilde p_n\left(\left\lfloor \frac{n}{2} \right\rfloor \right):=p_{n,i}\left(\left\lfloor \frac{n}{2} \right\rfloor ,i\right)+\sum_{j \neq i} \sum_{m \geq 0} p_{n,i}(m,j).
$$
Let $\tilde A^{(i)}_n$ denote its absorption time. 
By ($\mathsf H_{\mathrm{cr}}$) and Theorem 1 and Theorem 2 in \cite{HM11}, \linebreak $\big(\tilde X_n^{(i)}(\lfloor n^{\gamma} \cdot \rfloor/n,\tilde A^{(i)}_n/n^{\gamma}\big)$ converges to $(\tilde X^{(i)},\tilde I^{(i)})$ where $\tilde X^{(i)}$ is a self--similar Markov process  with characteristics
$$
\tilde \psi(q)=\mu^{(i,i)}(\{1\})q+\int_{(0,1)} \frac{1-x^{q}}{1-x}\left(\mu^{(i,i)}(\mathrm dx) +\frac{1}{2} \sum_{j\neq i} \mu^{(i,j)}((0,1]) \delta_{\frac{1}{2}} (\mathrm dx)\right),
$$
and $\tilde I^{(i)}$ is the time at which it is absorbed at 0. 
In fact, clearly, one can construct a joint version of the pair $\big(X_n^{(i)}{|_1},\tilde X_n^{(i)}\big)$ such that
$$
X_n^{(i)}{|_1}(k) = \tilde X_n^{(i)}(k) \mathbbm 1_{\{k <T_n^{(i)}(1)\}}.
$$ 
Together with the convergence of the rescaled process $X_n^{(1)}|_1$ settled above, this implies that 
$$\left(\left(\frac{X_n^{(i)}{|_1}(\lfloor n^{\gamma} t\rfloor)}{n},t\geq 0\right), \frac{T_{n}^{(i)}(1)}{n^{\gamma}},\left(\frac{\tilde X_n^{(i)}(\lfloor n^{\gamma} t\rfloor)}{n},t\geq 0\right), \frac{\tilde A_{n}^{(i)}}{n^{\gamma}}\right) $$ converges in distribution towards a quadruplet $\big(X,I, \tilde X, \tilde I \big)$, where $(X,\tilde X)$ is a coupling of $\gamma$--self--similar Markov processes such that $X(t) = \tilde X(t) \mathbbm 1_{\{t<I\}}$ for all $t \geq 0$, and $I<\tilde I$ a.s. To see that $I < \tilde I$ a.s., note that if $(\xi(t)\mathbbm 1_{\{t<T\}}+\infty \mathbbm 1_{\{t \geq T\}}, t\geq 0)$ denotes the underlying subordinator of $X$, which is killed at rate $\sum_{j \neq i} \mu^{(i,j)}((0,1])>0$, then the underlying subordinator of $\tilde X$ is $\xi+\tilde \xi$ where $\tilde \xi$ is a subordinator independent of $\xi$ whose first jump arises at time $T$. Next, assume, using Skorokhod's representation theorem, that the convergence of the quadruplet holds a.s. We then get that, a.s.,
\begin{equation}
\label{cvinfty}
\liminf_{n \rightarrow \infty} \frac{X_n^{(i)}\big(T_{n}^{(i)}(1)-1\big)}{n}=\liminf_{n \rightarrow \infty} \frac{\tilde X_n^{(i)}\big(T_{n}^{(i)}(1)-1\big)}{n} \geq \tilde X(I)>0.
\end{equation}
This, together with Lemma \ref{lemSko2} indeed gives (\ref{TempsPrecedent}).

\bigskip

\noindent $\bullet$ Third, we  immediately have by the Markov property of $(X,J)$, together with $(\mathsf H_{\mathrm{cr}})$, that conditionally on $\big(T^{(i)}_n(1),X^{(i)}_n(k), k \leq T^{(i)}_n(1)-1\big)$, and since $X^{(i)}_n(T^{(i)}_n(1)-1) \overset{\mathbb P}\rightarrow \infty$  (by (\ref{cvinfty})), that
$$
\left(\frac{X^{(i)}_n(T^{(i)}_n(1))}{X^{(i)}_n(T^{(i)}_n(1)-1)}, J_n^{(i)}\big(T^{(i)}_n(1)\big) \right)\overset{\mathrm{(d)}}{\underset{n \rightarrow \infty}\longrightarrow} \left(S_{\infty}, J^{(i)}\big(T^{(i)}(1)\big) \right),
$$
where the law of the limit is given by 
$$\mathbb E\left[f\left(S_{\infty},J^{(i)}\big(T^{(i)}(1)\big)\right) \right] = \frac{\sum_{j \neq i} \int_{(0,1]} f(x,j)\mu^{(i,j)}(\mathrm d x)}{\sum_{j \neq i} \mu^{(i,j)}((0,1])}.$$ Together with (\ref{TempsPrecedent}), and  the convergence of all positive moments of $n^{-\gamma}T_{n}^{(i)}(1)$ already mentioned, this finally proves the lemma for $p=1$.

\bigskip

\noindent $\bullet$ Now assume that the lemma is proved for all $q \leq p$ and fix a type $i$. In particular, we have that
$$
C_n:=\left(\left(\frac{X_n^{(i)}{|_p}(\lfloor n^{\gamma} t\rfloor)}{n},t\geq 0\right), \frac{T_n^{(i)}(p)}{n^{\gamma}}, \frac{X_n^{(i)}\big(T_{n}^{(i)}(p)-1\big)}{n}, \frac{X_n^{(i)}\big(T_{n}^{(i)}(p)\big)}{n}, J_n^{(i)}\big(T_n^{(i)}(p)\big) \right) 
$$
converges in distribution towards
$$
C:=\left(X^{(i)}{|_{p}}, T^{(i)}(p), X^{(i)}\big(T^{(i)}(p)-\big), X^{(i)}(T^{(i)}(p)\big), J^{(i)}\big(T^{(i)}(p)\big)\right).
$$
Set also
\begin{eqnarray*}
D_n&:=& \left(\left(\frac{X_n^{(i)}{|_{p+1}}(\lfloor T_n^{(i)}(p)+n^{\gamma}t\rfloor)}{n},t\geq 0 \right), \frac{T_n^{(i)}(p+1)-T_n^{(i)}(p)}{n^{\gamma}}, \frac{X_n^{(i)}(T_n^{(i)}(p+1)-1)}{n} \right., \\
&&\left.\frac{X_n^{(i)}(T_n^{(i)}(p+1))}{n},  J_n^{(i)}\big(T_n^{(i)}(p+1)\big)\right),
\end{eqnarray*}
and for all types $j$
\begin{eqnarray*}
D^{(j)}&:=&\left(\left(X^{(i)}(T^{(i)}(p)) \overline X^{(j)}{|_{1}}(X^{(i)}\big(T^{(i)}(p))^{-\gamma} t\big),t\geq 0\right), \big(X^{(i)}(T^{(i)}(p))\big)^{\gamma} \overline T^{(j)}(1), \right. \\  
&& X^{(i)}\big(T^{(i)}(p)\big) \overline X^{(j)}\big(T^{(j)}(1)-\big), \left. X^{(i)}\big(T^{(i)}(p)\big) \overline X^{(j)}\big(T^{(j)}(1)\big), \overline J^{(j)}\big(\overline T^{(j)}(1)\big)\right)
\end{eqnarray*}
with $\big(\overline X^{(j)},\overline J^{(j)}\big)$ independent of $\big(X^{(i)},J^{(i)}\big)$ and distributed as $\big(X^{(j)},J^{(j)}\big)$.

Then apply the strong Markov property at the stopping time $T_n^{(i)}(p)$ together with the fact that the lemma holds for $q=1$ and that $C_n$ converges in distribution towards $C$, to get that
$$
\left(C_n,D_n\right)\overset{(\mathrm d)}{\underset{n \rightarrow \infty}\longrightarrow} \big(C, D^{(J^{(i)}(T^{(i)}(p))} \big).
$$
Here we have used the fact that $X_n^{(i)}(T_n^{(i)}(p)) \overset{\mathbb P} \rightarrow \infty$, which is due to the convergence in distribution of this quantity divided by $n$ to $X^{(i)}(T^{(i)}(p))$, which is a.s. strictly positive, since the limiting  MAP changes its type infinitely often since $\mathcal A^{\mathrm{type}}=\emptyset$ and there is no killing.
Lastly, gluing the pieces thanks to Lemma \ref{lemSko1} leads to the statement of the first part of the lemma for $p+1$.

It remains to prove the convergence of all positive moments of $T_n^{(i)}(p+1)/n^{\gamma}$. Since we already know that this r.v. converges in distribution to $T^{(i)}(p+1)$, the convergence of moments will be proved if we check that
$$
\sup_{n\geq 1} \mathbb E \left[\left(\frac{T_n^{(i)}(p+1)}{n^{\gamma}} \right)^a \right]<\infty, \quad \forall a \geq 0.
$$
This is a direct consequence of the induction hypothesis and the fact that
$$
\sup_{n\geq 1} \mathbb E \left[\left(\frac{T_n^{(i)}(p+1)}{n^{\gamma}} \right)^a\right] \leq c_a \left(\sup_{n\geq 1} \mathbb E \left[\left(\frac{T_n^{(i)}(p)}{n^{\gamma}} \right)^a\right]+\sup_{n\geq 1}\mathbb E\left[\left(\frac{T_n^{(i)}(p+1)-T_n^{(i)}(p)}{n^{\gamma}} \right)^a\right]\right)
$$
for some finite $c_a$. Indeed, on the right--hand side the first supremum is finite, applying the induction hypothesis at $p$. And the second supremum is also finite, by the Markov property and the induction hypothesis applied at the initial rank $1$.
\end{proof}

\subsection{Proof of Theorem \ref{ThCritical}: scaling limits of the position marginal}
\label{sec:proofTh1}

Let
$$
Y_n^{(i)}(t):=\frac{X_n^{(i)}(\lfloor n^{\gamma}t\rfloor)}{n} \quad \text{and} \quad Y_n^{(i)}|_p(t):=\frac{X_n^{(i)}|_p(\lfloor n^{\gamma}t\rfloor)}{n}
$$
and note that the second process can be interpreted as $Y_n^{(i)}$ killed at its $p$--th type change, which is denoted by $T^{Y,(i)}_n(p)$ (and equal to $T^{(i)}_n(p)/n^{\gamma}$).

\bigskip

\noindent \textit{Proof of Theorem \ref{ThCritical}.} $\bullet$ Consider any Lipschitz function $f:\mathcal D([0,\infty),[0,\infty)) \rightarrow [0,\infty)$, say with Lipschitz constant $c_f$, where for distance on $\mathcal D([0,\infty),[0,\infty))$ we consider the classical distance metrizing the Skorokhod topology, see formula (16.4) in Billingsley \cite{Bill99} for a precise definition. We denote this distance by $d_{\mathrm{sko}}$ and recall that it is smaller than the uniform distance on $[0,\infty)$. 
Our goal is to show that $$\mathbb E[f(Y_n^{(i)})] \underset{n \rightarrow \infty}\rightarrow \mathbb E[f(X^{(i)})].$$
If proved for any Lipschitz function $f$, this will ensure that 
$$
Y_n^{(i)} \underset{n \rightarrow \infty}{\overset{(\mathrm d)}\longrightarrow} X^{(i)}.
$$
So fix $f$, let $\varepsilon>0$ and take an integer  $p$ so that $\mathbb E[X^{(i)}(T^{(i)}(p))] \leq \varepsilon$ (such a $p$ exists, by dominated convergence and since $X^{(i)}(T^{(i)}(p))$ converges to 0 -- see (\ref{eq:cvchgttype})). Then by Lemma \ref{LemCoupe},  we have that 
$\mathbb E[Y_n(T_n^{Y,(i)}(p))] \leq 2 \varepsilon$ for all $n$ large enough, say for $n \geq n_{\varepsilon,p}$. Thus,
\begin{eqnarray*}
\left|\mathbb E \left[f(Y_n^{(i)})- f(Y_n^{(i)}|_p)\right]\right| &\leq& c_f \mathbb E \left[d_{\mathrm{Sko}}(Y_n^{(i)},Y_n^{(i)}|_p)\right] \\
&\leq& c_f \mathbb E \left[Y_n^{(i)}(T_n^{Y,(i)}(p))\right] \leq 2c_f\varepsilon, \text{ for } n \geq n_{\varepsilon,p}.
\end{eqnarray*}
(The second inequality is due to the fact that $d_{\mathrm{Sko}}$ is smaller than the uniform distance on $[0,\infty)$ and the fact that $Y_n^{(i)},Y_n^{(i)}|_p$ coincide on $[0,T_n^{Y,(i)}(p))$, together with $Y_n^{(i)}|_p$ is null on $[T_n^{Y,(i)}(p), \infty)$ and $Y_n^{(i)}$ is non--increasing.) Similarly, we get that
$$
\left|\mathbb E \left[f(X^{(i)})- f(X^{(i)}|_p)\right]\right| \leq c_f\varepsilon.
$$
This entails that
\begin{eqnarray*}
\left|\mathbb E \left[f(Y_n^{(i)})- f(X^{(i)})\right]\right| \leq 3c_f\varepsilon + \left|\mathbb E \left[f(Y_n^{(i)}|_p)- f(X^{(i)}|_p)\right]\right|.
\end{eqnarray*}
Besides, by Lemma \ref{LemCoupe}, $Y_n^{(i)}|_p \rightarrow X^{(i)}|_p$ in distribution, so  
finally, we have proven that for all $\varepsilon>0$ and then all $n$ large enough,
$$
\left|\mathbb E \left[f(Y_n^{(i)})- f(X^{(i)})\right]\right| \leq (3 c_f +1) \varepsilon.
$$

\bigskip

\noindent $\bullet$ The convergence of the pair $\big(Y_n^{(i)},  Z_n^{(i)}\big)$ to $\left(X^{(i)},  Z^{(i)}\right) $ is then a consequence of Lemma \ref{lampsko}. $\hfill \square$

\subsection{Proof of Theorem \ref{ThCriticalJoint}: scaling limit of the absorption time}
\label{sec:proofTh2}

We start by proving the following lemma, using a coupling with a monotype Markov chain, and then turn to the proof of Theorem \ref{ThCriticalJoint}.

\begin{lem}
\label{lem:tightness}
Assume $(\mathsf H_{\mathrm{cr}})$ and that for all types $i \in \{1,\ldots ,\kappa\}$, there exists a type $j$ such that $\mu^{(i,j)}((0,1))>0$. Then for all types $i \in \{1,\ldots,\kappa\}$ and all $a \geq 0$,
$$
\sup_{n \in \mathbb N}\mathbb E\left[\left(\frac{A_n^{(i)}}{n^{\gamma}}\right)^a\right]<\infty.
$$
In particular, the sequence $\big(n^{-\gamma}A_n^{(i)}\big)$ is tight, $\forall i \in \{1,\ldots ,\kappa\}$.
\end{lem}

\begin{proof}
Since the number of types is finite, our additional assumption implies the existence of $r \in (0,1)$ such that $$\sum_{j \in \{1,\ldots,\kappa\}} \mu^{(i,j)}((0,r))>0, \quad \text{for all types } i \in \{1,\ldots ,\kappa\}.$$ By ($\mathsf H_{\mathrm{cr}}$), this in turn implies the existence of $c \in (0,1)$ and $n_0 \in \mathbb N$ such that
\begin{equation}
\label{eq:couplage}
\sum_{k=0}^{\lfloor a n \rfloor} p_{n}^{(i)}(k) \geq \frac{c}{n^{\gamma}}, \quad \forall n \geq n_0, \text{ }\forall i \in \{1,\ldots ,\kappa\}
\end{equation}
(we recall that $p_n^{(i)}(k)=\sum_{j \in \{1,\ldots ,\kappa\}}p_{n,i}(k,j)$ is the transition probability of the position marginal $X$ when its type is $i$).
Besides, the chain $X$ is assumed to be always absorbed at 0. Hence,  for  all $\ell \in \{1,\ldots ,n_0-1\}$ and all types $i \in \{1,\ldots ,\kappa\}$, there exists a non--negative integer $k_{\ell,i}\leq \ell-1$ such that $p_{\ell}^{(i)}(k_{\ell,i})>0$. We let
\begin{equation}
\label{eq:couplage2}
d:=\min_{(\ell,i) \in  \{1,\ldots ,n_0-1\} \times \{1,\ldots ,\kappa\}} p_{\ell}^{(i)}(k_{\ell,i})>0.
\end{equation}
Consider now a $\mathbb Z_+$--valued Markov chain $Y$ with transition probabilities
$$
\begin{array}{ll}
\vspace{0.15cm}
\circ \hspace{0.2cm} q_n(\lfloor r n \rfloor)=cn^{-\gamma} & \forall \ n \geq \lceil \frac{n_0}{r} \rceil \\ 
\vspace{0.15cm}
\circ \hspace{0.2cm} q_n(n)=1- cn^{-\gamma} & \forall \ n \geq \lceil \frac{n_0}{r} \rceil \\ 
\vspace{0.15cm}
\circ \hspace{0.2cm} q_n(n-1)=d & \forall \ 1 \leq n <\lceil \frac{n_0}{r} \rceil  \\
\circ \hspace{0.2cm} q_n(n)=1-d & \forall \ 1 \leq n <\lceil \frac{n_0}{r} \rceil 
\end{array}
$$
(and $q_0(0)=1$) and let $Y_n$ denote a version of the chain $Y$ starting from $n$. Fix a type $i$. Using (\ref{eq:couplage}), (\ref{eq:couplage2}) and the fact that $n \mapsto cn^{-\gamma}$ is decreasing, it is easy to see that one can couple the construction of the chains $(X^{(i)}_n,J^{(i)}_n)$ and $Y_n$ such that $$X^{(i)}_n(k) \leq Y_n(k), \quad  \text{for all } k \in \mathbb Z_+.$$ Note that the chain $Y_n$ is necessarily absorbed at 0, so we also have that
$$
A_n^{(i)}\leq A_{Y,n}
$$
where $A_{Y,n}$ is the absorption time of $Y_n$.
Moreover, clearly, 
$$
n^{\gamma}  \sum_{m=0}^n f\left(\frac{m}{n}\right)\left(1-\frac{m}{n} \right)q_n(m) \underset{n\rightarrow \infty}\longrightarrow c(1-r)f\left(r\right)
$$
for all continuous functions $f:[0,1] \rightarrow \mathbb R$.
So we are exactly in the conditions of the monotype setting studied in \cite[Theorems 1 and 2]{HM11}. In particular, we know that there is a positive r.v. $I_Y$ with all positive moments finite ($I_Y$ is the absorption time of the  self--similar process arising as scaling limit of $Y_n$) such that
$$
\frac{A_{Y,n}}{n^{\gamma}} \underset{n\rightarrow \infty}{\overset{\mathrm{(d)}}\longrightarrow} I_Y \quad \text{and} \quad  \mathbb E\left[\left(\frac{A_{Y,n}}{n^{\gamma}}\right)^a\right] \underset{n\rightarrow \infty}\longrightarrow \mathbb E\left[ (I_Y)^a\right].
$$
Since $A_n^{(i)}\leq A_{Y,n}$ for all $n$, this leads to the statement of the lemma.
\end{proof}

\bigskip

\noindent \textit{Proof of Theorem  \ref{ThCriticalJoint}}. The initial type $i$ is fixed. Theorem \ref{ThCritical} together with Lemma \ref{lem:tightness} imply the tightness of 
\begin{equation}
\label{eq:triplet}
\left(\frac{X_n^{(i)}(\lfloor n^{\gamma} \cdot \rfloor)}{n}, Z_n^{(i)}, \frac{A_n^{(i)}}{n^{\gamma}}\right), \ n \geq 1.
\end{equation}
Consider then a converging subsequence, indexed, say, by ($\psi(n)$), that converges to a limit denoted by $(X^{(i)},Z^{(i)},\sigma^{(i)})$. Since Theorem \ref{ThCritical} is already proven, the only thing we have to check is that $\sigma^{(i)}=I^{(i)}$ a.s., with $I^{(i)}$ the extinction time of $X^{(i)}$. Indeed, if this holds for all converging subsequences, this will 
\begin{enumerate}
\item[(1)] imply the convergence in distribution of (\ref{eq:triplet})  to $(X^{(i)},Z^{(i)},I^{(i)})$
\item[(2)] imply the convergence of all positive moments of $n^{-\gamma}A_n^{(i)}$ to those of $I^{(i)}$ (which are then necessarily finite), by using the convergence of $n^{-\gamma}A_n^{(i)}$ to $I^{(i)}$ together with the bounds of Lemma \ref{lem:tightness}.
\end{enumerate}
So, now, consider the converging subsequence indexed by ($\psi(n)$).
By the Skorokhod representation theorem, we may assume that the convergence holds almost surely. Then, note that 
\[
\frac{A_{\psi(n)}^{(i)}}{\psi(n)^{\gamma}}=\int_0^{\infty}\Big(Z_{\psi(n)}^{(i)}\Big)^{\gamma}(r) \mathrm dr.
\]
Hence we have in the limit, by Fatou's lemma, that $\sigma^{(i)} \geq I^{(i)}=\int_0^{\infty}\big(Z^{(i)}\big)^{\gamma}(r) \mathrm dr$ a.s. 

\vspace{0.07cm}

To prove that $\sigma^{(i)} = I^{(i)}$ a.s., it is now sufficient to show that $\mathbb E[\sigma^{(i)}] \leq \mathbb E[I^{(i)}]$. Note that, by Fatou's lemma again,
$$
\mathbb E[\sigma^{(i)}] \leq \liminf_n \mathbb E\left[ \frac{A_{\psi(n)}^{(i)}}{\psi(n)^{\gamma}}\right] \leq  \limsup_n \mathbb E\left[ \frac{A_{n}^{(i)}}{n^{\gamma}}\right]
$$
so it is actually enough to show that the latter $\limsup$ is smaller than $\mathbb E[I^{(i)}]$.
Recall that we have assumed that the set of absorbing types $\mathcal A^{\mathrm{type}}$ is empty. 
Fix $\varepsilon>0$ and then $p$ large enough so that
$$
\mathbb E\left[ \Big(X^{(i)}\big(T^{(i)}(p)\big)\Big)^{\gamma}\right] \leq \varepsilon
$$
(recall that $T^{(i)}(p)$ is the $p$--th time of type change in $X^{(i)}$ and recall (\ref{eq:cvchgttype})). 
By Lemma \ref{LemCoupe}, we know that $n^{-1}X^{(i)}_n\big(T^{(i)}_n(p)\big)$ converges in distribution to $X^{(i)}\big(T^{(i)}(p)\big)$, which implies
\begin{equation}
\label{eq:majoeps}
\limsup_n \mathbb E\left[ \left(\frac{X^{(i)}_n\big(T_n^{(i)}(p)\big)}{n}\right)^{\gamma}\right]
\leq 2\varepsilon.
\end{equation}
By Lemma \ref{LemCoupe} again,  the expectation of $n^{-\gamma}T^{(i)}_n(p)$ converges to that of $T^{(i)}(p)$ which implies
\begin{equation}
\label{ineg1}
\mathbb E \left[\frac{T^{(i)}_n(p)}{n^{\gamma}}\right] \leq \varepsilon + \mathbb E[T^{(i)}(p)] \leq \varepsilon + \mathbb E[I^{(i)}], \quad \text{for all $n$ large enough}.
\end{equation}
Then, the Markov property at (the stopping time) $T_n^{(i)}(p)$ implies that
$$
A_n^{(i)}-T_n^{(i)}(p)=\tilde A^{J_n^{(i)}(T_n^{(i)}(p))}_{X^{(i)}_n(T_n^{(i)}(p))},
$$
where given $(X^{(i)}_n(T_n^{(i)}(p)),J^{(i)}_n(T_n^{(i)}(p)))=(m,j)$,  the r.v. in the right--hand side is distributed as $A^{(j)}_{m}$. Note that
\begin{eqnarray*}
\mathbb E\left[\frac{\tilde A^{J_n^{(i)}(T_n^{(i)}(p))}_{X^{(i)}_n(T_n^{(i)}(p))}}{n^{\gamma}} \right]&=&\mathbb E\left[\frac{\tilde A^{J_n^{(i)}(T_n^{(i)}(p))}_{X^{(i)}_n(T_n^{(i)}(p))}}{\big(X^{(i)}_n(T_n^{(i)}(p))\big)^{\gamma}}\times \frac{\big(X^{(i)}_n(T_n^{(i)}(p))\big)^{\gamma}}{n^{\gamma}} \right] \\
&\leq& c2 \varepsilon
\end{eqnarray*}
for all $n$ large enough, where $c=\sup_{m \in \mathbb N, j \in \{1,\ldots,\kappa\}}\mathbb E\big[m^{-\gamma}A_m^{(j)}\big]$ is finite, by Lemma \ref{lem:tightness}. Then, writing $A^{(i)}_n=A^{(i)}_n-T_n^{(i)}(p) + T_n^{(i)}(p)$ and recalling (\ref{ineg1}), we get 
$$
\mathbb E \left[\frac{A^{(i)}_n}{n^{\gamma}}\right] \leq 2c\varepsilon + \mathbb E\left[\frac{T_n^{(i)}(p)}{n^{\gamma}}\right] \leq (2c+1)\varepsilon + \mathbb E[I^{(i)}]
$$
for $n$ large enough, which leads to the expected $\limsup \mathbb E [n^{-\gamma}A^{(i)}_n] \leq  \mathbb E[I^{(i)}]$.
$\hfill \square$

%**********************************************************************

\section{Mixing regime}
\label{SecMixing}

%**********************************************************************

In this section, we assume that the rate of type change is much larger than that of macroscopic jumps. We recall the notations $P_n$ and $p_n^{(i)}(m)$ introduced in Section \ref{SecDetails} for the type transition matrix and for the position transition probabilities of the chain $(X,J)$. We recall also that a \emph{$\mathsf Q$--matrix} on $\{1,\ldots,\kappa\}$ is a $\kappa\times\kappa$ matrix $Q$ such that the diagonal coefficients are nonpositive, the coefficients outside the diagonal are nonnegative and the sum of each line is $0$. These matrices serve as generators for continuous time Markov chains on $\{1,\ldots,\kappa\}$. A $\mathsf Q$--matrix is said to be \emph{irreducible} if the associated Markov chain is irreducible.

\medskip

\begin{mybox}
\noindent \textbf{Hypothesis $(\mathsf H_{\mathrm{mix}})$.} Assume that there exists $0\leq\beta<\gamma$ such that:
\begin{enumerate}
\item[(i)] There exist finite measures $(\mu^{(i)},i \in \{1,\ldots,\kappa\})$ on $[0,1],$ at least one of which is nontrivial, such that, for all continuous functions $f:[0,1]\rightarrow \mathbb R$,
\[
n^{\gamma} \sum_{m=0}^n f\left(\frac{m}{n}\right) \left(1-\frac{m}{n}\right)p_{n}^{(i)}(m) \underset{n\rightarrow \infty}\longrightarrow \int_{[0,1]} f(x)\mu^{(i)}(\mathrm dx). 
\]
\item[(ii)] Moreover, there exists an irreducible $\mathsf Q$--matrix $Q=(q_{i,j})_{i,j\in\{1,\ldots,\kappa\}}$ such that
\[
n^{\beta}(P_n-I) \underset{n\rightarrow \infty}\longrightarrow Q.
\]
\end{enumerate}
\end{mybox}

\noindent In this regime, we will observe that the types asymptotically ``mix". Precisely, at the $n^{\gamma}$ scale, the type of the chain changes instantly since $\beta<\gamma$, and the chain will act in the limit as if its type was a weighted combination of all types, given by the invariant measure of the matrix $Q$. We let $\pi=(\pi_i,i\in\{1,\ldots,\kappa\})$ be the unique invariant probability measure for $Q$, which exists by irreducibility. We let also for $i\in\{1,\ldots,\kappa\}$, $\psi_i$ be the Laplace exponent corresponding to the measure $\mu^{(i)}$, that is
\[
\psi_i(\lambda)=\mu^{(i)}(\{0\})+\lambda\mu^{(i)}(\{1\})+\int_{(0,1)} (1-x^{\lambda})\frac{\mu^{(i)}(\mathrm d x)}{1-x},
\]
and $\psi$ the mixed Laplace exponent:
\begin{equation}\label{mix:defpsi}
\psi(\lambda)=\sum_{i=1}^{\kappa} \pi_i\psi_i(\lambda).
\end{equation}
We define

\[
Y_n^{(i)}(t):=\frac{X_n^{(i)}(\lfloor n^{\gamma} t\rfloor)}{n}, \quad \quad 
Z_n^{(i)}(t):=Y_n^{(i)}(\tau_n^{(i)}(t)), \quad \quad \text{and} \quad K_n^{(i)}(t):=J_n^{(i)}(\lfloor n^{\gamma}\tau_n^{(i)}(t)\rfloor) \quad t \geq 0,
\]
where 
$\tau_n^{(i)}(t)=\inf\left\{u:\int_0^u (Y_n^{(i)}(r))^{-\gamma} \mathrm dr>t\right\}.$

\medskip

\begin{thm} 
\label{ThMixing}
Under assumption $(\mathsf H_{\mathrm{mix}})$,
\[
\left(\Bigg(\frac{X^{(i)}_n(\lfloor n^{\gamma}t \rfloor)}{n}, Z_n^{(i)}(t) \Bigg), t\geq 0\right) \ \overset{\mathrm{(d)}}{\underset{n \rightarrow \infty} \longrightarrow} \ \left(X,Z\right),
\]
where $\left(-\log(Z)\right)$ is a subordinator with Laplace exponent $\psi=\sum_{j=1}^{\kappa}\pi_j\psi_j$, $X=Z(\rho(\cdot))$ and
\[
\rho(t)=\inf \left\{u:\int_0^u \left(Z(r)\right)^{\gamma} \mathrm dr>t\right\}.
\]
The topology is the product topology on $\mathcal D\left([0,\infty),[0,\infty)\right)^2$.
\end{thm}

\medskip

Recall that $A_n^{(i)}$ is the absorption time of $X_n^{(i)}$ and let  $I$ be the absorption time of $X$ (which has positive moments of all orders since the Laplace exponent $\psi$ is not trivial). 

\begin{thm}\label{ThMixingjoint}
Assume, in addition to $(\mathsf H_{\mathrm{mix}})$, that the measures $(\mu^{(i)},i \in \{1,\ldots,\kappa\})$ are all nontrivial. In this case, jointly with the previous convergence, we have
\[ \frac{A_n^{(i)}}{n^{\gamma}}\overset{\mathrm{(d)}}{\underset{n \rightarrow \infty} \longrightarrow}I,\]
and for all $a \geq 0$,
\[\mathbb E\left[\left(\frac{A_n^{(i)}}{n^{\gamma}}\right)^a \right] {\underset{n \rightarrow \infty} \longrightarrow} \mathbb E \left[ I^a\right].\]
\end{thm}

\medskip

\noindent \textbf{Remark: possible extensions. 1.} As for Theorem \ref{ThCriticalJoint}, we believe that the convergences stated in \ref{ThMixingjoint} are still true without the additional assumption that the $(\mu^{(i)},i \in \{1,\ldots,\kappa\})$ are all nontrivial. However, this assumption leads to a fairly simple proof and therefore we will keep it in the following.

\noindent \textbf{2.} It is probably possible to write versions of Theorems \ref{ThMixing} and \ref{ThMixingjoint} for matrices $Q$ which are not irreducible,  and in fact one could also have intermediate results between the mixing and critical regimes, where we have several groups of types, inside of which the rate of type change is of order $n^{-\beta}$, but the rate of changing group is of order $n^{-\gamma}$. We will not consider such generalizations, the current subject matter already being quite complex.

\bigskip

The proofs of Theorem \ref{ThMixing} and Theorem \ref{ThMixingjoint} are partly inspired by the ones of \cite{HM11} in the monotype setting. The differences come from the multiplicity of types and their mixing, which significantly complicates the proofs.
We start below by implementing a series of preliminaries in Section \ref{secPrelim}, and then turn to the proofs of Theorem \ref{ThMixing} and Theorem \ref{ThMixingjoint} in Section \ref{secTh1} and Section \ref{secTh2} respectively. The proof of a key point on the mixing of types, Proposition \ref{prop:melange}, stated in Section  \ref{sec:presmixing}, is postponed to Section \ref{sec:PropMelange}.

\subsection{Preliminaries}
\label{secPrelim}

We set up in this section key steps to prove the statements of Theorem \ref{ThMixing} and Theorem \ref{ThMixingjoint}. In all the statements below, it is implicit that we work under $(\mathsf H_{\mathrm{mix}})$ (although this hypothesis is not necessary at every step).

\subsubsection{Generating functions and bounds}
We list here a few simple results on some generating functions which we will need later on. For $n\in\N$ and $i\in\{1,\ldots,\kappa\},$ we let $G_n^{(i)}$ be the function defined by for $\lambda>0$ by 
\begin{equation}\label{mix:defG}
G_n^{(i)}(\lambda)=\mathbb{E}\left[\left(\frac{X_n^{(i)}(1)}{n}\right)^{\lambda}\right].
\end{equation}
By assumption, we know that, for all $\lambda>0$,
\begin{equation}\label{mix:convG}
n^{\gamma}\big(1-G_n^{(i)}(\lambda)\big) \to \psi_i(\lambda).
\end{equation}
We then have the existence of a finite constant $c(\lambda)$ such that, for all $n\in\N$ and $i\in\{1,\ldots,\kappa\}$,
\begin{equation}\label{mix:c1}
1-G_n^{(i)}(\lambda)\leq n^{-\gamma}c(\lambda).
\end{equation}

\subsubsection{Tightness and different time scales}
\label{sec:tightness}

\begin{prop}
\label{prop:mix:tight}
The sequence of processes $(Y_n^{(i)},n\in\N)$ is tight in  $\mathcal D\left([0,\infty),[0,\infty)\right).$
\end{prop}

\begin{proof}
Our proof is essentially the same as that of Lemma 1 in \cite{HM11}, adapted to the multi--type case. We use Aldous' tightness criterion for the Skorokhod topology. Namely, since $Y_n^{(i)}$ is bounded, we need to prove

\begin{equation}\label{Aldous}
\underset{\theta_0\rightarrow 0}\lim\underset{n\to\infty}\limsup \underset{T\in \mathcal{J}(\G_n),T\leq t}\sup \;\underset{0\leq\theta\leq\theta_0}{\sup} \pr\big[|Y_n^{(i)}(T)-Y_n^{(i)}(T+\theta)|>\veps\big]=0 
\end{equation}
for all $t>0$ and $\veps>0$, where $\mathcal{J}(\G_n)$ is the set of stopping times for the filtration $\G_n$ which is the natural filtration of the process $(Y_n^{(i)})$.

To do this, we make use of the martingale $M_n=(M_n(k),k\geq0)$ defined for $k\geq0$ by
\[M_n(k)=\left(\frac{X_n^{(i)}(k)}{n}\right)^{\lambda} + \sum_{l=0}^{k-1} \left(\frac{X_n^{(i)}(l)}{n}\right)^{\lambda}\left(1-G^{(J_n^{(i)}(l))}_{X_n^{(i)}(l)}(\lambda)\right),\]
where $\lambda$ is any real number greater than $1\vee \gamma$. While $M_n$ is not a martingale for the filtration $\mathcal{G}_n$, it is a martingale for the larger filtration $\F_n$, the natural filtration of the process $(Y_n^{(i)},J_n^{(i)}).$ 
Given that $Y_n^{(i)}$ is non--increasing and $\lambda\geq 1$, we have $|Y_n^{(i)}(T)-Y_n^{(i)}(T+\theta)|^{\lambda}\leq (Y_n^{(i)}(T))^{\lambda}-(Y_n^{(i)}(T+\theta))^{\lambda}$ for $\theta\geq 0$ and $T$ a bounded stopping time. Using the optional stopping theorem and the fact that $\lambda\geq\gamma$, we obtain
\begin{align*}
\E\big[\big(Y_n^{(i)}(T)\big)^{\lambda}-\big(Y_n^{(i)}(T+\theta)\big)^{\lambda}\big]&=n^{-\lambda}\E\left[\sum_{l=\lfloor n^{\gamma} T\rfloor}^{\lfloor n^{\gamma}(T+\theta)\rfloor -1} \big(X_n^{(i)}(l)\big)^{\lambda}\big(1-G_{X_n^{(i)}(l)}^{(J_n^{(i)}(l))}(\lambda)\big)\right] \\
        &\leq c(\lambda)n^{-\lambda}\E\left[\sum_{l=\lfloor n^{\gamma} T\rfloor}^{\lfloor n^{\gamma}(T+\theta)\rfloor -1} \big(X_n^{(i)}(l)\big)^{\lambda-\gamma}\right] \\ 
        &\leq c(\lambda)n^{-\lambda}\E\left[\sum_{l=\lfloor n^{\gamma} T\rfloor}^{\lfloor n^{\gamma}(T+\theta)\rfloor -1} n^{\lambda-\gamma}\right] \\
        &\leq c(\lambda)(\theta+n^{-\gamma}),        
\end{align*}
and (\ref{Aldous}) is then a consequence of the Markov inequality.
\end{proof}

One may expect from this tightness  that the natural scale of time to see macroscopic changes in $X_n^{(i)}$ is the scale $n^{\gamma}$. The following lemma reinforces this, and shows that the scale of time to see a change of $J_n^{(i)}$ is much smaller.

\begin{lem}
\label{lemma:mix:scale}
\begin{itemize}
\item[\emph{(i)}] Let $(T_n,n\in\N)$ be any sequence of random times such that $n^{-\gamma}T_n$ converges in probability to $0$. Then $n^{-1}X_n^{(i)}(T_n)$ converges in probability to $1$.
\item[\emph{(ii)}] For $n\in\N$ and $\veps>0$, let
\[T_{n,\veps}=\inf \{k\geq 0: \, X_n^{(i)}(k)\leq n\veps\}.\]
Then, for any $\alpha<\gamma$, $n^{-\alpha}T_{n,\veps}$ tends to infinity in probability, in the sense that, for all $u>0$, $\pr(T_{n,\veps}> un^{\alpha})$ converges to $1$.
\item[\emph{(iii)}] For $n\in\N$, let $T_n=\inf\{k\in\N, J_n^{(i)}(k)\neq i\}$. Then $n^{-\gamma}T_n$ converges in probability to $0$.
\end{itemize}
\end{lem}
\begin{proof} For point (i), let $\veps>0$, and take $\eta>0$ and $n\in\N$ such that $\pr(T_n>\eta n^{\gamma})\leq \veps$. We then have, for all $\rho>0$, 
\[\pr\Big(X_n^{(i)}(T_n)\leq (1-\rho)n \Big)\leq \veps + \pr\Big(X_n^{(i)}(\lfloor\eta n^{\gamma}\rfloor)\leq (1-\rho)n\Big).\]
By Proposition \ref{prop:mix:tight}\footnote{Specifically, there is a compact set $K$ which contains $Y^{(i)}_n, \forall n$ with probability greater than $1-\veps$, and using Theorem 12.3 from \cite{Bill99}, $f(\eta)-f(0)$ converges to $0$ as $\eta$ tends to $0$, uniformly in $f\in K$.},  if $\eta$ is small enough, then this will be smaller than $2\veps$ for $n$ large enough.

\medskip

\noindent Point (ii) is a consequence of point (i), since $\pr(T_{n,\veps} \leq un^{\alpha})=\pr(X_n(\lfloor un^{\alpha})\rfloor\leq n\veps)$ tends to $0$.

\medskip

\noindent For point (iii), let $\alpha\in (\beta,\gamma)$, and choose any $\veps \in (0,1)$. Write
\[\pr(T_n> n^{\alpha})=\pr(T_n> n^{\alpha},T_{n,\veps}> n^{\alpha})+\pr(T_n> n^{\alpha},T_{n,\veps}\leq n^{\alpha}).\]
Noticing that the second term tends to $0$ by (ii), showing that the first also does will be enough to prove (iii). By $(\mathsf H_{\mathrm{mix}})$, we know that there exists some constant $C>0$ such that, for all $k$, \[\pr\big(J_n^{(i)}(k+1)\neq i \mid X^{(i)}_n(k)> n\veps, J^{(i)}_n(k)=i\big)\geq Cn^{-\beta}.\] 
One then deduces by induction that 
\[\pr(T_n>k, T_{n,\veps}\geq k)\leq (1-Cn^{-\beta})^k,\] 
and thus 
\[\pr(T_n>n^{\alpha},T_{n,\veps}> n^{\alpha})\leq (1-Cn^{-\beta})^{n^\alpha}, \]
which tends to $0$ because $\alpha>\beta$. The proof is then ended since $\gamma>\alpha$. Note that this argument  in fact shows that  $n^{-\alpha}T_n$ tends in probability to $0$ for all $\alpha>\beta$, but we will not need this improvement.
\end{proof}

\textbf{From now on, to free up some notational space, we will also drop all references to the original type in the notation, and thus refer to the processes as $X_n,J_n$ and so on.}

\subsubsection{Moving to continuous time} 
\label{sec:continuoustime}

Inspired by \cite{BK14}, we introduce a transformation which embeds our processes in continuous time, making them easier to manipulate. We do this by considering a standard Poisson process $(\mathcal{N}(t),t\geq 0)$ which is independent of all the $X_n$ and $J_n$, and letting for all $n\in\N$ and $t\geq 0$, 
$$X^c_n(t)=X_n(\mathcal{N}(t)) \quad \text{ and } \quad J^c_n(t)=J_n(\mathcal{N}(t)).$$ 
The process $\big(\big(X^c_n(t),J^c_n(t)\big),t\geq0\big)$ is thus a Markov process with transition rates given by the $(p_{(n,i)}(m,j)$). The functional law of large numbers ensures us that limit results for $(X_n,J_n)$ are equivalent to the same for $(X^c_n,J^c_n)$. Specifically, we have the following:
\begin{lem}\label{lem:passagecontinu} 
\begin{itemize}
\item[\emph{(i)}] Let $f_n$ be the function which maps $t\geq 0$ to $n^{-\gamma}\mathcal{N}(n^{\gamma}t),$ and $g_n$ a generalised inverse defined this way:
\[g_n(t)=\inf\{s \geq 0:\; f_n(s)= n^{-\gamma}\lfloor n^{\gamma}t\rfloor\}.\]
Then both $f_n$ and $g_n$ converge a.s. uniformly on compact sets to the identity function.

\item[\emph{(ii)}]  For all integers $k$, there exists a constant $c_k$ such that $\mathbb E \left[(\mathcal N(t))^k\right] \leq c_k(t^k \vee t)$, for all $t \geq 0$.
\end{itemize}
\end{lem}

\begin{proof} For point (i) we use classical arguments: since the considered functions are monotone and the limit we are looking for is continuous, we only need to prove a.s. pointwise convergence for, say, rational $t$. First for $f_n$, $n^{-\gamma}\mathcal{N}(n^{\gamma}t)$ a.s. converges to $t$ by the law of large numbers. The same then becomes true for the inverse: if any subsequence of $(g_n(t),n\in\N)$ converges to $s \in [0,+\infty]$ then, given that $f_n(g_n(t))=n^{-\gamma}\lfloor n^{\gamma}t\rfloor$, and that $f_n$ is non--decreasing and converges uniformly on compacts to the identity function, we get that $s=t$. Point (ii) is a standard result on moments of the Poisson distribution.
\end{proof}

\medskip

Now, since we have $X_n(\lfloor n^{\gamma}t\rfloor)=X_n^c(n^{\gamma}g_n(t))$, Lemma \ref{lem:passagecontinu} implies that
\begin{itemize}
\item Theorem \ref{ThMixing} (and \ref{ThMixingjoint}) can be proven by showing that $X_n^c$ (and its absorption time, as well as its moments) has the wanted scaling limit.
\item Proposition \ref{prop:mix:tight} and Lemma \ref{lemma:mix:scale} also apply to $X_n^c$ and $J_n^c$, with obvious modifications. 
\end{itemize}
We adapt all the previous notation, defining
\[
Y_n^c(t):=\frac{X_n^c( n^{\gamma} t)}{n}, \quad
Z_n^c(t):=Y_n^c(\tau_n^c(t)), \quad \text{and} \quad K_n^c(t):=J_n^c(n^{\gamma}\tau_n^{c}(t)) \quad t \geq 0,
\]
where 
$\tau_n^{c}(t)=\inf\left\{u:\int_0^u (Y^c_n(r))^{-\gamma} \mathrm dr>t\right\}.$ \textbf{We now aim at proving Theorem \ref{ThMixing} and Theorem \ref{ThMixingjoint} for the continuous--time process $(X_n^c,J_n^c)$}.

\smallskip

The tightness from Proposition \ref{prop:mix:tight} implies that $(Y_n^{c})$ will converge to $X$ in distribution if every converging subsequence of $(Y_n^{c})$ has $X$ as limiting distribution. We consider such a converging subsequence, and using Skorokhod's embedding theorem, suppose that this subsequence converges \textit{almost surely} to a process $Y'$. {\bf We will  only work on this subsequence from now on}, omitting sometimes to mention it. By Lemma \ref{lampsko}, this implies in fact that the pair $(Y_n^{c},Z_n^{c})$ converges a.s. to $(Y',Z')$, where $Z'$ is the Lamperti transform of $Y'$: for $t\geq 0$, $Z'(t)=Y'(\tau(t))$ where $\tau(t)=\inf\{s\geq0, \int_{0}^s (Y'(r))^{-\gamma} \mathrm d r>t\}$. What we want to do is to show that $\left(-\log(Z')\right)$ is necessarily a subordinator with Laplace exponent $\psi$ defined in (\ref{mix:defpsi}), which will be done by proving that $(Z'(t)^{\lambda}e^{t\psi(\lambda)},t\geq0)$ is a martingale for all $\lambda>0$.

We introduce some more notation: if $f\in\D([0,\infty),[0,\infty))$ satisfies $f(0)=1$ and $\veps>0$, then we let
\[T_{\veps}(f)=\inf\big\{t\geq 0, f(t)\leq \veps\big\}.\]
Note then well that, by Skorokhod convergence, for all $\veps>0$ except a countable set, $T_{\veps}(Y^{c}_n)$ and $T_{\veps}(Z_n^{c})$ converge a.s. to $T_{\veps}(Y')$ and $T_{\veps}(Z')$, and the stopped processes  $(Y_n^{c}(t\wedge T_{\veps}(Y_n^{c})),t\geq 0)$ and \linebreak $(Z_n^{c}(t\wedge T_{\veps}(Z_n^{c})),t\geq 0)$ converge in the Skorokhod sense to $(Y'(t\wedge T_{\veps}(Y')),t\geq 0)$ and $(Z'(t\wedge T_{\veps}(Z')),t\geq 0).$ This is explained in the proof of Lemma 3 of \cite{HM11}. {\bf We will now only work with such $\veps$}.

\subsubsection{About the mixing of types}
\label{sec:presmixing}

The following proposition formalises how the types mix in $(X_n^c,J_n^c)$.

\begin{prop}
\label{prop:melange}
For all $i\in\{1,\ldots,\kappa\}$ and $\veps>0$, we have the following convergence in probability of measures:
\[\mathbbm{1}_{\{s\leq  T_{\veps}(Z^c_n)\}} \mathbbm{1}_{\{K^c_n(s)=i\}} \mathrm ds \overset{(\pr)}{\underset{n\to\infty}\longrightarrow} \mathbbm{1}_{\{s\leq  T_{\veps}(Z')\}}\pi_i \mathrm ds.\]
\end{prop}

The meaning of Proposition \ref{prop:melange} is that the types spread themselves out evenly, and that we have each type $i$ a proportion $\pi_i$ of the time. As the proof will show, we must stop at time $T_{\veps}(Z_n)$ in order to use $(\mathsf H_{\mathrm{mix}})$.

\medskip

\noindent \textbf{Remark.} Convergence in probability implicitly refers to the Prokhorov metric for measures. Some of its elementary properties are provided in Appendix \ref{sec:cvprobamesures}.

Since the proof of Proposition \ref{prop:melange} is very involved and contains most of the difficulty, we postpone it to  Section \ref{prop:melange}, and first use it to prove Theorems \ref{ThMixing} and \ref{ThMixingjoint}.

\subsubsection{Some martingales}

\begin{lem} Let $\lambda>0$ and $n\in\N$, and define a process $M_n^{(\lambda)}$ by
\[M_n^{(\lambda)}(t)=\Big(\frac{X_n^c(t)}{n}\Big)^{\lambda}\exp\Big(\int_{0}^t \left(1-G^{(J_n^c(s))}_{X_n^c(s)}(\lambda)\right) \,\mathrm d s\Big)\]
if $X_n^c(t)\neq 0$, while $M_n^{(\lambda)}(t)=0$ if $X_n^c(t)= 0.$
Then $M_n^{(\lambda)}$ is a martingale in the natural filtration of $(X_n^c,J_n^c)$. As a consequence, the time--changed process $\mathcal{M}_n^{(\lambda)}$ defined by

\begin{equation}\label{themartingale}
\mathcal{M}_n^{(\lambda)}(t)=M_n^{(\lambda)}(n^{\gamma}\tau_n^c(t))=\big(Z_n^c(t)\big)^{\lambda}\exp\Big(\int_{0}^{n^{\gamma}\tau_n^c(t)} \Big(1-G^{(J_n^c(s))}_{X_n^c(s)}(\lambda)\Big)\,\mathrm d s\Big)
\end{equation}
is also a martingale.
\end{lem}
\begin{proof}
The martingale property of $\mathcal{M}_n^{(\lambda)}$ is a direct consequence of that of $M_n^{(\lambda)}$ and that the stopping time $\tau_n^c(t)$ is smaller than $t$ for all $t\geq0.$ We thus focus on $M_n^{(\lambda)}.$
Notice first that, for $s\leq t$, if $X_n^c(s)>0$, then
\begin{equation}\label{Markovformartingale}
M_n^{(\lambda)}(t)=M_n^{(\lambda)}(s)\Big(\frac{X_n^c(t)}{X_n^c(s)}\Big)^{\lambda}\exp\Big(\int_{s}^{t} \Big(1-G^{(J_n^c(u))}_{X_n^c(u)}(\lambda) \Big)\,\mathrm d u\Big).
\end{equation}
By the Markov property, conditionally on the past up to time $s$, the last two terms form a copy of $\big(M_{X_n^c(s)}^{(\lambda)}\big)'(t-s),$ where $\big(M_{X_n^c(s)}^{(\lambda)}\big)'$ is an independent version of the same martingale when the process starts at $(X_n^c(s),J_n^c(s)).$ Thus we are reduced to showing that $\mathbb{E}[M_n^{(\lambda)}(t)]=1$ for all $t\geq 0$, and we will do this by showing that the right derivative of this function is $0$ at all points.
When $h\geq 0$ is small, we know that the probability of the Poisson process $\mathcal{N}$ having one jump in $[t,t+h]$ is of order $h$, while the probability of there being two jumps is of order $O(h^2)$. We thus have, using (\ref{Markovformartingale}) a second time, the following asymptotic expansion:
\begin{align*}
& \mathbb{E}\left[M_n^{(\lambda)}(t+h)\right]=(1-h)\mathbb{E}\left[M_n^{(\lambda)}(t)\e^{h\left(1-G^{(J_n^c(t))}_{X_n^c(t)}(\lambda)\right)}\right]   \\  &+ \mathbb{E}\left[M_n^{(\lambda)}(t)\Big(\frac{X_n(\mathcal{N}(t)+1)}{X_n^c(t)}\Big)^{\lambda}\int_0^h e^{-s}\exp\left( s\left(1-G_{X_n^c(t)}^{(J_n^c(t))}(\lambda)\right)+(h-s)\left(1-G_{X_n(\mathcal{N}(t)+1)}^{J_n(\mathcal{N}(t)+1)}(\lambda)\right)\right)\mathrm d s\right] \\ &+ O(h^2).
\end{align*}
Note that the term inside the second expectation is by convention $0$ if $X_n^c(t)=0$. Since we have $G_k^{(j)}(\lambda)\leq1$ for all $j$ and $k$, $M_n^{(\lambda)}(t)\leq \e^t$ and $X_n(\mathcal{N}(t)+1)\leq X_n^c(t)$, we can safely integrate the $O$ terms and take them out of expected values, and thus
\begin{align*}
\mathbb{E}\left[M_n^{(\lambda)}(t+h)\right]&=(1-h)\mathbb{E}\left[M_n^{(\lambda)}(t)\Big(1+h\big(1-G^{(J_n^c(t))}_{X_n^c(t)}(\lambda)\big)\Big)\right] \\
&+h\mathbb{E} \left[M_n^{(\lambda)}(t)\Big(\frac{X_n(\mathcal{N}(t)+1)}{X_n^c(t)}\Big)^{\lambda}\int_0^h (1+O(s))(1+O(s))\big(1+O(h-s)\big)\mathrm ds\right]+O(h^2) \\
&=\mathbb{E}\left[M_n^{(\lambda)}(t)\big(1-hG^{(J_n^c(t))}_{X_n^c(t)}(\lambda)\big)\right]   +h\mathbb{E}\left[M_n^{(\lambda)}(t)\Big(\frac{X_n(\mathcal{N}(t)+1)}{X_n^c(t)}\Big)^{\lambda}\right]+O(h^2).
\end{align*}
Since the conditional expectation of $\Big(\frac{X_n(\mathcal{N}(t)+1)}{X_n^c(t)}\Big)^{\lambda}$ given $\big(X_n(\mathcal{N}(t)),J_n(\mathcal{N}(t))\big)=\big(X_n^c(t),J_n^c(t)\big)$ is equal to $G^{(J_n^c(t))}_{X_n^c(t)}(\lambda)$ by definition, we end up with no term of order $h$, and a derivative equal to $0$ at $t$.
\end{proof}

\subsection{End of the proof of Theorem \ref{ThMixing}: scaling limit of the position marginal}
\label{secTh1}

From Section \ref{sec:continuoustime}, it is sufficient to prove a version of Theorem \ref{ThMixing} for the continuous--time process $(X_n^c,J_n^c)$. Moreover, relying on the tightness established in Section \ref{sec:tightness}, it was noticed, still in Section \ref{sec:continuoustime}, that such a version of Theorem \ref{ThMixing}  will be proved if for any possible limit $Y'$ of a subsequence of $(Y_n^c)$, the process $\big((Z'(t))^{\lambda}e^{t\psi(\lambda)},t\geq0\big)$ is a martingale, where $Z'$ denotes the $(-\gamma)$--Lamperti transform of $Y'$. It was also noticed that there is not loss of generality in assuming that the convergences are almost sure. To simplify the notation below, we let $(Y_n^c,Z_n^c)$ denote a subsequence that converges (almost surely) to $(Y',Z')$, with a slight abuse in the indices notation.

Our aim is therefore to show that the martingale $\mathcal{M}_n^{(\lambda)}$ introduced in (\ref{themartingale}) converges to the process $\big((Z'(t))^{\lambda}e^{t\psi(\lambda)},t\geq0\big)$ in a strong enough sense for the latter to also be a martingale. To do so, we first fix an $\varepsilon>0$ with the properties required at the end of Section \ref{sec:continuoustime},  and stop $\mathcal{M}_n^{(\lambda)}$ at time $T_{\veps}(Z^c_n)$, and show that the process $\big(\mathcal{M}_n^{(\lambda)}(t\wedge T_{\veps}(Z^c_n)), t\geq 0\big)$ converges in probability for the Skorokhod metric to $\big(Z'(t\wedge T_{\veps}(Z'))^{\lambda}\exp(\psi(\lambda)(t\wedge T_{\veps}(Z')),t\geq 0\big)$. Recalling the definition (\ref{themartingale}) of the martingale $\mathcal{M}_n^{(\lambda)}$ and that $(Z_n^c)$ converges a.s. in the Skorokod sense to $Z'$, it only remains to check that the term
\[\exp\Big(\int_0^{ n^{\gamma}(\tau^c_n(t\wedge T_{\veps}(Y^c_n)))}\big(1-G^{(J^c_n( s ))}_{X^c_n(s)}(\lambda)\big)\mathrm d s\Big)
\]
converges in probability uniformly on all compact sets to $\exp\big(\psi(\lambda)(t\wedge T_{\veps}(Z'))\big).$ By a variation of a classical argument (using subsequences to bring ourselves back to almost--sure convergence, see the proof of Lemma \ref{lem:interversion} for a similar reasoning), since these functions are nondecreasing and the limit is continuous, we only need to show pointwise convergence in probability.

Write, for $n\geq \veps^{-1}$,
\[\int_{0}^{n^{\gamma}(\tau^c_n(t\wedge T_{\veps}(Y^c_n)))}\big(1-G^{(J^c_n(s))}_{X^c_n(s)}(\lambda)\big)\mathrm d s
   = \int_0^{t\wedge T_{\veps} (Z^c_n)} \big(1-G^{(K^c_n(r))}_{nZ^c_n(r)}(\lambda)\big)\big(nZ^c_n(r)\big)^{\gamma}\mathrm d r.\]
We split the integrand according to the different types. For all $j\in\{1,\ldots,\kappa\}$, we have by (\ref{mix:convG})
\[\big(1-G^{(j)}_{nZ^c_n(r)}(\lambda)\big)(nZ^c_n(r))^{\gamma} \underset{n \rightarrow \infty}\to \psi_j(\lambda)\]
and this is uniform in $r$ as long as we stay before time $T_{\veps}(Z^c_n)$, since we then have $nZ^c_n(r)\geq n\veps$. This lets us write 

\begin{align*}
\int_0^{t\wedge T_{\veps} (Z^c_n)} &\big(1-G^{(j)}_{nZ^c_n(r)}(\lambda)\big)(nZ^c_n(r))^{\gamma}\mathbbm{1}_{\{K^c_n(r)=j\}}\mathrm d r=\int_0^{t\wedge T_{\veps} (Z^c_n)} \psi_j(\lambda)\mathbbm{1}_{\{K^c_n(r)=j\}}\mathrm d r \\ 
& + \int_0^{t\wedge T_{\veps} (Z_n)}\left(\Big(1-G^{(j)}_{nZ^c_n(r)}(\lambda)\Big)(nZ^c_n(r))^{\gamma}-\psi_j(\lambda)\right)\mathbbm{1}_{\{K^c_n(r)=j\}}\mathrm d r.
\end{align*}
The first term of the right--hand side converges in probability to $\psi_j(\lambda)\pi_j(t\wedge T_{\veps}(Z'))$ by Proposition \ref{prop:melange} and the second to $0$ by the aforementioned uniform convergence.

Uniform integrability arguments will then transfer the martingale property of  $\mathcal{M}_n^{(\lambda)}$ to \linebreak $(Z'(t)^{\lambda}\exp(\psi(\lambda)t),t\geq 0).$ Specifically, note first that, for $n\geq \veps^{-1}$, using (\ref{mix:c1}), we have
\begin{align*}
\mathcal M_n^{(\lambda)}(t\wedge T_{\veps}(Z^c_n))
   &\leq \exp\Big(\int_0^{ t\wedge T_{\veps}(Z_n^c)}(nZ^c_n(r))^{\gamma}\big(1-G^{(K_n^c(r))}_{nZ^c_n(r)}(\lambda)\big)\mathrm d r\Big)\\
   &\leq \e^{c(\lambda)t},
\end{align*}
and thus, for fixed $t$, $\big(\mathcal{M}_n^{(\lambda)}(t\wedge T_{\veps}(Z^c_n)),n\geq\veps^{-1}\big)$ is uniformly integrable, implying by \cite{EK}, Example 7, p.362, that the limit process $\big(Z'(t\wedge T_{\veps}(Z'))^{\lambda}\exp\big(\psi(\lambda)(t\wedge T_{\veps}(Z'))\big),t\geq 0\big)$ is a martingale. Similarly, for fixed $t$ and for all $\veps>0$, $Z'(t\wedge T_{\veps}(Z'))^{\lambda}\exp\big(\psi(\lambda)(t\wedge T_{\veps}(Z'))\big)\leq \e^{\psi(\lambda)t}$, and we therefore have uniform integrability as $\veps$ tends to $0$, preserving the martingale property for the limit. \qed

\subsection{Proof of Theorem \ref{ThMixingjoint}: scaling limit of the absorption time}
\label{secTh2}

We assume here that for all $i\in\{1,\ldots,\kappa\}$, the measure $\mu^{(i)}$ of hypothesis $(\mathsf{H_{\mathrm{mix}}})$ is nontrivial. As a consequence, for all $\lambda>0$, there exists $c'(\lambda)>0$ such that, for $n$ which is not an absorbing state,
\begin{equation}\label{mix:c3}
1-G_n^{(i)}(\lambda)\geq n^{-\gamma}c'(\lambda), \quad \forall i \in \{1,\ldots,\kappa\}.
\end{equation}
Also, as in Section \ref{SecCritique}, we now make the extra assumption that {\bf the only absorbing state for $X$ is $0$.} Just as in that section, proving Theorem \ref{ThMixingjoint} under this assumption is enough to deduce the general case. Thus inequality (\ref{mix:c3}) becomes true for all $n\in\N$.

Our goal is to show that jointly with the convergence of $(Y_n^c,Z_n^c)$ towards $(X,Z)$ proved in the previous section,  the absorption time $A_n^c$ of $Y_n^c$ (or $X_n^c$) at 0 satisfies $A_n^c/n^{\gamma} \rightarrow I$ in distribution, and that there is also convergence of all positive moments. We recall that $I$ denotes the absorption time at 0 of the process $X$. We start with a preliminary lemma.

\begin{lem} 
\label{lem:expo}
For all $n\in\N,$ $\lambda>0,$ and $t\geq0$, we have
\[\mathbb{E}\big[(Z^c_n(t))^{\lambda}\big]\leq \e^{-c'(\lambda)t},\]
where $c'(\lambda)$ was introduced in (\ref{mix:c3}).
\end{lem}

\begin{proof}
Recall that when $Z^c_n(t)>0$
\[
(Z^c_n(t))^{\lambda}=\mathcal{M}_n^{(\lambda)}(t)  \exp\Big(\int_0^{ n^{\gamma}\tau^c_n(t)}\big(G^{(J^c_n(s ))}_{X^c_n( s )}(\lambda)-1\big)\mathrm d s\Big), \]
where $\big(\mathcal{M}_n^{(\lambda)}(t),t\geq0\big)$ is a martingale. Using (\ref{mix:c3}), we have, still when $Z^c_n(t)>0$,
\begin{align*}
\int_0^{n^{\gamma}\tau^c_n(t)}\big(G^{(J^c_n(s ))}_{X^c_n( s )}(\lambda)-1\big)\mathrm d s &\leq -c'(\lambda)\int_0^{ n^{\gamma}\tau^c_n(t)} X^c_n(s)^{-\gamma} \mathrm d s\\
&\leq -c'(\lambda)\int_0^t \mathrm d s.
\end{align*}
Hence $(Z^c_n(t))^{\lambda} \leq \mathcal{M}_n^{(\lambda)}(t) \exp(-c'(\lambda)t)$ in any case. We can then take the expectation.
\end{proof}

\medskip

The rest of the proof of Theorem \ref{ThMixingjoint} goes as the one of Theorem 2 in \cite{HM11}, so we only sketch it: since the only absorbing state is $0$, we have 
\begin{equation}
\label{eq:AnInt}
\frac{A^c_n}{n^{\gamma}}=\int_0^{\infty}(Z^c_n(r))^{\gamma} \mathrm d r,
\end{equation} 
and thus the expectations of $n^{-\gamma}A^c_n$ are uniformly bounded (using Lemma \ref{lem:expo}), and thus $(n^{-\gamma}A^c_n,n\in\N)$ is tight. Up to using the Skorokhod representation theorem and extracting, we can assume that the triplet $(Y^c_n,Z^c_n,n^{-\gamma}A^c_n)$ converges a.s. to $(X,Z,I'),$ and we only need to check that $I'=I$, where $I=\int_0^{\infty}Z(t)^{\gamma} \mathrm d t$ is the extinction time of $X$. The Skorokhod convergence first shows that $Y^c_n(n^{-\gamma}A^c_n)$, which is equal to $0$, converges to $X(I')$, implying $I'\geq I$. On the other hand, dominated convergence and Fatou's lemma give us
\[ \mathbb{E}[I']=\mathbb{E}\Big[\lim \int_0^{\infty}(Z^c_n(r))^{\gamma} \mathrm d r\Big]\leq \liminf \mathbb{E}\Big[ \int_0^{\infty}(Z^c_n(r))^{\gamma} \mathrm d r\Big]\underset{\text{by Lemma \ref{lem:expo}}}=\mathbb{E}\Big[\int_0^{\infty}Z(r)^{\gamma} \mathrm d r\Big] = \mathbb{E}[I],\] 
hence $n^{-\gamma}A^c_n$ converges in distribution to $I$. To get the convergence of all positive moments, it remains to show that $\sup_n \mathbb E[(n^{-\gamma}A^c_n) ^a]<\infty$ for all $a \geq 0$ which is easy to see by using (\ref{eq:AnInt}) together with Hölder's inequality and Lemma \ref{lem:expo}.
\qed

\subsection{Proof of Proposition \ref{prop:melange}: mixing of types}
\label{sec:PropMelange}

It remains to prove Proposition \ref{prop:melange}. We recall that it is assumed that $(Y_n^c,Z_n^c)$ converges almost surely to $(Y',Z')$. The main idea will be to couple the bivariate chain $(X_n^c,J_n^c)$ with a standard $\{1,\ldots,\kappa\}$--valued continuous--time Markov chain with $\mathsf Q$--matrix $Q$, so that, after an appropriate time--change, $J_n^c$ behaves asymptotically as this standard Markov chain.
In order to do so, we first notice in Section \ref{sec:addmixing} that we can do additional assumptions on the model, without loss of generality. Section \ref{sec:LampertiBeta} then introduces Lamperti transform of $(X_n^c,J_n^c)$ in the $n^{\beta}$--time scale. The idea is that in this scale, the type--component resembles asymptotically to the above mentioned  Markov chain with $\mathsf Q$--matrix $Q$. This approximation is studied in Section \ref{sec:coupling} and the end of the proof of Proposition \ref{prop:melange} is given in Section \ref{sec:endmixing}.

\subsubsection{Foreword: a few changes}\label{sec:mix:changements}
\label{sec:addmixing}

As in the critical case, we change the transition probabilities $(p_{n,i}(m,j))$ slightly in a way which does not change the scaling limit. Here, the aim is to make some waiting times we will consider in the following sections, and their moments, finite. First, as already noticed several times, we can assume with no loss of generality that the only absorbing state for the position component is $0$. Then we define, for $m\leq n$ and $i,j\in\{1,\ldots,\kappa\}$, $p'_{n,i}(m,j)$ this way:
\begin{itemize}
\item[$\circ$] $p'_{n,i}(m,j)=p_{n,i}(m,j)$ if $m>2$
\item[$\circ$] $p'_{n,i}(2,j)=p_{n,i}(0,j)+p_{n,i}(1,j)+p_{n,i}(2,j)$
\item[$\circ$] $p'_{2,i}(1,1)=1$
\item[$\circ$] $p'_{1,i}(1,i+1)=1$ for $i\leq\kappa-1$
\item[$\circ$] $p'_{1,\kappa}(0,1)=1$
\item[$\circ$] $p'_{0,1}(0,1)=1$
\end{itemize}
Note that $(p'_{n,i}(m,j))$ then also satisfies $(\mathsf H_{\mathrm{mix}})$ and that proving Proposition \ref{prop:melange} for a Markov chain with transitions $(p'_{n,i}(m,j))$ will also prove it for the general case. As such we will now assume that the $(p_{n,i}(m,j))$ have been replaced by the $(p'_{n,i}(m,j)).$  Hence the following consequences: 

\begin{lem}\label{lem:esperancefinie}
For all $n\geq 2$, there is always at least one change of type before $X_n$ reaches $0$. Moreover this absorption time at 0, denoted by $A_n^c$, has finite positive moments of all orders:
\[\E[(A^c_n)^a]<\infty, \quad \text{ for all }a \geq 0 \text{ and all }n \in \mathbb N.\]
\end{lem}

\begin{proof} The first assertion is obvious by definition of $(p'_{n,i}(m,j)).$ Next, $(X_n^c,J_n^c)$ is a continuous time Markov chain on a finite state space with unique absorbing point $(0,1)$. If we add a small transition rate from $(0,1)$ to all the other states, then it becomes irreducible, at which point standard results imply that the time taken to go from one state to another has some finite exponential moments, and so in particular the time to reach $(0,1)$ has finite $a$-th moment for all $a\geq 0$. Note that the finiteness of these moments was already checked in the proof of Theorem \ref{ThMixingjoint}, under the extra condition that the measures $\mu^{(i)}$ of hypothesis ($\mathsf H_{\mathrm{mix}}$) are all nontrivial.
\end{proof}

{\bf In the following we assume that the conclusions of this lemma are valid, with no loss of generality.}

\subsubsection{Preparation: using the $n^{\beta}$ timescale}
\label{sec:LampertiBeta}

\noindent \textbf{The $n^{\beta}$ scale}. In order to prove Proposition \ref{prop:melange}, we will use another Lamperti--type time--change, this time using the index $\beta$ and the time scale of $n^{\beta}$, which are more appropriate for the study of the types. We let, for $n\in\N$ and $t\geq 0$,
\[Y^{(\beta)}_n(t)=\frac{X^c_n(n^{\beta}t)}{n} \quad \text{and} \quad
\tau_n^{(\beta)}(t) = \inf \bigg\{u, \int_0^u (Y^{(\beta)}_n(r))^{-\beta}\mathrm{d}r >t\bigg\}.
\]
In particular, we have $\mathrm d \tau_n^{(\beta)}(t)=\big(Y^{(\beta)}_n\big(\tau_n^{(\beta)}(t)\big)\big)^{\beta}\mathrm d t$. We then let
\[Z_n^{(\beta)}(t)=Y^{(\beta)}_n(\tau_n^{(\beta)}(t)) \quad \text{and} \quad K^{(\beta)}_n(t)=J^c_n\big(n^{\beta} \tau_n^{(\beta)}(t)\big).\]
Note that, once again by Lemma \ref{lampsko}, the process $\big(Y_n^{(\beta)}(\tau_n^{(\beta)}(n^{\gamma-\beta}t)),t\geq 0\big)$ then converges to the process $\big(Y'(\tau^{(\beta)}(t)),t\geq 0)$, where 
	\[\tau^{(\beta)}(t) = \inf \bigg\{u, \int_0^u (Y'(t))^{-\beta}\mathrm{d}r >t\bigg\},
\]
and the maps $t\to n^{\beta-\gamma}\tau_n^{(\beta)}(n^{\gamma-\beta}t)$ converge uniformly on compact sets to $\tau^{(\beta)}$.
In particular, letting $$S_{n,\veps}=\inf\{t\geq 0:\,Y_n^{(\beta)}(\tau_n^{(\beta)}(n^{\gamma-\beta}t))\leq \veps\},$$ then $S_{n,\veps}$ converges to $S_{\veps}=(\tau^{(\beta)})^{-1}(T_{\veps}(Y'))$. All these convergences are almost sure.

We will later need the following observation. Let $T_0(Z_n^{(\beta})$ denote the absorption time at 0 of $Z_n^{(\beta)}$. For $t<T_0(Z^{(\beta)}_n)$, we have $n^{\beta}\tau_n^{(\beta)}(t)\geq t$, i.e. the time--change speeds time up. Thus $T_0(Z^{(\beta)}_n)\leq A^c_n$, implying by Lemma \ref{lem:esperancefinie} that, for all $a\geq 0$ and all $n \in \mathbb N$,
\begin{equation}\label{eq:momentsfinis}
\E\Big[\big(T_0(Z^{(\beta)}_n\big)^a\Big]<\infty
\end{equation}

\medskip

\noindent \textbf{Mixing in the $n^{\beta}$ scale}. By the upcoming Lemma \ref{lem:mixchangescale}, proving Proposition \ref{prop:melange} can be done by instead proving that the types in $K_n^{(\beta)}$ mix after a time of order $n^{\gamma-\beta}$:
\begin{equation}\label{eq:melangebeta}
\mathbbm{1}_{\{s\leq  S_{n,\veps}\}}\mathbbm{1}_{\{K^{(\beta)}_n(n^{\gamma-\beta}s)=i\}}\mathrm d s \overset{(\pr)}{\underset{n \rightarrow \infty}\longrightarrow} \pi_i \mathbbm{1}_{\{s\leq  S_{\veps}\}} \mathrm{d} s.
\end{equation}

\begin{lem}\label{lem:mixchangescale} For all $n\in\N$, let $a_n$ and $b_n$ be positive random variables, $f_n$ be a random c\`adl\`ag function from $[0,a_n]$ to $\{0,1\}$ (extended to be constantly $0$ after $a_n$), and $F_n$ a random increasing bijection from $[0,a_n]$ to $[0,b_n]$ (extended to be constantly $b_n$ after $a_n$). We call $F_n'$ the right--derivative of $F_n$, which we assume to exist everywhere and be c\`adl\`ag. Assume that, as $n$ tends to infinity:

\begin{itemize}
\item[\emph{(a)}] $a_n$ converges a.s. to $a>0$, $b_n$ converges a.s. to $b>0$.
\item[\emph{(b)}]  $F_n$ converges uniformly a.s. to a continuous function $F$ of which the right--derivative $F'$ exists everywhere and is c\`adl\`ag, and $F'_n$ converges in the Skorokhod sense to $F'$.
\item[\emph{(c)}]  the measure $\mathbbm{1}_{\{x\leq a_n\}} \mathbbm{1}_{\{f_n(x)=1\}}\mathrm d x$ converges weakly in probability to $\lambda \mathbbm{1}_{\{x\leq a\}}\mathrm d x$, for some $\lambda\geq0$.
\end{itemize}

Then we also have the following weak convergence of measures in probability: 

\begin{equation}\label{eq:cvlemme}
\mathbbm{1}_{\{x\leq b_n\}}\mathbbm{1}_{\{f_n(F_n^{-1}(x))=1\}}\mathrm dx \overset{(\pr)}{\underset{n \rightarrow \infty}\longrightarrow}  \lambda \mathbbm{1}_{\{x\leq b\}}\mathrm d x.
\end{equation}
\end{lem}
\begin{proof}
Our first step is showing that, {\textit{ if the convergence of (c) is almost--sure, then (\ref{eq:cvlemme}) is also in fact an a.s. convergence}}. In this case we can drop the probabilistic notation and assume everything is deterministic. Let $g$ be any continuous and bounded function on $\R_+$, we have
\begin{align*}
\int_{0}^{b_n} g(x) &\mathbbm{1}_{\{f_n(F_n^{-1}(x))=1\}}\mathrm dx = \int_0^{a_n} g(F_n(x))\mathbbm{1}_{\{f_n(x)=1\}}F_n'(x)\mathrm d x \\
                 &=  \int_0^{a_n} \Big(g(F_n(x))F_n'(x)-g(F(x))F'(x)\Big)\mathbbm{1}_{\{f_n(x)=1\}}\mathrm d x +\int_0^{a_n}g(F(x))F'(x)\mathbbm{1}_{\{f_n(x)=1\}}\mathrm d x.
\end{align*}
The first term tends to $0$ because the Skorokhod convergence of $g(F_n(x))F_n'(x)$ to $g(F(x))F'(x)$ implies $L_1$ convergence, see Lemma \ref{lem:skotoL}. The second term converges to $\lambda \int_0^{a} g(F(x))F'(x)\mathrm d x=\lambda \int_{0}^{b} g(x) \mathrm d x$. We can use the convergence (c) despite $F'$ being c\`adl\`ag and not necessarily continuous, because we are only using absolutely continuous measures. 

For the general case, we use Lemma \ref{lem:extraction}. Thus we take a subsequence of $\mathbbm{1}_{\{x\leq b_n\}}\mathbbm{1}_{\{f_n(F_n^{-1}(x))=1\}}\mathrm dx$, and we look to extract a sub--subsequence which converges a.s. to $\lambda \mathbbm{1}_{\{x\leq b\}}\mathrm d x$. This is immediate: we just extract a subsequence such that (c) is a.s., and we are then back to the deterministic case, ending the proof.
\end{proof}

To be precise, Proposition \ref{prop:melange} follows from (\ref{eq:melangebeta}) and Lemma \ref{lem:mixchangescale} by taking $a_n=S_{n,\veps}$, $a=S_{\veps}$, $f_n(t)=\mathbbm{1}_{\{K^{(\beta)}_n(n^{\gamma-\beta}t)=i\}}$, $F_n(t)=(\tau_n^c)^{-1}\big(n^{\beta-\gamma}(\tau_n^{(\beta)})(n^{\gamma-\beta}t)\big)$ and $F=\tau^{-1}\circ\tau^{(\beta)}.$ Note that  $(\tau_n^c)^{-1}$ converges uniformly to $\tau^{-1}$ on $[0,T_{\veps}(Y')]$, by a similar argument to the proof of the uniform convergence of $g_n$ in Lemma \ref{lem:passagecontinu}.

Finally, in order to prove (\ref{eq:melangebeta}), we can use Lemma \ref{lem:interversion} and restrict ourselves to showing the convergence of the masses assigned to intervals of the form $[0,t]$ with $t>0$. Thus we want to prove this convergence in probability:
\[
\int_{0}^{t\wedge S_{n,\veps}} \mathbbm{1}_{\{K^{(\beta)}_n(n^{\gamma-\beta}s)=i\}}\mathrm d s \overset{(\pr)}{\underset{n \rightarrow \infty}\longrightarrow}  (t\wedge S_{\veps}) \pi_i , 
\]
which can then be written in a more concise way by including all the types:
\begin{equation}\label{simplemelange}
\int_{0}^{t\wedge S_{n,\veps}} \delta_{K^{(\beta)}_n(n^{\gamma-\beta}s)}\mathrm d s \overset{(\pr)}{\underset{n \rightarrow \infty}\longrightarrow}   (t\wedge S_{\veps}) \pi. 
\end{equation}
Our aim is now to prove (\ref{simplemelange}).

\subsubsection{A special coupling}
\label{sec:coupling}

Recall that, conditionally on $X_n^c(t)=k$, the infinitesimal jump rates of $J_n^c$ just after time $t\geq 0$ are given by the matrix $P_k$. Letting $Q_k=P_k-I$ be the corresponding $\mathsf Q$--matrix, we have by assumption 
\[Q_k=k^{-\beta}Q + o(k^{-\beta}).\]
Consider what happens when we use $\tau^{(\beta)}_n$. Given that $\mathrm d \tau_n^{(\beta)}(t)=(Z^{(\beta)}_n(t))^{\beta}\mathrm d t$, we have that, conditionally on $X_n(n^{\beta}\tau^{(\beta)}_n(t))=k$, the jump rates of $K_n^{(\beta)}$ are given by the $\mathsf Q$--matrix $k^{\beta}Q_k$, which is close to $Q$. By using a coupling argument, we will show that $K_n^{(\beta)}$ is close enough to a continuous time Markov chain with $\mathsf Q$--matrix $Q$, and equation (\ref{simplemelange}) will follow from the ergodic theorem. Specifically, let $(L(t),t\geq 0)$ be a Markov chain in continuous time with $\mathsf Q$--matrix $Q$, the following almost--sure limit is classical:
\[\frac{1}{t}\int_{0}^{t} \delta_{L(s)} \mathrm d s \underset{t\to\infty}\longrightarrow \pi\]
and it follows that, since $S_{n,\veps}$ converges a.s. to $S_{\veps}$,

\begin{equation}\label{eq:limiteL}\int_0^{t\wedge S_{n,\veps}} \delta_{L(n^{\gamma-\beta}s)} \mathrm d s \underset{n\to\infty}{\overset{\mathrm{a.s.}}\longrightarrow} (t\wedge S_{\veps})\pi.
\end{equation}
We will now build a coupling of all the $\big(K_n^{(\beta)}(s),s\leq n^{\gamma-\beta}(t\wedge S_{n,\veps})\big)$ with $L$ such that
$\int_0^{t\wedge S_{n,\veps}}  \delta_{K^{(\beta)}_n(n^{\gamma-\beta}s)}\mathrm d s$ is close enough to $\int_0^{t\wedge S_{n,\veps}}  \delta_{L(n^{\gamma-\beta}s)}\mathrm d s.$

\bigskip

\noindent \textbf{Comparison of the first jumps.} For all $n\in\N$, let $$\eta_n=\underset{k\geq n}\sup |k^{\beta}Q_k -Q|,$$ where $|.|$ denotes the supremum norm of a matrix.

\begin{lem}\label{lem:closetoL}
Let $i\in\{1,\ldots,\kappa\}$ be the initial type and $\sigma_1(K_n^{(\beta)})$ denote the first jump time of $K_n^{(\beta)}$ \emph(with the convention that this time is infinite when there is no jump\emph). Then:
\begin{itemize}
\item[$\mathrm{(i)}$] $\sigma_1(K_n^{(\beta)})$ converges in distribution to an exponential time with parameter $|q_{i,i}|$, and there is convergence of all the positive moments, that is, for $a>0$,
\[\E\big[\big(\sigma_1(K_n^{(\beta)})\big)^a\big] \underset{n \rightarrow \infty}\longrightarrow \frac{\Gamma(a+1)}{|q_{i,i}|^a}, \]
where $\Gamma$ denotes the standard Gamma function.
\item[$\mathrm{(ii)}$] $K_n^{(\beta)}\big(\sigma_1(K_n^{(\beta)})\big)$ converges in distribution to the first jump of $L$, that is $\pr\big(K_n^{(\beta)}(\sigma_1(K_n^{(\beta)}))=j\big)$ converges to $\frac{q(i,j)}{|q(i,i)|}$ for all $j\neq i$.
\end{itemize}
\end{lem}

\begin{proof}
At the heart of both proofs lies the fact that, by (iii) and (i) of Lemma \ref{lemma:mix:scale}, $Z_n^{(\beta)}(\sigma_1(K_n^{(\beta)}))$ converges to $1$ almost surely. 

\noindent $\bullet$ This proves point (ii) almost immediately: notice that, conditionally on $Z_n^{(\beta)}\big(\sigma_1((K_n^{(\beta)})^-)\big)=x>0$, the distribution of the jump is then given by $\frac{P_{nx}(i,j)}{1-P_{nx}(i,i)}$ for $j\neq i$, which is seen by $(\mathsf H_{\mathrm{mix}})$ (ii) to converge a.s. to $\frac{q(i,j)}{|q(i,i)|}$. To remove the conditioning, note that, if we take any $y\in (0,1)$, we have
\begin{align*}
\pr \Big(K_n^{(\beta)}\big(\sigma_1(K_n^{(\beta)})\big)&=j\Big)  =\pr\Big(K_n^{(\beta)}\big(\sigma_1(K_n^{(\beta)})\big)=j,Z_n^{(\beta)}(\sigma_1((K_n^{(\beta)})^-))< y\Big)\\ &+ \sum_{x\geq y,nx\in\Z_+} \pr\left(Z_n^{(\beta)}\big(\sigma_1((K_n^{(\beta)})^-)\big)=x\right)\frac{P_{nx}(i,j)}{1-P_{nx}(i,i)}.
\end{align*}
The first term of the right--hand side tends to $0$, while for the second term, the uniform convergence of the various $\frac{P_{nx}(i,j)}{1-P_{nx}(i,i)}$ to $\frac{q(i,j)}{|q(i,i)|}$ as $n$ tends to infinity for $x\geq y$ gives us the wanted conclusion.

\noindent $\bullet$ For the convergence in distribution stated in (i), we use the structure of our time--changed process. When the position component is at $x>0$, then the waiting time until the next jump of the position component is an exponential variable with parameter $(nx)^{\beta}$, and this jump has probability $1-P_{nx}(i,i)$ of inducing a change of type. Thus, still conditionally on $Z_n^{(\beta)}\big(\sigma_1((K_n^{(\beta)})^-)\big)=x$, we can write the following stochastic domination:
\[\sum_{i=1}^{G} \mathcal{E}_i \preceq \sigma_1(K_n^{(\beta)})\preceq \sum_{i=1}^{G'} \mathcal{E}'_i,\]
where $\preceq$ indicates stochastic domination, the $(\mathcal{E}_i)$ (resp. $(\mathcal{E}'_i)$)  form an i.i.d. sequence of exponential variables with parameter $n^{\beta}$ (resp  $(nx)^{\beta}$ ), $G$ (resp $G'$) is an independent geometric variable with parameter $(nx)^{-\beta}(|q(i,i)|+\eta_{nx})$ (resp $n^{-\beta}(|q(i,i)|-\eta_{nx})$). One readily checks that both the upper and lower bound have the appropriate convergence in distribution. We show it for the lower bound, using the moment generating function, and leave the upper bound to the reader: for $t>0$, we have
\begin{align*}
\mathbb{E}\left[\exp\left(-t\sum_{i=1}^{G} \mathcal{E}_i\right)\right]&=\mathbb{E}\left[\left(1+tn^{-\beta}\right)^{-G}\right] \\
                              &= \frac{(nx)^{-\beta}\big(|q(i,i)|+\eta_{nx}\big)\left(1+tn^{-\beta}\right)^{-1}}{1-\left(1-(nx)^{-\beta}(|q(i,i)|+\eta_{nx})\right)\left(1+tn^{-\beta}\right)^{-1}}.
\end{align*}
As $(n,x)$ tends to $(\infty,1),$ this converges to $\frac{|q(i,i)|}{|q(i,i)|+t}$, which is the moment generating function of the wanted exponential distribution. The same argument by uniform convergence as in the proof of (ii) shows then that we can remove the conditioning, and $\mathbb{E}[\e^{-t\sigma_1(K_n^{(\beta)})}]$ also converges to $\frac{|q(i,i)|}{|q(i,i)|+t}$.

\noindent $\bullet$ To deduce from this convergence in distribution the convergence of all positive moments, we will prove that, for all $k\in\N,$ $\E[\big(\sigma_1(K_n^{(\beta)})\big)^k]$ is uniformly bounded in $n$. This is enough to conclude since the r.v. $\big(\sigma_1(K_n^{(\beta)})\big)^a$ are then uniformly integrable for all $a>0$. For $n\in\N$ and $k\in\Z_+$, let 

\[u_{n,k}=\E[\sigma_1(K_n^{(\beta)})^k],\quad v_{n,k}=\sup_{\substack{m\leq n \\ l\leq k}} u_{m,l},\quad \text{and}\quad w_k=\underset{n\in\N}\sup\, v_{n,k}.\] Note that $u_{n,k}$ is finite for all $n$ and $k$, by (\ref{eq:momentsfinis}), and thus $v_{n,k}$ also is. Our aim is to show that $w_k$ is finite for all $k$.
To that purpose, let $n_0$ be large enough such that, for $n\geq n_0$, we have $n^{\beta}(1-P_n(i,i))\geq |q_{i,i}|/2$. We will prove that, for all $k\in\N$, 
\begin{equation}\label{eq:bornemoments}
w_k\leq v_{n_0,k}\vee \frac{2w_{k-1}k!e}{|q_{i,i}|}.
\end{equation}
Since $w_0=1$, an induction then finishes the proof.

Let therefore $k\in\N$, we prove equation (\ref{eq:bornemoments}) by showing that, for $n>n_0$, $v_{n,k}\leq v_{n-1,k}\vee \frac{2w_{k-1}k!e}{|q_{i,i}|}$. Assume that $v_{n,k}> v_{n-1,k}$ (otherwise there is nothing to do), implying $v_{n,k}=u_{n,k}$. Use the structure of the process to write
\[\sigma_1(K_n^{(\beta)})=E_n+\mathbbm{1}_{\{J_n(1)=i\}}\sigma_1(K'_{X_n(1)})\]
where $E_n$ is an independent exponential variable with parameter $n^{\beta}$ and given $X_n(1)=\ell$, $K'_{X_n(1)}$ is independent of $(E_n,J_n(1))$ and distributed as $K^{(\beta)}_{\ell}$. We can then bound the $k$-th moment thus:

\[u_{n,k}\leq \frac{k!}{n^{k\beta}}+P_n(i,i)\left(\sum_{l=1}^{k-1}{{k}\choose{l}}\frac{l!}{n^{l\beta}}v_{n,k-l} +v_{n,k}\right).\]
Bounding all instances of $v_{n,k-l}$ by $w_{k-1}$, and $n^{-l}$ by $n^{-1}$, we get 
\[v_{n,k}\leq P_n(i,i)v_{n,k}+\frac{1}{n^{\beta}}\left(k!+w_{k-1}\sum_{l=1}^{k-1}{{k}\choose{l}}l! \right).\]
It follows that
\begin{align*}
v_{n,k}(1-P_n(i,i))n^{\beta}&\leq k!\Big(1+w_{k-1}\sum_{l=1}^{k-1}\frac{1}{(k-l)!}\Big) \\
                                        &\leq k!w_{k-1}e.
\end{align*}
Since $n>n_0$, we have $(1-P_n(i,i))n^{\beta}\geq \frac{|q_{i,i}|}{2}$, and thus $v_{n,k}\leq  \frac{2w_{k-1}k!e}{|q_{i,i}|},$ ending the proof.
\end{proof}

Standard coupling results then imply that there exists a deterministic non--increasing sequence $(\rho_n)_{n\in\N}$ which converges to $0$ and such that we can couple $(X^c_n,J^c_n)$ with $L$ in such a way that, calling  $\sigma_1(L)$ the first jump time of $L$, we have
\[\sigma_1(K^{(\beta)}_n)\overset{\mathrm{a.s.}}{\underset{n\to\infty}\longrightarrow} \sigma_1(L) \]
and
\[\pr\big(K^{(\beta)}_n(\sigma_1(K^{(\beta)}_n))\neq L(\sigma_1(L))\big)\leq \rho_n, \forall n\in\N\]
for any initial type $i$. Note that the a.s. convergence is in fact also an $L_1$ convergence, by a standard variation of Scheff\'e's lemma: separating the positive and negative parts, $(\sigma_1(L)-\sigma_1(K^{(\beta)}_n))_+$ is nonnegative and dominated by $\sigma_1(L)$ and thus its expectation converges to $0$, and then we write the negative part as $\E[(\sigma_1(L)-\sigma_1(K^{(\beta)}_n))_-]=\E[(\sigma_1(L)-\sigma_1(K^{(\beta)}_n))_+]-\E[\sigma_1(L)-\sigma_1(K^{(\beta)}_n)]$ and see that its limit is also $0$. Thus, up to changing our sequence $(\rho_n)_{n\in\N}$, we now also have
\[\mathbb{E}\big[|\sigma_1(K^{(\beta)}_n)-\sigma_1(L)|\big]\leq \rho_n.\]

\bigskip

\noindent \textbf{Next jumps.} The processes $K^{(\beta)}_n$ and $L$ are now in a sense coupled until their first respective jumps. To continue the coupling, if they make the same first jump then we continue as above, and if they don't, we have to make them equal again, which we do by running $L$ for some more time until it reaches the same value as $K^{(\beta)}_n$. Let us formalise this. Let $(\sigma_i(K^{(\beta)}_n),i\in\Z_+)$ and $(\sigma_i(L),i\in\Z_+)$ be the lists of jump times of $K^{(\beta)}_n$ (these jump times are infinite by convention after the last jump) and $L$, with the extra convention that $\sigma_0(K^{(\beta)}_n)=\sigma_0(L)=0$. We also let $W_i(K_n^{(\beta)})=\sigma_i(K^{(\beta)}_n)-\sigma_{i-1}(K^{(\beta)}_n)$ be the $i$-th waiting time of $K_n^{(\beta)}$ for $i\in\Z_+$, and $W_i(L)$ be the same for $L$. We build an auxiliary process $L'$, its jump times $(\sigma_i(L'),i\in\Z_+)$ and waiting times $(W_i(L'),i\in\Z_+)$ and an increasing sequence of random integers $(i_k,k\in\Z_+),$ such that $L'$ has the same list of jumps as $K_n^{(\beta)}$, and $W_{k+1}(L')=W_{i_k+1}(L)$ for all $k\in \Z_+$. We do it by induction:
\begin{itemize}
\item $i_0=0$, $L'(0)=L(0)=K_n^{(\beta)}(0)$ and $\sigma_0(L')=0$;
\item for all $k\geq 0$, knowing $i_k$ and $\sigma_{k}(L')$, let
\[\sigma_{k+1}(L')=\sigma_{k}(L')+W_{i_k+1}(L),\]
with $L'(t)=K_n^{(\beta)}\big(\sigma_k(K_n^{(\beta)})\big)$ for $t\in[\sigma_{k}(L'), \sigma_{k+1}(L')),$ and let
\[i_{k+1}=\inf\Big\{i\geq i_k+1:\;L(\sigma_i(L))=K^{(\beta)}_n(\sigma_{k+1}(K^{(\beta)}_n))\Big\}.\]
\end{itemize}
This defines $L'$ uniquely. Now, let $\mathcal{A}_n(k)$ be the sigma--field generated by the $\sigma_i(L'),\sigma_i(K_n^{(\beta)})$ for $i\leq k$, the values of $L'$ and $K_n^{(\beta)}$ at these respective jump times, as well as the $X_n^c(n^{\beta}\tau_n^{(\beta)}(\sigma_i(K_n^{(\beta)})),$ $i \leq k$. 
By repeating the previous coupling at each jump, the processes $L,L',X_n^c,J_n^c$ can be built such that
\begin{equation}\label{bigcoupling}
\E\Big[\big|W_{k+1}(K^{(\beta)}_n)-W_{k+1}(L')\big|\mid \mathcal{A}_n(k),X_n^c\big(n^{\beta}\tau_n^{(\beta)}(\sigma_k(K_n^{(\beta)}))\big)\geq n\veps\Big]\leq  \rho_{\lfloor n\veps\rfloor}
\end{equation}
\begin{equation}
\label{ik}
\pr\big(i_{k+1}\neq i_k+1\mid \mathcal{A}_n(k),X_n^c\big(n^{\beta}\tau_n^{(\beta)}(\sigma_k(K_n^{(\beta)}))\big)\geq n\veps\big)\leq \rho_{\lfloor n\veps\rfloor}.
\end{equation}
In this coupling we can, and will, moreover assume that the jump times of $L$ that are not involved in $L'$ are independent of $L',X_n^c,J_n^c$.

From now on, let $t>0$ be a fixed time. We run the coupling until $k$ reaches the value $k_{\rm{max}}(n)$ defined by
\[k_{\rm{max}}(n)=\inf\Big\{k\in\N:\; \sigma_k(K_n^{(\beta)}) > n^{\gamma-\beta} t \text{ or } X_n^{c}\big(n^{\beta}\tau_n^{(\beta)}(\sigma_k(K_n^{(\beta)}))\big)<n\veps\Big\}.\]
We will need a few properties concerning $k_{\rm{max}}(n)$ and $\sigma_{k_{\rm{max}}(n)}(K_n^{\beta}).$

\begin{lem}\label{lem:kmax}
We have the following:
\begin{itemize}
\item[$\mathrm{(i)}$] For all $\delta>0$, there exists $C_{\delta}$ such that $\E\big[W_{k_{\rm{max}}(n)}(K_n^{(\beta)})\big]\leq C_{\delta}n^{\delta}$ for all $n$ large enough.
\item[$\mathrm{(ii)}$] There exists $C>0$ such that $\mathbb{E}[k_{\rm{max}}(n)]\leq Cn^{\gamma-\beta}$ for all $n$ large enough.
\item[$\mathrm{(iii)}$] $n^{\beta-\gamma}\sigma_{k_{\rm{max}}(n)}(K_n^{\beta}) - t\wedge S_{n,\veps}$ tends to $0$ in $L_1$, as does $n^{\beta-\gamma}\sigma_{k_{\rm{max}}(n)}(L') - t\wedge S_{n,\veps}.$
\end{itemize}
\end{lem}

\begin{proof}
Before proving (i), we first establish a weaker version of (ii). Note that, by (\ref{bigcoupling}), there exists $a>0$ such that, for $n\in\N$ large enough and all $k\in\N$
\[\E\Big[W_k(K^{(\beta)}_n)\mid \mathcal{A}_n(k-1),k\leq k_{\rm{max}}(n)\Big]\geq a.\]
As such, by Wald's formula (Lemma \ref{wald}), we have
\begin{align*}
a\E[k_{\rm{max}}(n)]&\leq \E\Bigg[\sum_{k=1}^{k_{\rm{max}}(n)}W_k(K_n^{(\beta)})\Bigg]
= \E[\sigma_{k_{\rm{max}}(n)}(K_n^{(\beta)})]
=\E[\sigma_{k_{\rm{max}}(n)-1}(K_n^{(\beta)})]+\E[W_{k_{\rm{max}}(n)}(K_n^{(\beta)})]\\
&\leq tn^{\gamma-\beta}+ \E[W_{k_{\rm{max}}(n)}(K_n^{(\beta)})]. \numberthis \label{eq:debutmajorationkmax}
\end{align*}
This will be needed in the proof of (i).

\noindent $\bullet$ Point (i) takes more work. As a first step, let us first show that, for all $b>0$, there exists $c_b>0$ such that
\begin{equation}\label{eq:kmaxjumpbthpower}
\E\left[\big(W_{k_{\rm{max}}(n)}(K_n^{(\beta)})\big)^b\right]\leq c_b \Big(n^{\gamma-\beta}+\E\big[W_{k_{\rm{max}}(n)}(K_n^{(\beta)})\big]\Big).
\end{equation}
Write
\begin{align*}
\E\left[\big(W_{k_{\rm{max}}(n)}(K_n^{(\beta)})\big)^b\right] &= \sum_{k=1}^\infty \E\Big[\Big(W_k(K_n^{(\beta)}\big)^b\mathbbm{1}_{\{k=k_{\rm{max}}(n)\}}\Big] \\
&\leq \sum_{k=1}^\infty \E\Big[\Big(W_k(K_n^{(\beta)}\big)^b\mathbbm{1}_{\{k\leq k_{\rm{max}}(n)\}}\Big] \\
&\leq \sum_{k=1}^\infty \E\Big[\Big(W_k(K_n^{(\beta)})\Big)^b\mid k\leq k_{\rm{max}}(n)\Big]\pr\big(k\leq k_{\rm{max}}(n)\big).
\end{align*}
As before, we can apply the Markov property at time $k-1$ and Lemma \ref{lem:closetoL} to get a constant $c$ not depending on $n$ or $k$ such that
\[\E\Big[\big(W_k(K_n^{(\beta)})\big)^b\mid k\leq k_{\rm{max}}(n)\Big]\leq c.\]
We can now write
\begin{align*}
\E\left[\big(W_{k_{\rm{max}}(n)}(K_n^{(\beta)})\big)^b\right]&\leq c\sum_{k=1}^{\infty}\pr(k\leq k_{\max}(n)) \\
             &\leq c\E[k_{\rm{max}}(n)] \\
             &\leq ca^{-1}\big(n^{\gamma-\beta}t+\E[W_{k_{\rm{max}}(n)}(K_n^{(\beta)})]\big),          
\end{align*}
where the last line comes from (\ref{eq:debutmajorationkmax}). This gives (\ref{eq:kmaxjumpbthpower}). Now to prove (i), let $\delta>0$, and $b>1$, and write
\begin{align*}
\E\Big[W_{k_{\rm{max}}(n)}(K_n^{(\beta)})\Big]&\leq n^{\delta}+\E\Big[W_{k_{\rm{max}}(n)}(K_n^{(\beta)})\mathbbm{1}_{\{W_{k_{\rm{max}}(n)}(K_n^{(\beta)})>n^{\delta}\}}\Big]\\
           &\leq n^{\delta}+n^{-(b-1)\delta}\E\big[\big(W_{k_{\rm{max}}(n)}(K_n^{(\beta)})\big)^b\big] \\
           &\leq n^{\delta}+n^{-(b-1)\delta}c_b\Big(n^{\gamma-\beta}+\E\big[W_{k_{\rm{max}}(n)}(K_n^{(\beta)})\big]\Big).
\end{align*}
Notice that $\E\big[W_{k_{\rm{max}}(n)}(K_n^{(\beta)})\big]$ is finite for $n>2\veps^{-1}$ because, with the changes made in Section \ref{sec:mix:changements}, there is at least one change of type after $\sigma_{k_{\rm{max}}(n)-1}(K_n^{(\beta)})$, implying $W_{k_{\rm{max}}(n)}(K_n^{(\beta)})\leq T_0(Z_n^{(\beta)})$ which has finite expectation by (\ref{eq:momentsfinis}). We can then write
\[\E\Big[W_{k_{\rm{max}}(n)}(K_n^{(\beta)})\Big](1-n^{-(b-1)\delta}c_b)\leq n^{\delta}+n^{\gamma-\beta-(b-1)\delta}c_b,\]
which yields (i) if $b>\max(1,(\gamma-\beta)/\delta)$.

\noindent $\bullet$ Point (ii) is obtained by combining (\ref{eq:debutmajorationkmax}) with point (i).

\noindent $\bullet$ For the first part of point (iii), notice that $\sigma_{k_{\rm{max}}(n)-1}(K_n^{(\beta)})\leq n^{\gamma-\beta}(t\wedge S_{n,\veps})\leq \sigma_{k_{\rm{max}}(n)}(K_n^{(\beta)})$ and thus $\E\big[|n^{\beta-\gamma}\sigma_{k_{\rm{max}}(n)}(K_n^{(\beta)}) - t\wedge S_{n,\veps}|\big]\leq n^{\beta-\gamma}\E\big[\sigma_{k_{\rm{max}}(n)}(K_n^{(\beta)})-\sigma_{k_{\rm{max}}(n)-1}(K_n^{(\beta)})\big]$. By (i) and the hypothesis $\gamma>\beta$, we get that $\E\big[|n^{\beta-\gamma}\sigma_{k_{\rm{max}}(n)}(K_n^{(\beta)}) - t\wedge S_{n,\veps}|\big]$ tends to $0$.
The second part of (iii) is then reduced to showing that $n^{\beta-\gamma}(\sigma_{k_{\rm{max}}(n)}(L')-\sigma_{k_{\rm{max}}(n)}(K_n^{(\beta)}))$ tends to $0$ in $L_1$. Rewriting this as the sum of the differences of the waiting times and then using Wald's formula again we obtain
\begin{eqnarray*}
\E\Big[\left|\sigma_{k_{\rm{max}}(n)}(L')-\sigma_{k_{\rm{max}}(n)}(K_n^{(\beta)})\right|\Big]&\leq&
\E\left[\sum_{k=1}^{k_{\rm{max}}(n)}\left|W_k(K_n^{(\beta)})-W_k(L')\right|\right] \\ 
&\underset{\text{Lemma \ref{wald}} + (\ref{bigcoupling})}\leq& \rho_{\lfloor n\veps\rfloor} \E[k_{\rm{max}}(n)] \\ 
&\underset{\text{(ii)}} \leq& Cn^{\gamma-\beta}\rho_{\lfloor n\veps\rfloor},
\end{eqnarray*}
thus ending the proof, since $\rho_{\lfloor n\veps\rfloor}$ has limit $0$.
\end{proof}

\subsubsection{Proof of (\ref{simplemelange})}
\label{sec:endmixing}

Let 
$$I_L(n)=\int_{0}^{t\wedge S_{n,\veps}} \delta_{L(n^{\gamma-\beta}s)}\mathrm d s, \quad I_{L'}(n)=\int_{0}^{t\wedge S_{n,\veps}} \delta_{L'(n^{\gamma-\beta}s)}\mathrm d s \quad \text{and} \quad I_{K}(n)=\int_{0}^{t\wedge S_{n,\veps}} \delta_{K_n^{(\beta)}(n^{\gamma-\beta}s)}\mathrm d s.$$ We will argue that both $I_{L'}(n)-I_K(n)$ and $I_L(n)-I_{L'}(n)$ converge in $L_1$ to the zero measure as $n$ goes to infinity, which, combined with (\ref{eq:limiteL}), will give (\ref{simplemelange}). Since the considered measures are in a finite--dimensional vector space, we use the simple norm $|.|$ given, for a measure $\nu$ on $\{1,\ldots,\kappa\}$, by
\[|\nu|=\sum_{i=1}^{\kappa} |\nu(i)|.\]

\noindent $\bullet$ Notice first that, knowing that $L'$ and $K_n^{(\beta)}$ have the same jumps, but simply jump at different times, we can bound $|I_{L'}(n)-I_K(n)|$ by 
\begin{align*}
\big|I_{L'}(n)-I_{K}(n)\big|\leq \frac{1}{n^{\gamma-\beta}} \Bigg( &\sum_{k=1}^{k_{\mathrm{max}}(n)}\big|W_k(K^{(\beta)}_n)-W_k(L')\big| +\\ &\big|n^{\gamma-\beta}(t\wedge S_{n,\veps}) - \sigma_{k_{\rm{max}}(n)}(K_n^{(\beta)})\big| + \big|n^{\gamma-\beta}(t\wedge S_{n,\veps}) - \sigma_{k_{\rm{max}}(n)}(L')\big|\Bigg).
\end{align*}
By Lemma \ref{lem:kmax}, the second and third terms in the brackets tend to $0$ in  $L_1$ when divided by $n^{\gamma-\beta}$. The first one has already been shown to converge in $L_1$ to $0$ at the end of the proof of Lemma \ref{lem:kmax}.

\noindent $\bullet$  Comparing $I_L(n)$ and $I_{L'}(n)$ requires more work. We will, in order, prove that all the following random variables converge to $0$ in $L_1$:
\begin{itemize}
\item[$\mathrm{(i)}$] $n^{\beta-\gamma}\left(\sigma_{i_{k_{\mathrm{max}}(n)}}(L)-\sigma_{k_{\mathrm{max}}(n)}(L')\right),$
\item[$\mathrm{(ii)}$] $\Big|I_{L'}(n)-n^{\beta-\gamma}\sum_{k=1}^{k_{\mathrm{max}}(n)}W_k(L')\delta_{L'(\sigma_{k-1}(L'))}\Big|,$
\item[$\mathrm{(iii)}$] $\Big|I_L(n)-n^{\beta-\gamma}\sum_{k=1}^{i_{k_{\mathrm{max}}(n)}}W_k(L)\delta_{L(\sigma_{k-1}(L))}\Big|,$
\item[$\mathrm{(iv)}$] $\big|I_L(n)-I_{L'}(n)\big|.$
\end{itemize}

\noindent $\circ$ The proof of (i) relies on two basic observations: for every $k$, $i_{k+1}-i_k$ is, conditionally on $\mathcal{A}_n(k)$ and $k+1\leq k_{\rm{max}}(n)$, equal to $1$ with probability at least $1-\rho_{\lfloor n\veps\rfloor}$ (see (\ref{ik})) and, conditionally on it not being equal to $1$ (an event with probability less than $\rho_{\lfloor n\veps\rfloor}$), it is $1$ plus some hitting time of the discrete, finite state space Markov chain embedded in $L$, and thus bounded in expectation by irreducibility. We can write \[
\mathbb{E}\big[i_{k+1}-i_k\mid \mathcal{A}_n(k),k+1\leq k_{\rm{max}}(n)\big]\leq 1 + D\rho_{\lfloor n\veps\rfloor}\] for some constant $D>0$. Moreover, if we additionally condition on the value of $i_{k+1}-i_k$, then $\sigma_{i_{k+1}}(L)-\sigma_{i_k}(L)-(\sigma_{k+1}(L')-\sigma_k(L'))$ is just the time it takes for $L$ to go from $L\big(\sigma_{{i_k}+1}(L)\big)$ to $K_n^{(\beta)}\big(\sigma_{k+1}(K_n^{(\beta)}\big)$, knowing it needs $i_{k+1}-i_k-1$ independent jumps to do so. Thus
\[\E\left[\sigma_{i_{k+1}}(L)-\sigma_{i_k}(L)-(\sigma_{k+1}(L')-\sigma_k(L'))\mid  \mathcal{A}_n(k),k+1\leq k_{\rm{max}}(n),i_{k+1}-i_k \right]\leq D'(i_{k+1}-i_k-1)\]
where $D'=\displaystyle\underset{i\in\{1,\ldots,\kappa\}}\sup 1/|q_{i,i}|$. So finally,
\begin{equation}
\label{majoDD'}
\mathbb E \left[ \sigma_{i_{k+1}}(L)-\sigma_{i_k}(L)-(\sigma_{k+1}(L')-\sigma_k(L'))\mid k+1\leq k_{\rm{max}}(n) \right] \leq D'D\rho_{\lfloor n\veps\rfloor}.
\end{equation}
Write then
\begin{eqnarray*}
\E\left[\sigma_{i_{k_{\rm{max}}(n)}}(L)-\sigma_{k_{\rm{max}}(n)}(L')\right]&=&\E\left[\sum_{k=0}^{k_{\rm{max}}(n)-1}\sigma_{i_{k+1}}(L)-\sigma_{i_k}(L)-\big(\sigma_{k+1}(L')-\sigma_k(L')\big)\right] \\
&\underset{(\ref{majoDD'})+\text{Lemma \ref{wald}}} \leq&  D'D\rho_{\lfloor n\veps\rfloor}\E\big[k_{\rm{max}}(n)\big]
\end{eqnarray*}
and by Lemma \ref{lem:kmax} (ii), this tends to $0$ when multiplied by $n^{\beta-\gamma}$.

\noindent $\circ$ Item (ii) is proved by noting that 
\[\bigg|I_{L'}(n)-n^{\beta-\gamma}\sum_{k=1}^{k_{\mathrm{max}}(n)}W_k(L')\delta_{L'(\sigma_{k-1}(L'))}\bigg|=\bigg| n^{\beta-\gamma}\sigma_{k_{\mathrm{max}}(n)}(L')-t\wedge S_{n,\veps}\bigg|,\]
and so its limit is $0$ by Lemma \ref{lem:kmax} (iii).

\noindent $\circ$ For (iii), notice similarly that \[\bigg|I_{L}(n)-n^{\beta-\gamma}\sum_{k=1}^{i_{k_{\mathrm{max}}(n)}}W_k(L)\delta_{L(\sigma_{k-1}(L))}\bigg|=\bigg| n^{\beta-\gamma}\sigma_{i_{k_{\mathrm{max}}(n)}}(L)-t\wedge S_{n,\veps}\bigg|,\]
and so its limit is $0$ by (i) and Lemma \ref{lem:kmax} (iii).

\noindent $\circ$  Finally for (iv), notice first that, by (ii) and (iii),
\[\underset{n\to\infty}\lim \big|I_L(n)-I_{L'}(n)\big|=\underset{n\to\infty}\lim n^{\beta-\gamma}\Bigg|\sum_{k=1}^{i_{k_{\mathrm{max}}(n)}}W_k(L)\delta_{L(\sigma_{k-1}(L))} - \sum_{k=1}^{k_{\mathrm{max}}(n)}W_k(L')\delta_{L'(\sigma_{k-1}(L')}\Bigg|.\] Note that, by construction, for all $k\leq k_{\mathrm{max}}(n)$, the $k$-th term of the second sum is equal to the $(i_{k-1}+1)$-th term in the first one, and so we can write
\[\bigg|\sum_{i=i_{k-1}+1}^{i_k}W_i(L)\delta_{L(\sigma_{i-1}(L))}-W_k(L')\delta_{L'(\sigma_{k-1}(L'))}\bigg| \leq \sum_{i=i_{k-1}+1}^{i_k}W_i(L)-W_k(L').\]
Summing over $k$, we have
\begin{align*}
\bigg|\sum_{k=1}^{i_{k_{\mathrm{max}}(n)}}W_k(L)\delta_{L(\sigma_{k-1}(L))} - \sum_{k=1}^{k_{\mathrm{max}}(n)}W_k(L')\delta_{L'(\sigma_{k-1}(L'))}\bigg|&\leq \sum_{k=1}^{k_{\rm{max}}(n)}\left(\sigma_{i_k}(L)-\sigma_{i_{k-1}}(L)-(\sigma_k(L')-\sigma_{k-1}(L'))\right) \\
  &\leq \sigma_{i_{k_{\mathrm{max}}(n)}}(L)-\sigma_{k_{\mathrm{max}}(n)}(L'),
\end{align*}
and (i) ends the proof of (iv). \qed

%**********************************************************************

\section{Solo regime}
\label{SecSolo}

%**********************************************************************

We now focus on cases where the rate of type change is much smaller than the rate of macroscopic jumps. The chain will therefore not change type in the scaling limit, with a dynamic that only depends on its initial type, which brings us back to the standard monotype setting.

\begin{mybox}
\noindent \textbf{Hypothesis $(\mathsf H_{\mathrm{sol}})$.} We fix a type $i$ and assume that there exists $\gamma>0$ such that:
\begin{enumerate}
\item[(i)] There exists a non--trivial, finite measure $\mu^{(i)}$ on $[0,1]$ such that for all continuous functions $f:[0,1]\rightarrow \mathbb R$,
$$
n^{\gamma} \sum_{m=0}^n f\left(\frac{m}{n}\right) \left(1-\frac{m}{n}\right)p_{n,i}(m,i) \underset{n\rightarrow \infty}\longrightarrow \int_{[0,1]} f(x)\mu^{(i)}(\mathrm dx). 
$$
\item[(ii)] Moreover,  
$$
\sum_{j\in \{1,\ldots,\kappa\} \backslash\{i\}}P_n(i,j)=o(n^{-\gamma}).
$$
\end{enumerate}
\end{mybox}

\smallskip

As before, we let $Z_n^{(i)}$ denote the Lamperti transform of $X_n^{(i)}$ defined by (\ref{Zn}) via the time--change (\ref{timechange}). 

\smallskip

\begin{thm} 
\label{ThSolo}  Fix a type $i$ and assume $(\mathsf H_{\mathrm{sol}})$ for $i$. 

\noindent \emph{(i)} Then,
$$
\left(\frac{X^{(i)}_n(\lfloor n^{\gamma} \cdot \rfloor)}{n}, Z_n^{(i)}(\lfloor n^{\gamma} \cdot \rfloor)\right) \ \overset{\mathrm{(d)}}{\underset{n \rightarrow \infty} \longrightarrow} \ \big(X^{(i)}, Z^{(i)} \big),
$$
where $-\log(Z^{(i)})$ is a subordinator with Laplace transform 
$$
\psi_{(i)}(q)=\mu^{(i)}(\{0\})+\mu^{(i)}(\{1\})q+\int_{(0,1)}\left(\frac{1-x^q}{1-x}\right)\mu^{(i)}(\mathrm dx),
$$
and $X^{(i)}$ is the $\gamma$--Lamperti transform of $Z^{(i)}$.

\noindent \emph{(ii)} Assume moreover that $\liminf_{n \rightarrow \infty} n^{-\gamma}\sum_{k=0}^{\lfloor rn \rfloor} \sum_{l \in \{1,\ldots,\kappa\}} p_{n,j}(k,\ell)>0$ for some $r<1$ and all types $j$.
Then, jointly with the previous convergence
$$
\frac{A_n^{(i)}}{n^{\gamma}} \overset{\mathrm{(d)}}{\underset{n \rightarrow \infty} \longrightarrow} I^{(i)},
$$
with $I^{(i)}$ the extinction time of $X^{(i)}$. Additionally, there is convergence of all positive moments of $A_n^{(i)}/n^{\gamma}$ to those of $I^{(i)}$, which are all finite.
\end{thm}

\bigskip

\noindent \textbf{Remark.} In words, the additional assumption in Theorem \ref{ThSolo} (ii) says that the probability of doing a jump larger than $n-\lfloor an \rfloor$ is asymptotically larger than $cn^{-\gamma}$ for some $c>0$, whatever the starting type $j$. This assumption is probably too strong to get the conclusion of (ii) -- for example, in the case where $p_{n,i}(n,i)=1-n^{-\gamma}$,  $p_{n,i}(\lfloor n/2 \rfloor,i)=n^{-\gamma}$ and $\sum_{j\in \{1,\ldots,\kappa\} \backslash\{i\}}P_n(i,j)=n^{-\gamma-\varepsilon}$, $\varepsilon>0$, one can check that both (i) and (ii) hold, without any additional assumption. However we are not able to prove that $(\mathsf H_{\mathrm{sol}})$ alone implies the convergence of the absorption times in general. 

\begin{proof}
(i). Let $T_n^{\mathrm{type}}$ be the first time at which $X_n^{(i)}$ either changes its type, or is absorbed. Then, consider, on the one hand, the transition probabilities defined for $n,m \in \mathbb Z_+, m\leq n$, by
$$q_n(m)=p_{n,i}(m,i)+\sum_{j \in \{1,\ldots,\kappa\} \backslash \{i\}} p_{n,i}(m,j), \quad \forall m \leq n.$$
And, on the other hand, the transition probabilities defined for $n,m \in \mathbb Z_+, m\leq n$, by
$$r_n(m)=p_{n,i}(m,i) \quad \forall 1 \leq m \leq n, \quad  r_n(0)=p_{n,i}(0,i)+\sum_{j \in \{1,\ldots,\kappa\} \backslash \{i\}} p_{n,i}(m,j).$$
We can then couple the construction of $(X^{(i)}_n,J_n^{(i)})$ with that of two Markov chains $X^{\mathrm{(q)}}_n, X_n^{(\mathrm r)}$ on $\mathbb Z_+$ starting from $n$, with respective transition probabilities $(q_n(m)), (r_n(m))$ and such that $$X_n^{(\mathrm r)}(k)=X^{(i)}_n(k)\mathbbm 1_{\{k \leq T_n^{\mathrm{type}}-1\}},  \  \forall k \quad  \text{and} \quad X^{(i)}_n(k)=X^{\mathrm{(q)}}_n(k), \  \forall k \leq T_n^{\mathrm{type}}.$$ 
Next, note that  (i) and (ii) of $(\mathsf H_{\mathrm{sol}})$ imply the convergences
$$
n^{\gamma} \sum_{m=0}^n f\left(\frac{m}{n}\right) \left(1-\frac{m}{n}\right)q_{n}(m) \underset{n\rightarrow \infty}\longrightarrow \int_{[0,1]} f(x)\mu^{(i)}(\mathrm dx),
$$
$$
n^{\gamma} \sum_{m=0}^n f\left(\frac{m}{n}\right) \left(1-\frac{m}{n}\right)r_{n}(m) \underset{n\rightarrow \infty}\longrightarrow \int_{[0,1]} f(x)\mu^{(i)}(\mathrm dx),
$$
for all continuous $f:[0,1]\rightarrow \mathbb R$. This, together with Theorem 1 and Theorem 2 of \cite{HM11}, implies in turn that 
\begin{equation}
\label{soloeq1}
\left(\frac{X^{(\mathrm r)}_n(\lfloor n^{\gamma} \cdot \rfloor)}{n}, \frac{A^{(\mathrm r)}_n}{n^{\gamma}}\right) \ \overset{\mathrm{(d)}}{\underset{n \rightarrow \infty} \longrightarrow} \ \left(X^{(i)},  I^{(i)} \right), \quad \left(\frac{X^{(\mathrm q)}_n(\lfloor n^{\gamma} \cdot \rfloor)}{n}, \frac{A^{(\mathrm q)}_n}{n^{\gamma}}\right) \ \overset{\mathrm{(d)}}{\underset{n \rightarrow \infty} \longrightarrow} \ \left(X^{(i)},  I^{(i)} \right),
\end{equation}
with obvious notation, as well as the convergence of all positive moments of $n^{-\gamma}A_n^{(\mathrm r)}$ and $n^{-\gamma}A_n^{(\mathrm q)}$ to those of $I^{(i)}$. Note also that $$T_n^{\mathrm{type}}=A_n^{(\mathrm{r})}\quad \text{and} \quad A_n^{(\mathrm q)}=T_n^{\mathrm{type}}+\tilde A^{(\mathrm q)}_{X_n^{(\mathrm q)}(T_n^{\mathrm{type}})},$$ with $\tilde A^{(\mathrm q)}$ a process distributed as $A^{(\mathrm q)}$, independent of $X_n^{(\mathrm q)}(T_n^{\mathrm{type}})$ (for this we use that $T_n^{\mathrm{type}}$ is a randomized stopping time for $X_n^{(\mathrm q)}$). This, together with (\ref{soloeq1}), implies that $n^{-\gamma}\tilde A^{(\mathrm q)}_{X_n^{(\mathrm q)}(T_n^{\mathrm{type}})}$ converges to 0 in probability, which in turn implies  that 
\begin{equation}
\label{cv0}
\frac{X_n^{(i)}(T_n^{\mathrm{type}})}{n}=\frac{X_n^{(\mathrm q)}(T_n^{\mathrm{type}})}{n} \overset{\mathbb P}{\underset{n \rightarrow \infty} \longrightarrow} 0
\end{equation}
(note that $I^{(i)}>0$ a.s.).
So, finally, we have that
$$
\frac{X^{(\mathrm r)}_n(\lfloor n^{\gamma} \cdot \rfloor)}{n}  \underset{n\rightarrow \infty}{\overset{\mathrm{(d)}}\longrightarrow} X^{(i)}, \quad \text{for the Skorokhod topology}
$$
and 
$$
\left \|\frac{X^{(i)}_n(\lfloor n^{\gamma} \cdot \rfloor)}{n}-\frac{X^{(\mathrm r)}_n(\lfloor n^{\gamma} \cdot \rfloor)}{n}   \right \|_{\infty} \leq \frac{X_n^{(i)}(T_n^{\mathrm{type}})}{n} \underset{n\rightarrow \infty}{\overset{\mathbb P}\longrightarrow} 0
$$
and we conclude with a Slutsky--type argument that $n^{-1}X^{(i)}_n(\lfloor n^{\gamma} \cdot \rfloor)  \underset{n\rightarrow \infty}{\overset{\mathrm{(d)}}\longrightarrow} X^{(i)}.$

\bigskip

\noindent (ii) With our additional assumption, it is easy to prove, in a way very similar to the proof of Lemma \ref{lem:tightness}, that for all $a \geq 0$ and all types $j$,
\begin{equation}
\label{tightnessSolo}
\sup_{n \in \mathbb N}\mathbb E\left[\left(\frac{A_n^{(j)}}{n^{\gamma}}\right)^a\right]<\infty.
\end{equation}
Then, using the Markov property at time $T_n^{\mathrm{type}}$, we write 
$$
A_n^{(i)}=T_n^{\mathrm{type}}+\tilde A^{(J^{(i)}_n(T_n^{\mathrm{type}}))}_{X^{(i)}_n(T_n^{\mathrm{type}})}
$$
with $(\tilde A^{(j)}_k, k \geq 0, j \in \{1,\ldots,\kappa\})$ distributed as $(A^{(j)}_k, k \geq 0, j \in \{1,\ldots,\kappa\})$ and independent of \linebreak $(J^{(i)}_n(T_n^{\mathrm{type}}),X^{(i)}_n(T_n^{\mathrm{type}}))$. By (\ref{tightnessSolo}) and (\ref{cv0}), we have that 
$$\frac{\tilde A^{(J^{(i)}_n(T_n^{\mathrm{type}}))}_{X^{(i)}_n(T_n^{\mathrm{type}})}}{n^{\gamma}}\underset{n\rightarrow \infty}{\overset{\mathrm L^a} \longrightarrow}  0$$ for all $a \geq 0$. Besides,  Theorem 1 and Theorem 2 of \cite{HM11}  imply (\ref{soloeq1}) and the convergence of all positive moments of $n^{-\gamma}A_n^{(\mathrm r)}$ (equivalently $n^{-\gamma}T_n^{\mathrm{type}}$) to those of $I^{(i)}$. All this together implies the convergence in distribution of $n^{-\gamma} A_n^{(i)}$ to $I^{(i)}$ and that $\mathbb E\big[(n^{-\gamma} A_n^{(i)})^a \big]<\infty$ for all $a \geq 0$. Hence the conclusion. 
\end{proof}

%**********************************************************************

\section{Applications}

%**********************************************************************

As mentioned in the Introduction, the description of the scaling limits of non--increasing Markov chains on $\mathbb Z_+$ was an essential tool to describe the scaling limits of several random objects: random walks, coalescence or fragmentation--coalescence processes, trees, maps.

Our initial motivation to extend these results to  Markov chains on $\mathbb Z_+ \times \{1,\ldots, \kappa\}$ was to develop applications to the scaling limits of multi--type Markov branching trees, which is a natural family of trees carrying types, that includes some models of randomly growing trees, and multi--type Galton--Watson trees. These applications require some work and will be developed in the upcoming paper \cite{HS15}.

There are however others interesting, and more direct, applications. We mention here two of them.

\subsection{Collisions in coalescents in varying environment}

The $\Lambda$--coalescents were introduced by \cite{PitCoag99} and \cite{Sag99} and studied by several authors since then. These models allow multiple collisions (i.e more than 2 particles may coalesce at once) and the coalescing mechanism is driven by a finite measure on $[0,1]$, usually denoted by $\Lambda$. Roughly, such a process takes its values in the set of partitions of $\mathbb N$, is Markovian, exchangeable, and such that the rate at which $n$ particles (blocks) coalesce into $k$ particles (blocks), for $1\leq k \leq n-1$, is
\[
r_n(k)=\binom{n}{k-1}\int_{[0,1]}x^{n-k-1}(1-x)^{k-1}\Lambda(\mathrm dx), \quad 1 \leq k \leq n-1.
\]
The case where $\Lambda=\delta_0$ corresponds to Kingman's coalescent.
We refer to \cite{Berest09} for a review on that topic. 

We consider here a variation of this model where the environment may influence the coalescing mechanism, which is therefore allowed to vary from generation to generation. A generalization in the same spirit was already considered in \cite{Mohle02}.

\bigskip

\noindent \textbf{Coalescing mechanism.}  We assume that there are $\kappa$ possible environments.
Let $\Lambda^{(i)}, 1 \leq i \leq \kappa$ be $\kappa$ finite, non--trivial measures on $[0,1]$ such that $\Lambda^{(i)}(\{0\})=0$ and
\begin{equation}
\label{Hyp1Coal}
\int_{[u,1]} x^{-2}\Lambda^{(i)}(\mathrm dx) \underset{u \rightarrow 0}\sim c^{(i)} u^{-\gamma}
\end{equation}
for some $\gamma \in (0,1)$ and some strictly positive constants $c^{(i)}, 1 \leq i \leq \kappa$. To each of these measures, we associate the following Laplace exponent
\begin{equation}
\label{CoalLapl}
\psi_{(i)}(q)=\frac{1}{\Gamma(2-\gamma)c^{(i)}} \int_{[0,1]}\left(1-(1-x)^{q}\right)x^{-2}\Lambda^{(i)}(\mathrm dx).
\end{equation}
Besides, we let $P_n,n\geq 1$ be $\kappa \times \kappa$ stochastic matrices such that 
\begin{equation}
\label{Hyp2Coal}
n^{\beta}\left(P_n-I\right) \underset{n \rightarrow \infty} \rightarrow Q
\end{equation}
for some $\beta \geq 0$ and some irreducible $\mathsf Q$--matrix $Q$, and hence a unique stationary distribution denoted by $\pi=(\pi(i),1 \leq i \leq \kappa)$. 

The  coalescing mechanism then evolves as follows. In environment $i$, the particles coalesce according to the mechanism $\Lambda^{(i)}$, i.e., the probability that $n$ particles coalesce into $k$ particles is  
$$
p_n^{(i)}(k)=\frac{1}{Z^{(i)}_n}\binom{n}{k-1}\int_{[0,1]}x^{n-k-1}(1-x)^{k-1}\Lambda^{(i)}(\mathrm dx), \quad 1 \leq k \leq n-1
$$
where $Z^{(i)}_n$ is a normalizing constant.
Moreover, the probability that the environment changes from $i$ to $j$ when $n$ particles coalesce is $P_n(i,j)$,
so that the transition probabilities of our chain on $\mathbb Z_+ \times \{1,\ldots,\kappa\}$ is 
$$p_{(n,i)}(k,j)=P_n(i,j) p_n^{(i)}(k).$$ 
When the matrix $P_n$ is constant, independent of $n$, this corresponds to situations where the change of environments does not depend on the number of present particles.

\bigskip

\noindent \textbf{Number of collisions.} Starting from $n$ large, the quantity we are interested in is the total number of collisions (that is, the number of steps) until all the $n$ initial particles have coalesce in a unique particle. We let $K_n^{(i)}$ denote this random variable when the $n$ initial articles are in environment $i$. When there is a unique environment ($\kappa=1$), this question has been treated by several authors \cite{GY07,GIM08,IM08,IMM09,HM11,GIM11}. In a varying environment, we obtain as a direct consequence of our results:

\begin{thm}
Assuming (\ref{Hyp1Coal}) and (\ref{Hyp2Coal}), we have for all $i_0 \in \{1,\ldots, \kappa\}$, 
$$
\frac{K_n^{(i_0)}}{n^{\gamma}} \underset{n \rightarrow \infty}{\overset{\mathrm{(d)}}\longrightarrow} \int_0^{\infty} \exp(-\gamma \xi_r)\mathrm dr
$$
where,
\begin{enumerate}
\item[\emph{(i)}] If $\beta=\gamma$: $\xi$ is the first marginal of a \emph{MAP} $(\xi,J)$ on  $\mathbb R_+ \times \{1,\ldots,\kappa\}$, starting from $(\xi(0),J(0))=(0,i_0)$ and with characteristics: 
\begin{enumerate}
\item[$\circ$] $\lambda_{i,j}=Q(i,j)$ and $B_{i,j}=\delta_0$ for all $i \neq j$ 
\item[$\circ$]  $(1+Q(i,i) \mathbbm 1_{\{\beta=0\}})\psi^{(i)}$, with $\psi^{(i)}$ as defined in (\ref{CoalLapl}), for all types $i$.
\end{enumerate}
\item[\emph{(ii)}] If $0\leq \beta<\gamma$: $\xi$ is a subordinator with Laplace exponent $\sum_{i=1}^{\kappa} \pi(i) \psi^{(i)}.$
\item[\emph{(iii)}] If $\beta>\gamma>0$:  $\xi$ is a subordinator with Laplace exponent $\psi^{(i_0)}$.
\end{enumerate}
There is also convergence of all positive moments of $K_n^{(i_0)}/n^{\gamma}$ to those of $\int_0^{\infty} \exp(-\gamma \xi_r)\mathrm dr$. Moreover: if we denote by $X^{(i_0)}_n(k), k \geq 0$ the number of particles after $k$ collision steps, starting from $n$ particles in environment $i_0$,
$$
\frac{X_n^{(i_0)}(\lfloor n^{\gamma} \cdot\rfloor)}{n} \underset{n \rightarrow \infty}{\overset{\mathrm{(d)}}\longrightarrow} \exp(-\xi_{\rho}),
$$
where $\rho$ is the usual time--change $\rho(t)=\inf \{u:\int_0^u \exp(-\gamma \xi_r) \mathrm d r>t\}$. 
\end{thm}

\begin{proof}
By Lemma 8 and Lemma 9 of \cite{HM11}, we know that under (\ref{Hyp1Coal}),
$$
n^{\gamma} \sum_{k=0}^n \left(1-\left( \frac{k}{n}\right)^q\right) p^{(i)}_{n}(k)\underset{n \rightarrow \infty}{\longrightarrow} \psi_{(i)}(q), \quad  \text{for all } q \geq 0.
$$
 which, together with Hypothesis (\ref{Hyp2Coal}) and Theorems \ref{ThCritical}, \ref{ThCriticalJoint}, \ref{ThMixing}, \ref{ThMixingjoint} and \ref{ThSolo} readily yields the result.
\end{proof}

\subsection{Markov random walks with a barrier}

Random walks with barriers are variants of the usual random walks with i.i.d. increments, conditioned on not going over or below some fixed real numbers. In \cite{HM11} and \cite{IM08}, where they were also linked to some coalescent processes, some results on their scaling limits are established, in particular when the increments are heavy--tailed. We generalize these in a simple multi--type setting.

\bigskip

We consider a Markov random walk $\big((S_k,J_k),k\geq 0)$ on $\mathbb{Z}\times\{1,\ldots,\kappa\}$.  This process is the discrete analogue of a MAP and  a natural generalization of a random walk with i.i.d. increments. It is a process such that, conditionally on $\mathcal{F}_k$, where $\F_k$ is the sigma-field generated by $(S_l,J_l)$ for $l\leq k$, the distribution of $(S_{k+1}-S_k,J_{k+1})$ only depends on $J_k$. Otherwise said, $(J_k)_{k\geq 0}$ is a Markov chain on $\{1,\ldots,\kappa\}$ (often called the \emph{driving chain}) and, if $J$ jumps from $i$ to $j$, then the corresponding jump of $S$ has a distribution $(q_m^{(i,j)},m\in \Z)$ independently of the past, where the $(q_m^{(i,j)},m\in\Z)$ are probability distributions on $\Z$. We focus exclusively on the case where $(S_k,k\geq 0)$ is \emph{nondecreasing} i.e. the jump distributions $(q_m^{(i,j)},m\in\Z)$ are supported on $\Z_+$. In this case, the process is also sometimes referred to as a \emph{Markov renewal process}. For background on these processes, we refer to the work of Alsmeyer \cite{Alsmeyer2} and the references therein.

We will consider a variant of the Markov random walk which has a barrier at an integer $n\in\N$. Informally, this is a version of $\big((S_k,J_k),k\geq 0)$ such that each jump of $(S,J)$ is conditioned on not taking the $S$ component higher than level $n$. To be specific, let  $n\in\N$, $P=(P(i,j))_{1\leq i,j\leq \kappa}$ be a stochastic matrix, and  for $i,j\in\{1,\ldots,\kappa\}$ and $k\in\Z_+$, set $$\overline{q_{k}}^{(i,j)}=\sum_{l\geq k+1} q^{(i,j)}_l \quad \text{and} \quad \overline{q_{k}}^{(i)}=\sum_{j'\in \{1,\ldots,\kappa\}}P(i,j')\overline{q_{k}}^{(i,j')}.$$  We define a Markov chain $\big((S_n(k),J_n(k)),k\geq 0\big)$ on $\{0,\ldots,n\}\times\{1,\ldots,\kappa\}$ with the explicit jump rates $q^{\{n\}}_{(s,i)}(t,j)$ given by, for $(i,j,s,t)\in\{1,\ldots,\kappa\}^2\times\{0,\ldots,n\}^2$, $s \leq t$:
\[q^{\{n\}}_{(s,i)}(t,j)=
\begin{cases}
         \frac{P(i,j)q^{(i,j)}_{t-s}}{1-\overline{q_{n-s}}^{(i)}}\qquad&\text{if }\overline{q_{n-s}}^{(i)}<1,\\
         \mathbbm{1}_{\{{t=s}\}}P(i,j)\qquad\:&\text{if }\overline{q_{n-s}}^{(i)}=1.
         \end{cases}
\]
Moreover, we always start with $S_0(n)=0$, while $J_0(n)$ is deterministic. Under this setting, it is clear that, letting $X_n(k)=n-S_n(k),$ the process \[\big((X_n(k),J_n(k)),k\geq 0\big)\]
is a Markov chain on $\{0,\ldots,n\}\times\{1,\ldots,\kappa\}$, and its transition probabilities are given by
\[p_{(s,i)}(t,j)=\begin{cases}
         \frac{P(i,j)q^{(i,j)}_{s-t}}{1-\overline{q_s}^{(i)}}\qquad&\text{if }\overline{q_{s}}^{(i)}<1,\\
         \mathbbm{1}_{\{{t=s}\}}P(i,j)\qquad\:&\text{if }\overline{q_{s}}^{(i)}=1.
         \end{cases}
\]
These do not depend on $n$, and as such fall under the framework of the paper. Hence we use again the notation $(X^{(i)}_n(k),J_n^{(i)}(k))$ to signify that the starting point is $(n,i)$.  

\bigskip

We can then give under a few conditions the scaling limit of $X_n$ and its absorption time $A_n$.
To this end, notice first that for $n$ large enough and fixed $i,j$,
\[\sum_{k=0}^n p_{(n,i)}(k,j)=\frac{1-\overline{q_n}^{(i,j)}}{1-\overline{q_n}^{(i)}}P(i,j) \underset{n \rightarrow \infty}\longrightarrow P(i,j).\]
In other words, when $n$ is large, the types behave as a random walk with transition matrix $P$, which means that we could only end up in the mixing regime, with $\beta=0$ and $Q=P-I$. 
\\

\begin{thm} We assume that the matrix $P$ is irreducible, and call $\pi$ its invariant measure.

\emph{(i)} Let $\gamma\in(0,1)$ and assume that, for all $i,j\in\{1,\ldots,\kappa\}$, there exists $a_i>0$ such that $n^{\gamma}\overline{q_n}^{(i)}$ converges to $a_i$. Then
\[
\left(\bigg(\frac{X_n^{(i)}(\lfloor n^{\gamma}t \rfloor)}{n}, t\geq 0\bigg), \frac{A_n^{(i)}}{n^{\gamma}}\right)  \overset{\mathrm{(d)}}{\underset{n \rightarrow \infty} \longrightarrow} \Big(\left(Z_{\rho(t)},t\geq 0\right),\int_0^{\infty}(Z_t)^{\gamma}\mathrm d t \Big),
\]
where $-\log Z$ is a subordinator with Laplace exponent $\psi$ defined for $\lambda\geq 0$ by
\[\psi(\lambda)=\sum_{i=1}^{\kappa} \pi_i a_i\int_0^{\infty}(1-\e^{-\gamma x})\frac{\gamma\e^{-x}\mathrm d x}{(1-\e^{-x})^{\gamma+1}}\]
and $\rho(t)=\inf \left\{u:\int_0^u Z(r)^{\gamma} \mathrm dr>t\right\}.$ We also have convergence of all positive moments for the second coordinate.

\emph{(ii)} Assume that, for all $i\in\{1,\ldots,\kappa\}$, $m_i:=\sum_{j=1}^{\kappa}\sum_{k=1}^{\infty} k P(i,j)q_k^{(i,j)}$ is finite. Then, letting $m=\sum_{i=1}^{\kappa} m_i$, we have
\[
\left(\bigg(\frac{X_n^{(i)}(\lfloor nt \rfloor)}{n}, t\geq 0\bigg), \frac{A_n^{(i)}}{n}\right ) \ \overset{\mathrm{(d)}}{\underset{n \rightarrow \infty} \longrightarrow} \Big(\big((1-mt)\vee 0,t\geq 0\big),\frac{1}{m} \Big).
\]
We also have convergence of all positive moments for the second coordinate.
\end{thm}

\begin{proof} Having in mind the remark above the theorem, we now just need to properly apply Theorem \ref{ThMixingjoint} to both cases.
For point (i), we have to prove that, for $i\in\{1,\ldots,\kappa\}$, and $f$ a continuous function on $[0,1]$,
\[\frac{1}{n^{\gamma}}\sum_{k=0}^n\sum_{j=1}^{\kappa}\frac{P(i,j)q^{(i,j)}_{n-k}}{1-\overline{q_n}^{(i)}}\Big(1-\frac{k}{n}\Big)f\Big(\frac{k}{n}\Big) \underset{n \rightarrow \infty}\longrightarrow  a_i\gamma\int_0^1 f(x)(1-x)^{-\gamma}\mathrm d x,\]
We can restrict ourselves to the case where $f$ is continuously differentiable, and we end up with the same computation as in the proof of Theorem 3 in \cite{HM11}, part (i), we do not repeat it here.
Similarly, for point (ii), noticing that $(1-mt\vee0,t\geq0)$ is the Lamperti transform of the subordinator $(mt,t\geq0)$, we have to prove that
\[n\sum_{k=0}^n\sum_{j=1}^{\kappa}\frac{P(i,j)q^{(i,j)}_{n-k}}{1-\overline{q_n}^{(i)}}\Big(1-\frac{k}{n}\Big)f\Big(\frac{k}{n}\Big) \underset{n \rightarrow \infty}\longrightarrow  m_if(1).\]
By Proposition 3 of \cite{HM11}, we can restrict ourselves to $f(x)=\frac{1-x^{\lambda}}{1-x}$ for $\lambda>0$ (extended by $f(1)=\lambda$), in which case the proof, once again, bears no difference to that of part (ii) of Theorem 3 of \cite{HM11}.
\end{proof}

\bigskip

\noindent \textbf{Remark.} One could imagine various other models of Markov random walks with a barrier. For example, instead of conditioning the walk on not taking the $S$ component higher than $n$, we could have killed the walk the first time that $S$ exceeds $n$. Or we could have imagine a model where the types still form a Markov chain with transition matrix $P$ and we only condition the position component to not jump over $n$. The results one gets for these model have mostly the same flavor, and thus we do not present them here.

%**********************************************************************

\section{Appendix} 
\label{Appendix}

%**********************************************************************

\subsection{A few results on the Skorokhod topology}\label{sec:appsko}

We start with the proof of Lemma \ref{lampsko} and then settle a few lemmas useful for the proof of Lemma \ref{LemCoupe} and Lemma \ref{lem:mixchangescale}

\bigskip

\noindent\textbf{Proof of Lemma \ref{lampsko}.} For (i), notice first that, by standard arguments, since the $\tau_n$ are all increasing, we only need to show pointwise convergence. For $t<T_0(g)$ this is simple, since, given $f(\tau(t))>0$, the equation $\int_0^{\tau_n(t)} f_n(r)^{\alpha} \mathrm{d}r = t$ shows that $\tau_n(t)$ can not have any subsequential limit which is larger or smaller than $\tau(t)$. For $t\geq T_0(g)$ (such that $f(\tau(t))=0$), we have by definition $\tau(t)=T_0(f),$ and must then show that $\tau_n(t)\to T_0(f).$ It is a direct consequence of the Skorokhod convergence of $f_n$ to $f$ that $\liminf \tau_n(t)\geq T_0(f).$ For the limsup, let $a>T_0(f)$, assume that a subsequence of $\tau_n(t)$ is greater than $a$. Along this subsequence, we then have $\tau_n(t)-a\leq f_n(a)^{-\alpha}(t-\tau_n^{(-1)}(a))$, which implies $t\geq f_n(a)^{\alpha}(\tau_n(t)-a)$. However, since $a>T_0(f)$, $f_n(a)$ tends to $0$ and thus this implies that $\tau_n(t)$ converges to $a$, a contradiction since we could replace $a$ by $(a+T_0(f))/2$.

Point (ii) is then an easy consequence of point (i). Recall that there exists for $n\in\N$ a time--change $\lambda_n$ which converges uniformly on compact sets to the identity function, such that $f_n\circ\lambda_n$ converges uniformly to $f$ on compact sets. Then, letting $\mu_n=\tau_n^{(-1)}\circ\lambda_n\circ\tau$, $\mu_n$ also converges uniformly on compact sets to the identity, and $g_n\circ\mu_n$ converges uniformly on compact sets to $g$.
\qed

\bigskip

\smallskip

Now, for all pairs of functions $f,g$ in $\mathcal D\left([0,\infty),[0,\infty)\right)$ and all $t>0$, we define a new function $\verb"glue"^{[t]}(f,g)\in \mathcal D\left([0,\infty),[0,\infty)\right)$ as follows:
\begin{equation}
\label{glue}
\verb"glue"^{[t]}(f,g)(s)=f(s), \quad \forall s<t, \quad \verb"glue"^{[t]}(f,g)(s)=g(s-t), \quad \forall s \geq t.
\end{equation}

\begin{lem}
\label{lemSko1}
Assume that:
\begin{enumerate}
\item[$\bullet$] $f_n \rightarrow f$ in $\mathcal D\left([0,\infty),[0,\infty)\right)$, with $f_n$ non--increasing
\item[$\bullet$] $g_n \rightarrow g$ in $\mathcal D\left([0,\infty),[0,\infty)\right)$
\item[$\bullet$] $t_n \in \mathbb R_+\rightarrow t \in \mathbb R_+$
\item[$\bullet$] $f_n(t_n-) \rightarrow f(t-)$
\end{enumerate}
Then,L
$$
\verb"glue"^{[t_n]}(f_n,g_n) \longrightarrow \verb"glue"^{[t]}(f,g) \quad \text{in } \mathcal D\left([0,\infty),[0,\infty)\right).
$$
\end{lem}

\begin{proof}
We use for this Proposition 6.5, chapter 3 of \cite{EK}. Since $f_n$ and $g_n$ converge in the Skorokhod sense, it is easy to see that conditions (a),(b),(c) of this proposition are satisfied for every time $s \neq t$. For $s=t$, let $s_n \rightarrow t$. If $s_n \geq t_n$ for all $n$ large enough, then $h_n(s_n)=g_n(s_n-t_n) \rightarrow g(0)=h(t)$. If $s_n < t_n$ for all $n$ large enough, then $h_n(s_n)=f_n(s_n)$. Let $\varepsilon>0$ such that 
$
f_n(t-\varepsilon) \rightarrow f(t-\varepsilon)
$
(recall that this holds for every $\varepsilon>0$ such that $t-\varepsilon$ is not a jump time of $f$).
Since $f_n$ is non--increasing, we have $f_n(t_n-)\leq f_n(s_n) \leq f_n(t-\varepsilon)$ for $n$ large enough, hence
$$
f(t-) \leq \liminf_n f_n(s_n) \leq \limsup_n f_n(s_n)\leq f(t-\varepsilon).
$$
We conclude, by letting $\varepsilon \rightarrow 0$ along an appropriate subsequence, that $f_n(s_n) \rightarrow f(t-)=h(t-)$. Hence assertions (a),(b) and (c) of Proposition 6.5 are satisfied for $h_n,h$ and the result follows.
\end{proof}

\bigskip

\begin{lem}
\label{lemSko2}
Let $f_n,f$ be non--increasing non--negative c\`adl\`ag functions such that
$$
f_n \underset{n\rightarrow \infty}\longrightarrow f \text{ on } \mathcal D([0,\infty),[0,\infty)).
$$
Assume that $t_n=\inf \{s:f_n(s)=0\} \rightarrow t=\inf\{s:f(s)=0\}<\infty$. 
\begin{enumerate}
\item[$\mathrm{(i)}$] If moreover $f(t-)>0$ and $\liminf f_n(t_n-) >0$, then $f_n(t_n-) \rightarrow f(t-)$.
\item[$\mathrm{(ii)}$] If $f(t-)=0$, then $f_n(t_n-) \rightarrow 0=f(t-)$.
\end{enumerate}
\end{lem}

However it is easy to build examples where $f_n(t_n-)\rightarrow 0$ whereas $f(t-)>0$ (e.g. $f(s)=\mathbbm 1_{\{s<1\}}$, $f_n(s)=\mathbbm 1_{\{s<1\}}+n^{-1}\mathbbm 1_{\{1\leq s<1+n^{-1}\}}$).

\begin{proof}
Case (i). By definition of the Skorokhod topology, we know that there is a sequence of times $s_n \rightarrow t$ such that $f_n(s_n-) \rightarrow f(t-)$ and $f_n(s_n)\rightarrow f(t)=0$. Note that since the functions $f_n$ are non--increasing and since  $\liminf f_n(t_n-) >0$, we necessarily have that $s_n \geq t_n$ for all $n$ large enough. On the other hand, if $s_n>t_n$, then $f_n(s_n-)=0$, and this is not possible for $n$ large enough since $f_n(s_n-) \rightarrow f(t-)>0$. So finally $s_n=t_n$ for all $n$ large enough and $f_n(t_n-) \rightarrow f(t-)$.

Case (ii). For all $\delta>0$, let $\varepsilon>0$ such that $f(t-\varepsilon) \leq \delta$ and $f_n(t-\varepsilon)\rightarrow f(t-\varepsilon)$ (such an $\varepsilon$ exists since $f_n(s)\rightarrow f(s)$ for a.e. $s$). Since the $f_n$ are non--increasing, this leads to $\limsup_n f_n(t_n-)\leq \delta$ for all $\delta>0$.
\end{proof}

\begin{lem}\label{lem:skotoL}
Suppose that $f_n$ and $f$ are c\`adl\`ag functions on $[0,1]$ such that $f_n$ converges to $f$ in the Skorokhod topology. Then $f_n$ also converges to $f$ in $L_1([0,1])$.
\end{lem}
\begin{proof}
Let $\veps>0$. We know that, for $n$ large enough, there exists a continuous and increasing time--change $\tau_n$ such that $|\tau_n(x)-x|\leq \veps$ and $|f_n(x)-f(\tau_n(x))|\leq \veps$ for all $x\in[0,1].$ Letting $\overline{f}(x,\veps)=\underset{|y-x|\leq \veps} \sup f(y)$ and $\underline{f}(x,\veps)=\underset{|y-x|\leq \veps} \inf f(y),$
we then have
\[\underline{f}(\cdot,\veps)-\veps \leq f_n\leq \overline{f}(\cdot,\veps)+\varepsilon.\]
The proof is ended by noting that both $\underline{f}(\cdot,\veps)$ and $\overline{f}(\cdot,\veps)$ converge in $L_1$ to $f$ by the monotone convergence theorem, since $f$ is bounded and has countably many discontinuities.
\end{proof}
\subsection{Weak convergence in probability of measures}\label{sec:cvprobamesures}

The notion of weak convergence of finite measures on $[0,\infty)$ can be metrized by the Prokhorov metric, defined by 
\[d (\mu, \nu) := \inf \left\{ \varepsilon > 0 ~|~ \mu(A) \leq \nu (A^{\varepsilon}) + \varepsilon \ \text{and} \ \nu (A) \leq \mu (A^{\varepsilon}) + \varepsilon \ \text{for all} \ A \in \mathcal{B}(M) \right\},\]
where $\mu$ and $\nu$ are two finite measures on $[0,\infty)$ and $A^{\veps}$ denotes the $\veps$--enlargement of $A$.

With this metric then comes a notion of convergence in probability for random measures. We list here a few elementary properties which are of use. In all three upcoming lemmas, $(\mu_n,n\in\N)$ and $\mu$ are some random finite measures on  $[0,\infty).$

\begin{lem}\label{lem:extraction} As $n$ tends to infinity, $\mu_n$ converges in probability to $\mu$ if and only if, for any subsequence of $(\mu_n,n\in\N)$, one can extract another subsequence which converges a.s. to $\mu$.
\end{lem}

This is classical, and in fact true for random variables in any metric space, not just random measures. 
The next lemma is just a consequence of the fact that continuous maps preserve convergence in distribution.
\begin{lem}\label{lem:cvprobaintegration} Assume that $\mu_n$ converges in probability to $\mu$ and let $f$ be a continuous and bounded function on $[0,\infty)$. Then $\int_{[0,\infty)} f(x) \mathrm d \mu_n(x)$ converges in probability to $\int_{0,\infty} f(x) \mathrm d \mu(x)$
\end{lem}

We end with a partial variant of the Portmanteau theorem in probability.
\begin{lem}\label{lem:interversion}
Assume that, for all $t\geq 0$, $\mu_n([0,t])$ converges in probability to $\mu([0,t])$ as $n$ tends to infinity. Then $\mu_n$ converges in probability to $\mu$.
\end{lem}

\begin{proof}
We use Lemma \ref{lem:extraction}: let $(\nu_n,n\in\N)$ be an extracted subsequence, and we will extract from it a subsequence which converges a.s. to $\mu$. Let $(t_k,k\in\N)$ be an enumeration of the nonnegative rational numbers. We then let $\sigma_1$ be an extraction such that $\nu_{\sigma_1(n)}([0,t_1])$ converges a.s. to $\mu([0,t_1])$, then extract $\sigma_2$ from $\sigma_1$ such that $\nu_{\sigma_2(n)}([0,t_2])$ converges a.s. to $\mu([0,t_2]),$ and so on: for all $k$, $\sigma_{k}$ is an extraction such that $\nu_{\sigma_k(n)}([0,t_i])$ converges a.s. to $\mu([0,t_i])$ for all $i\leq k$. We then do a diagonal extraction and let $\sigma(n)=\sigma_n(n)$, and we then get that, for all rational $t$, $\nu_{\sigma(n)}([0,t])$ converges a.s. to $\mu([0,t])$. Now for irrational $t$, we get by monotonicity arguments
\[\mu([0,t))\leq\liminf \nu_{\sigma(n)}([0,t]) \leq \limsup \nu_{\sigma(n)}([0,t])\leq \mu([0,t]). \]
Thus, if $t$ is a continuity point of $\mu([0,\cdot])$, $\nu_{\sigma(n)}([0,t])$ converges to $\mu([0,\cdot])$, and by the Portmanteau theorem, $ \nu_{\sigma(n)}$ converges a.s. to $\mu$. This ends the proof.
\end{proof}

\subsection{Wald's formula}

We use the following variant of Wald's formula:
\begin{lem}\label{wald}
 Let $X_n,n\geq 1$ be real--valued random variables, $N$ a random integer, and assume that there exists $a\geq 0$ such that, for all $n$, $\mathbb{E}[X_n \mid N\geq n]\geq a$, then
\[\mathbb{E}\left[\sum_{i=1}^N X_i\right]\geq a\E[N].\]
This stays true if we swap $\leq$ for $\geq$.
\end{lem}
\begin{proof} Just notice that
$\mathbb{E}\big[\sum_{i=1}^N X_i\big]=\sum_{i=1}^{\infty} \mathbb{E}\left[X_i\mathbbm{1}_{\{N\geq i\}}\right] \geq\sum_{i=1}^{\infty} a\pr[N\geq i] \geq a\E[N].$
\end{proof}
%**********************************************************************
\bibliographystyle{siam}
\addcontentsline{toc}{section}{References}
\bibliography{frag}

\begin{thebibliography}{10}

\bibitem{Ald96}
{\sc D.~Aldous}, {\em Probability distributions on cladograms}, in Random
  discrete structures ({M}inneapolis, {MN}, 1993), vol.~76 of IMA Vol. Math.
  Appl., Springer, New York, 1996, pp.~1--18.

\bibitem{Alsmeyer2}
{\sc G.~Alsmeyer}, {\em On the {M}arkov renewal theorem}, Stochastic Process.
  Appl., 50 (1994), pp.~37--56.

\bibitem{AM15}
{\sc G.~Alsmeyer and A.~Marynych}, {\em Renewal approximation for the
  absorption time of a decreasing {M}arkov chain},  (2015).
\newblock Preprint -- arXiv:1509.01704.

\bibitem{asmussen}
{\sc S.~Asmussen}, {\em Applied Probability and Queues}, Applications of
  mathematics : stochastic modelling and applied probability, Springer, 2003.

\bibitem{Berest09}
{\sc N.~Berestycki}, {\em Recent progress in coalescent theory}, vol.~16, 2009.

\bibitem{Ber08}
{\sc J.~Bertoin}, {\em Homogeneous multitype fragmentations}, in In and out of
  equilibrium. 2, vol.~60 of Progr. Probab., Birkh\"auser, Basel, 2008,
  pp.~161--183.

\bibitem{BerFire}
\leavevmode\vrule height 2pt depth -1.6pt width 23pt, {\em Fires on trees},
  Ann. Inst. Henri Poincar\'e Probab. Stat., 48 (2012), pp.~909--921.

\bibitem{BCK15}
{\sc J.~Bertoin, N.~Curien, and I.~Kortchemski}, {\em Random planar maps $\&$
  growth-fragmentations},  (2015).
\newblock Preprint -- arXiv:1507.02265.

\bibitem{BK14}
{\sc J.~Bertoin and I.~Kortchemski}, {\em Self-similar scaling limits of
  {M}arkov chains on the positive integers}, Ann. Appl. Probab., 26 (2016),
  pp.~2556--2595.

\bibitem{Bill99}
{\sc P.~Billingsley}, {\em Convergence of Probability Measures}, Wiley, New
  York, 1968.

\bibitem{BDMcS08}
{\sc N.~Broutin, L.~Devroye, E.~McLeish, and M.~de~la Salle}, {\em The height
  of increasing trees}, Random Structures Algorithms, 32 (2008), pp.~494--518.

\bibitem{CPR13}
{\sc L.~Chaumont, H.~Pant{\'{\i}}, and V.~Rivero}, {\em The {L}amperti
  representation of real-valued self-similar {M}arkov processes}, Bernoulli, 19
  (2013), pp.~2494--2523.

\bibitem{EK}
{\sc S.~N. Ethier and T.~G. Kurtz}, {\em Markov processes}, Wiley Series in
  Probability and Mathematical Statistics: Probability and Mathematical
  Statistics, John Wiley \& Sons Inc., New York, 1986.

\bibitem{GIM11}
{\sc A.~Gnedin, A.~Iksanov, and A.~Marynych}, {\em On {$\Lambda$}-coalescents
  with dust component}, J. Appl. Probab., 48 (2011), pp.~1133--1151.

\bibitem{GIM08}
{\sc A.~Gnedin, A.~Iksanov, and M.~M{\"o}hle}, {\em On asymptotics of
  exchangeable coalescents with multiple collisions}, J. Appl. Probab., 45
  (2008), pp.~1186--1195.

\bibitem{GY07}
{\sc A.~Gnedin and Y.~Yakubovich}, {\em On the number of collisions in
  {$\Lambda$}-coalescents}, Electron. J. Probab., 12 (2007), pp.~no. 56,
  1547--1567.

\bibitem{HSurvey16}
{\sc B.~Haas}, {\em Scaling limits of {M}arkov branching trees and
  applications},  (2016).
\newblock Preprint -- arXiv:1605.07873.

\bibitem{HM11}
{\sc B.~Haas and G.~Miermont}, {\em Self-similar scaling limits of
  non-increasing {M}arkov chains}, Bernoulli Journal, 17 (2011),
  pp.~1217--1247.

\bibitem{HM12}
\leavevmode\vrule height 2pt depth -1.6pt width 23pt, {\em Scaling limits of
  {M}arkov branching trees with applications to {G}alton-{W}atson and random
  unordered trees}, Ann. Probab., 40 (2012), pp.~2589--2666.

\bibitem{HMPW08}
{\sc B.~Haas, G.~Miermont, J.~Pitman, and M.~Winkel}, {\em Continuum tree
  asymptotics of discrete fragmentations and applications to phylogenetic
  models}, Ann. Probab., 36 (2008), pp.~1790--1837.

\bibitem{HS15}
{\sc B.~Haas and R.~Stephenson}, {\em Scaling limits of multi-type {M}arkov
  branching trees and applications}.
\newblock In preparation.

\bibitem{HS14}
{\sc B.~Haas and R.~Stephenson}, {\em Scaling limits of {$k$}-ary growing
  trees}, Ann. Inst. Henri Poincar\'e Probab. Stat., 51 (2015), pp.~1314--1341.

\bibitem{IMM09}
{\sc A.~Iksanov, A.~Marynych, and M.~M{\"o}hle}, {\em On the number of
  collisions in beta(2,b)-coalescents}, Bernoulli, 15 (2009), pp.~829--845.

\bibitem{IM08}
{\sc A.~Iksanov and M.~M{\"o}hle}, {\em On the number of jumps of random walks
  with a barrier}, Adv. in Appl. Probab., 40 (2008), pp.~206--228.

\bibitem{KKPW14}
{\sc A.~Kuznetsov, A.~E. Kyprianou, J.~C. Pardo, and A.~R. Watson}, {\em The
  hitting time of zero for a stable process}, Electron. J. Probab., 19 (2014),
  pp.~no. 30, 26.

\bibitem{Lamp62}
{\sc J.~Lamperti}, {\em Semi-stable stochastic processes}, Trans. Amer. Math.
  Soc., 104 (1962), pp.~62--78.

\bibitem{Lamp72}
\leavevmode\vrule height 2pt depth -1.6pt width 23pt, {\em Semi-stable {M}arkov
  processes. {I}}, Z. Wahrscheinlichkeitstheorie und Verw. Gebiete, 22 (1972),
  pp.~205--225.

\bibitem{Mohle02}
{\sc M.~M{\"o}hle}, {\em The coalescent in population models with
  time-inhomogeneous environment}, Stochastic Process. Appl., 97 (2002),
  pp.~199--227.

\bibitem{PitCoag99}
{\sc J.~Pitman}, {\em Coalescents with multiple collisions}, Ann. Probab., 27
  (1999), pp.~1870--1902.

\bibitem{Sag99}
{\sc S.~Sagitov}, {\em The general coalescent with asynchronous mergers of
  ancestral lines}, J. Appl. Probab., 36 (1999), pp.~1116--1125.

\end{thebibliography}

\end{document}